\UseRawInputEncoding 

\documentclass[10pt]{article}
\usepackage[letterpaper, left=1in, top=1in, right=1in, bottom=1in, verbose, ignoremp]{geometry}

\usepackage{url}\RequirePackage[colorlinks,citecolor=blue, linkcolor=blue,urlcolor = blue]{hyperref}
\usepackage{latexsym,amssymb,amsmath,amsfonts,graphicx,color,fancyvrb,amsthm,enumerate,subcaption,mathrsfs}
\usepackage[longnamesfirst,authoryear,round]{natbib}
\usepackage[dvipsnames]{xcolor}
\usepackage{xy}\xyoption{all} \xyoption{poly} \xyoption{knot}
\usepackage{float}
\usepackage{bm}
\usepackage{bbm}
\usepackage{multicol,multirow}
\usepackage{array}
\usepackage{relsize}
\usepackage{chngcntr}
\usepackage{etoolbox}
\usepackage[labelformat=simple]{subcaption}

\usepackage{caption}
\usepackage{pgfplots}
\pgfplotsset{compat=1.10}
\usepgfplotslibrary{fillbetween}
\usepackage{tikz}
\usetikzlibrary{decorations.pathreplacing,calc}
\usetikzlibrary{shapes,backgrounds}
\usetikzlibrary{patterns}
\usetikzlibrary{spy}
\usetikzlibrary{cd}
\usepackage{amsopn}
\usepackage{calc}
\usepackage{enumitem}
\usepackage{algorithm}
\usepackage[noend]{algpseudocode}
\usepackage{hyperref}
\usepackage{mathtools}

\usepackage{tikz-cd} 
\usepackage{amscd} 
\usepackage{comment}
\usepackage[capitalize,nameinlink,compress]{cleveref}
\usepackage{crossreftools}
\crefname{figure}{Figure}{Figures} 
\crefname{equation}{}{} 
\crefname{assumption}{Assumption}{Assumptions}
\crefname{subsection}{Subsection}{Subsections}
\usepackage[export]{adjustbox}

\newcounter{cdrow}

\newtheorem{theorem}{Theorem}[]
\newtheorem*{theorem*}{Theorem}
\newtheorem{corollary}[theorem]{Corollary}
\newtheorem{lemma}[theorem]{Lemma}

\newtheorem*{claim*}{Claim}

\theoremstyle{definition}
\newtheorem{definition}[theorem]{Definition}
\newtheorem*{definition*}{Definition}

\theoremstyle{remark}
\newtheorem{remark}[theorem]{Remark}

\newtheorem*{example*}{Example}

\def\log{{\rm log}}

\makeatletter
\newcommand*{\op}{%
  \DOTSB
  \mathop{\vphantom{\bigoplus}\mathpalette\matt@op\relax}%
  \slimits@
}
\newcommand\matt@op[2]{%
  \vcenter{\m@th\hbox{\resizebox{\widthof{$#1\bigoplus$}}{!}{$\boxplus$}}}%
}
\makeatother

\newcommand{\R}{\mathbb{R}}

\makeatletter
\def\@biblabel#1{}
\makeatother

\makeatletter
\patchcmd{\NAT@citex}
  {\@citea\NAT@hyper@{%
     \NAT@nmfmt{\NAT@nm}%
     \hyper@natlinkbreak{\NAT@aysep\NAT@spacechar}{\@citeb\@extra@b@citeb}%
     \NAT@date}}
  {\@citea\NAT@nmfmt{\NAT@nm}%
   \NAT@aysep\NAT@spacechar\NAT@hyper@{\NAT@date}}{}{}

\patchcmd{\NAT@citex}
  {\@citea\NAT@hyper@{%
     \NAT@nmfmt{\NAT@nm}%
     \hyper@natlinkbreak{\NAT@spacechar\NAT@@open\if*#1*\else#1\NAT@spacechar\fi}%
       {\@citeb\@extra@b@citeb}%
     \NAT@date}}
  {\@citea\NAT@nmfmt{\NAT@nm}%
   \NAT@spacechar\NAT@@open\if*#1*\else#1\NAT@spacechar\fi\NAT@hyper@{\NAT@date}}
  {}{}

\makeatother

\begin{document}
\def\spacingset#1{\renewcommand{\baselinestretch}%
{#1}\small\normalsize} \spacingset{1}

\begin{flushleft}
{\Large{\textbf{Wavelet-Based Density Estimation for Persistent Homology}}}
\newline
\\
Konstantin H\"{a}berle$^{1,2,\dagger}$, Barbara Bravi$^{1,\dagger}$, and Anthea Monod$^{1,\dagger}$
\\
\bigskip
\bf{1} Department of Mathematics, Imperial College London, UK
\\
\bf{2} Chair for Mathematical Information Science, ETH Zurich, Switzerland
\\
\bigskip
$\dagger$ Corresponding e-mails: haeberlk@ethz.ch; b.bravi21@imperial.ac.uk; a.monod@imperial.ac.uk
\end{flushleft}


\section*{Abstract}
Persistent homology is a central methodology in topological data analysis that has been successfully implemented in many fields and is becoming increasingly popular and relevant.  The output of persistent homology is a persistence diagram---a multiset of points supported on the upper half plane---that is often used as a statistical summary of the topological features of data.  In this paper, we study the random nature of persistent homology and estimate the density of expected persistence diagrams from observations using wavelets; we show that our wavelet-based estimator is optimal.  Furthermore, we propose an estimator that offers a sparse representation of the expected persistence diagram that achieves near-optimality. We demonstrate the utility of our contributions in a machine learning task in the context of dynamical systems.


\paragraph{Keywords:} Nonparametric density estimation; wavelets; persistent homology; persistence measures.

\section{Introduction}
\label{sec:intro}
Topological data analysis (TDA) is a recent field of data science emanated from applied and computational topology that extracts and studies topological information from complex and high-dimensional data structures using theory from algebraic topology. A fundamental method in TDA is \emph{persistent homology}, which adapts the classical algebraic topological concept of \emph{homology} to study topological invariants associated with a dataset; while the homology of a dataset corresponds to its ``shape,'' persistent homology captures both the ``shape'' and ``size'' of the dataset. Persistent homology has been successfully employed in a variety of applications including signal analysis \citep{TDA_signal_analysis}, computer vision \citep{TDA_CompVision}, cancer patient survival \citep{TDA_progression_disease}, and viral evolution \citep{TDA_viral_evolution}, and is becoming an increasingly popular technique in data analysis.

\emph{Persistence diagrams} are the resulting objects of persistent homology; they encode the lifetimes of topological features of datasets and thus exhibit a random nature. The randomness of persistence diagrams is a challenging topic of study because the algebraic construction of persistent homology induces a highly complex geometric space \citep{turner-2014-frechet}. In this paper, we focus on persistence diagrams in random settings and study their distributional behavior nonparametrically using methods from uncertainty quantification. Specifically, we propose a Haar wavelet estimator to estimate the \emph{expected persistence diagram}, which, under mild assumptions, is known to be absolutely continuous with respect to the Lebesgue measure and therefore has a Lebesgue density \citep{Chazal_Divol_EPD}.  We prove that the Haar wavelet estimator is \emph{minimax}: its maximal risk is minimal among all possible estimators of the expected persistence diagram. Furthermore, we also propose a thresholding Haar wavelet estimator, which offers a sparse (or \emph{compressed}) representation of the expected persistence diagram and thus computational advantages in practical settings, and show that it achieves near-optimal minimax rates.  This paper is, to the best of our knowledge, the first to use wavelet estimators to estimate expected persistence diagrams.

Our work is largely inspired by previous work by \cite{Chazal_Divol_EPD}, who propose a kernel density estimator (KDE) and prove its convergence to the true density function in the $L^2$-norm.  However, it has since been shown by \cite{divol2021estimation} that the \emph{optimal partial transport} metric $\mathrm{OT}_p$, rather than the $L^2$ metric, is the natural metric for expected persistence diagrams; this metric was used to show that the empirical mean of persistence diagrams is minimax. The resulting proposed estimator, however, does not have a Lebesgue density in general and its support moreover tends to be very large, making it impractical in applications.  In addition to our proposed Haar wavelet comprising a Lebesgue density and our thresholding Haar wavelet being sparse, compared to KDEs, Haar wavelets have the advantage of being locally adaptive.  The local adaptivity property enables a finer local analysis, which may be more difficult to achieve with kernels depending on choice of bandwidth.

The remainder of this paper is organized as follows. \Cref{sec:preliminaries} provides the formal background to persistent homology, persistence measures, and their resulting metric space, which is the context of our work. \Cref{sec:minimax} presents our main result of minimaxity of wavelet estimators for expected persistence diagrams as well as the thresholding wavelet estimators and discusses their properties. \Cref{sec:simulations} presents numerical experiments and verifications of our derived theory. Furthermore, we provide a practical implementation of our work to a classification problem in the context of dynamical systems. We conclude our paper with a discussion of our work and some ideas for future research in \cref{sec:end}.

\section{Preliminaries}
\label{sec:preliminaries}
In this section, we overview the construction of persistence diagrams from persistent homology and discuss their properties. We then present the generalization of persistence diagrams to persistence measures as the setting in which we work. In particular, we present the expected persistence diagram as our object of study which we aim to estimate. 

In addition, we recall basic concepts and results from wavelet theory that are essential for constructing nonparametric density estimators. 

\subsection{Persistent Homology}
Persistent homology studies the topological features of a simplicial complex or topological space across multiple scales. In particular, it tracks the evolution of connected components, loops, and higher-dimensional cavities with respect to a nested sequence of topological spaces, i.e., a \emph{filtration}. For computational feasibility due to the existence of efficient algorithms, discretizations of these topological spaces are often considered in the form of simplices and simplicial complexes that can be seen as skeletal representations of the topological space.  A $k$-simplex is the convex hull of $k+1$ affinely independent points $x_0, x_1, \ldots, x_k$; a set of $k$-simplices assembled in a combinatorial fashion forms a simplicial complex $K$. See \cref{fig:simplicial_complex} for an illustration.

\begin{figure}[t!]
	\centering
	\begin{subfigure}{0.49\textwidth}
		\includegraphics[scale=0.9]{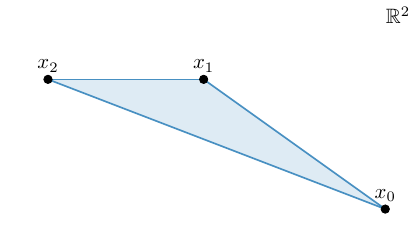}
		\caption{}
		\label{subfig:simplex}
	\end{subfigure}
	\begin{subfigure}{0.49\textwidth}
		\includegraphics[scale=0.9]{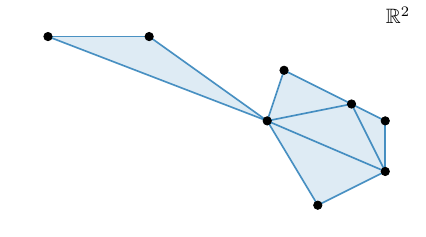}
		\caption{}
		\label{subfig:simplicialcomplex}
	\end{subfigure}
	\caption{\subref{subfig:simplex} $2$-simplex formed by $3$ affinely independent points $x_0,x_1,x_2 \in \mathbb R^2$. \subref{subfig:simplicialcomplex} Simplicial complex consisting of $2$-simplices.}
	\label{fig:simplicial_complex}
\end{figure}

Given a point cloud $X$, an important and widely-used simplicial complex is the \emph{Vietoris--Rips (VR) complex}, used to build the Vietoris--Rips (or simply Rips) filtration $(K_t)_{t \geq 0} \coloneqq (\mathrm{VR}(X,t))_{t\geq 0}$. The VR complex for $X$ and filtration parameter $t$, denoted $\mathrm{VR}(X,t)$, is the simplicial complex whose vertex set is $X$, where $\{x_0,\dots,x_k\}\subseteq X$ spans a $k$-simplex if and only if the diameter between two points, $x_i$ and $x_j$, is less than the threshold value, $t$: $d^X(x_i,x_j)\leq t$, for all $i,j\in\{0,1,\dots,k\}$, where $d^X$ denotes a metric on $X$.  In this paper, we focus on point clouds as our topological space of interest and their discretization by VR complexes.

Computing the $k$th simplicial homology groups $H_k(K_t)$ with coefficients in a field for dynamic values of $t\geq 0$ yields a sequence of vector spaces. Note that $K_t \subseteq K_{t'}$ whenever $t\leq t'$, which induces linear maps between $H_k(K_t)$ and $H_k(K_{t'})$. The family of vector spaces together with their linear maps is called a \emph{persistence module}.  Persistence modules are uniquely decomposable into a direct sum of interval modules up to permutations \citep{zomorodian_computing_2005}. The collection of these indecomposables is referred to as a \emph{barcode} where the intervals are \emph{bars}, each representing the evolution of a topological feature ($k$-dimensional hole). In particular, if $(t_1,t_2)\subset \mathbb R$ is such an interval, then the corresponding topological feature appears at scale $t_1$ in the filtration $(K_t)_{t\geq 0}$ and disappears at scale $t_2$. We call $t_1$ and $t_2$ the \emph{birth} and \emph{death times}, respectively. This information can be encoded in a \emph{persistence diagram} by taking birth and death times as ordered pairs and plotting them. The persistence diagram is then the output of persistent homology, see \cref{fig:PH_procedure,fig:PD}.

\begin{figure}[t!]
	\centering
	\begin{subfigure}[t]{0.3\textwidth}
		\centering
		\begin{tikzpicture}[scale=0.63, every node/.style={scale=1}]
\definecolor{crimson2143940}{RGB}{214,39,40}
\definecolor{darkgray176}{RGB}{176,176,176}
\definecolor{darkorange25512714}{RGB}{255,127,14}
\begin{axis}[
unit vector ratio*=1 1 1,
tick align=outside,
tick pos=left,
x grid style={darkgray176},
xmin=-.3, xmax=4.5,
xtick style={color=black},
y grid style={darkgray176},
ymin=-0.5, ymax=3.2,
ytick style={color=black}
]
\addplot [draw=crimson2143940, fill=crimson2143940, mark=*, only marks]
table{%
x  y
1.5 0.5
0.8 0.8
0.75 1.25
0.7 1.7
1.2 2.2
1.65 2.25
2.2 2.1
2.5 1.5
2.2 0.8
2.5 1.2
2.55 0.45
3.2 0.55
3.5 1
3.1 1.5
};
\draw[draw=black,fill=darkorange25512714,opacity=0.1] (axis cs:1.5,0.5) circle (0);
\draw[draw=black,fill=darkorange25512714,opacity=0.1] (axis cs:0.8,0.8) circle (0);
\draw[draw=black,fill=darkorange25512714,opacity=0.1] (axis cs:0.75,1.25) circle (0);
\draw[draw=black,fill=darkorange25512714,opacity=0.1] (axis cs:0.7,1.7) circle (0);
\draw[draw=black,fill=darkorange25512714,opacity=0.1] (axis cs:1.2,2.2) circle (0);
\draw[draw=black,fill=darkorange25512714,opacity=0.1] (axis cs:1.65,2.25) circle (0);
\draw[draw=black,fill=darkorange25512714,opacity=0.1] (axis cs:2.2,2.1) circle (0);
\draw[draw=black,fill=darkorange25512714,opacity=0.1] (axis cs:2.5,1.5) circle (0);
\draw[draw=black,fill=darkorange25512714,opacity=0.1] (axis cs:2.2,0.8) circle (0);
\draw[draw=black,fill=darkorange25512714,opacity=0.1] (axis cs:2.5,1.2) circle (0);
\draw[draw=black,fill=darkorange25512714,opacity=0.1] (axis cs:2.55,0.45) circle (0);
\draw[draw=black,fill=darkorange25512714,opacity=0.1] (axis cs:3.2,0.55) circle (0);
\draw[draw=black,fill=darkorange25512714,opacity=0.1] (axis cs:3.5,1) circle (0);
\draw[draw=black,fill=darkorange25512714,opacity=0.1] (axis cs:3.1,1.5) circle (0);
\end{axis}
\end{tikzpicture}
		\caption{$t=0$}
		\label{fig:t=0}
	\end{subfigure}
	\hfill
	\begin{subfigure}[t]{0.3\textwidth}
		\centering
		\begin{tikzpicture}[scale=0.63, every node/.style={scale=1}]
\definecolor{crimson2143940}{RGB}{214,39,40}
\definecolor{darkgray176}{RGB}{176,176,176}
\definecolor{darkorange25512714}{RGB}{255,127,14}
\definecolor{steelblue31119180}{RGB}{31,119,180}
\begin{axis}[
unit vector ratio*=1 1 1,
tick align=outside,
tick pos=left,
x grid style={darkgray176},
xmin=-.3, xmax=4.5,
xtick style={color=black},
y grid style={darkgray176},
ymin=-0.5, ymax=3.2,
ytick style={color=black}
]
\addplot [draw=crimson2143940, fill=crimson2143940, mark=*, only marks]
table{%
x  y
1.5 0.5
0.8 0.8
0.75 1.25
0.7 1.7
1.2 2.2
1.65 2.25
2.2 2.1
2.5 1.5
2.2 0.8
2.5 1.2
2.55 0.45
3.2 0.55
3.5 1
3.1 1.5
};
\draw[draw=black,fill=darkorange25512714,opacity=0.1] (axis cs:1.5,0.5) circle (0.25);
\draw[draw=black,fill=darkorange25512714,opacity=0.1] (axis cs:0.8,0.8) circle (0.25);
\draw[draw=black,fill=darkorange25512714,opacity=0.1] (axis cs:0.75,1.25) circle (0.25);
\draw[draw=black,fill=darkorange25512714,opacity=0.1] (axis cs:0.7,1.7) circle (0.25);
\draw[draw=black,fill=darkorange25512714,opacity=0.1] (axis cs:1.2,2.2) circle (0.25);
\draw[draw=black,fill=darkorange25512714,opacity=0.1] (axis cs:1.65,2.25) circle (0.25);
\draw[draw=black,fill=darkorange25512714,opacity=0.1] (axis cs:2.2,2.1) circle (0.25);
\draw[draw=black,fill=darkorange25512714,opacity=0.1] (axis cs:2.5,1.5) circle (0.25);
\draw[draw=black,fill=darkorange25512714,opacity=0.1] (axis cs:2.2,0.8) circle (0.25);
\draw[draw=black,fill=darkorange25512714,opacity=0.1] (axis cs:2.5,1.2) circle (0.25);
\draw[draw=black,fill=darkorange25512714,opacity=0.1] (axis cs:2.55,0.45) circle (0.25);
\draw[draw=black,fill=darkorange25512714,opacity=0.1] (axis cs:3.2,0.55) circle (0.25);
\draw[draw=black,fill=darkorange25512714,opacity=0.1] (axis cs:3.5,1) circle (0.25);
\draw[draw=black,fill=darkorange25512714,opacity=0.1] (axis cs:3.1,1.5) circle (0.25);
\addplot [semithick, steelblue31119180, opacity=0.8]
table {%
0.8 0.8
0.75 1.25
};
\addplot [semithick, steelblue31119180, opacity=0.8]
table {%
0.75 1.25
0.7 1.7
};
\addplot [semithick, steelblue31119180, opacity=0.8]
table {%
1.2 2.2
1.65 2.25
};
\addplot [semithick, steelblue31119180, opacity=0.8]
table {%
2.5 1.5
2.5 1.2
};
\addplot [semithick, steelblue31119180, opacity=0.8]
table {%
2.2 0.8
2.5 1.2
};
\addplot [semithick, steelblue31119180, opacity=0.8]
table {%
2.2 0.8
2.55 0.45
};
\end{axis}
\end{tikzpicture}
		\caption{$t=0.5$}
		\label{fig:t=0.25}
	\end{subfigure}
	\hfill
	\begin{subfigure}[t]{0.3\textwidth}
		\centering
\begin{tikzpicture}[scale=0.63, every node/.style={scale=1}]
\definecolor{crimson2143940}{RGB}{214,39,40}
\definecolor{darkgray176}{RGB}{176,176,176}
\definecolor{darkorange25512714}{RGB}{255,127,14}
\definecolor{steelblue31119180}{RGB}{31,119,180}
\begin{axis}[
unit vector ratio*=1 1 1,
tick align=outside,
tick pos=left,
x grid style={darkgray176},
xmin=-.3, xmax=4.5,
xtick style={color=black},
y grid style={darkgray176},
ymin=-0.5, ymax=3.2,
ytick style={color=black}
]
\addplot [draw=crimson2143940, fill=crimson2143940, mark=*, only marks]
table{%
x  y
1.5 0.5
0.8 0.8
0.75 1.25
0.7 1.7
1.2 2.2
1.65 2.25
2.2 2.1
2.5 1.5
2.2 0.8
2.5 1.2
2.55 0.45
3.2 0.55
3.5 1
3.1 1.5
};
\draw[draw=black,fill=darkorange25512714,opacity=0.1] (axis cs:1.5,0.5) circle (0.33);
\draw[draw=black,fill=darkorange25512714,opacity=0.1] (axis cs:0.8,0.8) circle (0.33);
\draw[draw=black,fill=darkorange25512714,opacity=0.1] (axis cs:0.75,1.25) circle (0.33);
\draw[draw=black,fill=darkorange25512714,opacity=0.1] (axis cs:0.7,1.7) circle (0.33);
\draw[draw=black,fill=darkorange25512714,opacity=0.1] (axis cs:1.2,2.2) circle (0.33);
\draw[draw=black,fill=darkorange25512714,opacity=0.1] (axis cs:1.65,2.25) circle (0.33);
\draw[draw=black,fill=darkorange25512714,opacity=0.1] (axis cs:2.2,2.1) circle (0.33);
\draw[draw=black,fill=darkorange25512714,opacity=0.1] (axis cs:2.5,1.5) circle (0.33);
\draw[draw=black,fill=darkorange25512714,opacity=0.1] (axis cs:2.2,0.8) circle (0.33);
\draw[draw=black,fill=darkorange25512714,opacity=0.1] (axis cs:2.5,1.2) circle (0.33);
\draw[draw=black,fill=darkorange25512714,opacity=0.1] (axis cs:2.55,0.45) circle (0.33);
\draw[draw=black,fill=darkorange25512714,opacity=0.1] (axis cs:3.2,0.55) circle (0.33);
\draw[draw=black,fill=darkorange25512714,opacity=0.1] (axis cs:3.5,1) circle (0.33);
\draw[draw=black,fill=darkorange25512714,opacity=0.1] (axis cs:3.1,1.5) circle (0.33);
\addplot [semithick, steelblue31119180, opacity=0.8]
table {%
0.8 0.8
0.75 1.25
};
\addplot [semithick, steelblue31119180, opacity=0.8]
table {%
0.75 1.25
0.7 1.7
};
\addplot [semithick, steelblue31119180, opacity=0.8]
table {%
1.2 2.2
1.65 2.25
};
\addplot [semithick, steelblue31119180, opacity=0.8]
table {%
1.65 2.25
2.2 2.1
};
\addplot [semithick, steelblue31119180, opacity=0.8]
table {%
2.5 1.5
2.5 1.2
};
\addplot [semithick, steelblue31119180, opacity=0.8]
table {%
2.5 1.5
3.1 1.5
};
\addplot [semithick, steelblue31119180, opacity=0.8]
table {%
2.2 0.8
2.5 1.2
};
\addplot [semithick, steelblue31119180, opacity=0.8]
table {%
2.2 0.8
2.55 0.45
};
\addplot [semithick, steelblue31119180, opacity=0.8]
table {%
2.55 0.45
3.2 0.55
};
\addplot [semithick, steelblue31119180, opacity=0.8]
table {%
3.2 0.55
3.5 1
};
\addplot [semithick, steelblue31119180, opacity=0.8]
table {%
3.5 1
3.1 1.5
};
\end{axis}

\end{tikzpicture}
		\caption{$t=0.67$}
		\label{fig:t=0.33}
		\vspace{5mm}
	\end{subfigure}
	\begin{subfigure}[b]{0.3\textwidth}
		\centering
		\begin{tikzpicture}[scale=0.63, every node/.style={scale=1}]
\definecolor{crimson2143940}{RGB}{214,39,40}
\definecolor{darkgray176}{RGB}{176,176,176}
\definecolor{darkorange25512714}{RGB}{255,127,14}
\definecolor{steelblue31119180}{RGB}{31,119,180}
\begin{axis}[
unit vector ratio*=1 1 1,
tick align=outside,
tick pos=left,
x grid style={darkgray176},
xmin=-.3, xmax=4.5,
xtick style={color=black},
y grid style={darkgray176},
ymin=-0.5, ymax=3.2,
ytick style={color=black}
]
\addplot [draw=crimson2143940, fill=crimson2143940, mark=*, only marks]
table{%
x  y
1.5 0.5
0.8 0.8
0.75 1.25
0.7 1.7
1.2 2.2
1.65 2.25
2.2 2.1
2.5 1.5
2.2 0.8
2.5 1.2
2.55 0.45
3.2 0.55
3.5 1
3.1 1.5
};
\draw[draw=black,fill=darkorange25512714,opacity=0.1] (axis cs:1.5,0.5) circle (0.4);
\draw[draw=black,fill=darkorange25512714,opacity=0.1] (axis cs:0.8,0.8) circle (0.4);
\draw[draw=black,fill=darkorange25512714,opacity=0.1] (axis cs:0.75,1.25) circle (0.4);
\draw[draw=black,fill=darkorange25512714,opacity=0.1] (axis cs:0.7,1.7) circle (0.4);
\draw[draw=black,fill=darkorange25512714,opacity=0.1] (axis cs:1.2,2.2) circle (0.4);
\draw[draw=black,fill=darkorange25512714,opacity=0.1] (axis cs:1.65,2.25) circle (0.4);
\draw[draw=black,fill=darkorange25512714,opacity=0.1] (axis cs:2.2,2.1) circle (0.4);
\draw[draw=black,fill=darkorange25512714,opacity=0.1] (axis cs:2.5,1.5) circle (0.4);
\draw[draw=black,fill=darkorange25512714,opacity=0.1] (axis cs:2.2,0.8) circle (0.4);
\draw[draw=black,fill=darkorange25512714,opacity=0.1] (axis cs:2.5,1.2) circle (0.4);
\draw[draw=black,fill=darkorange25512714,opacity=0.1] (axis cs:2.55,0.45) circle (0.4);
\draw[draw=black,fill=darkorange25512714,opacity=0.1] (axis cs:3.2,0.55) circle (0.4);
\draw[draw=black,fill=darkorange25512714,opacity=0.1] (axis cs:3.5,1) circle (0.4);
\draw[draw=black,fill=darkorange25512714,opacity=0.1] (axis cs:3.1,1.5) circle (0.4);
\path [draw=none, fill=steelblue31119180, fill opacity=0.15]
(axis cs:2.5,1.5)
--(axis cs:2.2,0.8)
--(axis cs:2.5,1.2)
--cycle;
\path [draw=none, fill=steelblue31119180, fill opacity=0.15]
(axis cs:2.5,1.5)
--(axis cs:2.5,1.2)
--(axis cs:3.1,1.5)
--cycle;
\path [draw=none, fill=steelblue31119180, fill opacity=0.15]
(axis cs:2.2,0.8)
--(axis cs:2.5,1.2)
--(axis cs:2.55,0.45)
--cycle;
\addplot [semithick, steelblue31119180, opacity=0.8]
table {%
1.5 0.5
0.8 0.8
};
\addplot [semithick, steelblue31119180, opacity=0.8]
table {%
1.5 0.5
2.2 0.8
};
\addplot [semithick, steelblue31119180, opacity=0.8]
table {%
0.8 0.8
0.75 1.25
};
\addplot [semithick, steelblue31119180, opacity=0.8]
table {%
0.75 1.25
0.7 1.7
};
\addplot [semithick, steelblue31119180, opacity=0.8]
table {%
0.7 1.7
1.2 2.2
};
\addplot [semithick, steelblue31119180, opacity=0.8]
table {%
1.2 2.2
1.65 2.25
};
\addplot [semithick, steelblue31119180, opacity=0.8]
table {%
1.65 2.25
2.2 2.1
};
\addplot [semithick, steelblue31119180, opacity=0.8]
table {%
2.2 2.1
2.5 1.5
};
\addplot [semithick, steelblue31119180, opacity=0.8]
table {%
2.5 1.5
2.2 0.8
};
\addplot [semithick, steelblue31119180, opacity=0.8]
table {%
2.5 1.5
2.5 1.2
};
\addplot [semithick, steelblue31119180, opacity=0.8]
table {%
2.5 1.5
3.1 1.5
};
\addplot [semithick, steelblue31119180, opacity=0.8]
table {%
2.2 0.8
2.5 1.2
};
\addplot [semithick, steelblue31119180, opacity=0.8]
table {%
2.2 0.8
2.55 0.45
};
\addplot [semithick, steelblue31119180, opacity=0.8]
table {%
2.5 1.2
2.55 0.45
};
\addplot [semithick, steelblue31119180, opacity=0.8]
table {%
2.5 1.2
3.1 1.5
};
\addplot [semithick, steelblue31119180, opacity=0.8]
table {%
2.55 0.45
3.2 0.55
};
\addplot [semithick, steelblue31119180, opacity=0.8]
table {%
3.2 0.55
3.5 1
};
\addplot [semithick, steelblue31119180, opacity=0.8]
table {%
3.5 1
3.1 1.5
};
\end{axis}
\end{tikzpicture}
		\caption{$t=0.8$}
		\label{fig:t=0.4}
	\end{subfigure}
	\hfill
	\begin{subfigure}[b]{0.3\textwidth}
		\centering
		\begin{tikzpicture}[scale=0.63, every node/.style={scale=1}]
\definecolor{crimson2143940}{RGB}{214,39,40}
\definecolor{darkgray176}{RGB}{176,176,176}
\definecolor{darkorange25512714}{RGB}{255,127,14}
\definecolor{steelblue31119180}{RGB}{31,119,180}
\begin{axis}[
unit vector ratio*=1 1 1,
tick align=outside,
tick pos=left,
x grid style={darkgray176},
xmin=-.3, xmax=4.5,
xtick style={color=black},
y grid style={darkgray176},
ymin=-0.5, ymax=3.2,
ytick style={color=black}
]
\addplot [draw=crimson2143940, fill=crimson2143940, mark=*, only marks]
table{%
x  y
1.5 0.5
0.8 0.8
0.75 1.25
0.7 1.7
1.2 2.2
1.65 2.25
2.2 2.1
2.5 1.5
2.2 0.8
2.5 1.2
2.55 0.45
3.2 0.55
3.5 1
3.1 1.5
};
\draw[draw=black,fill=darkorange25512714,opacity=0.1] (axis cs:1.5,0.5) circle (0.5);
\draw[draw=black,fill=darkorange25512714,opacity=0.1] (axis cs:0.8,0.8) circle (0.5);
\draw[draw=black,fill=darkorange25512714,opacity=0.1] (axis cs:0.75,1.25) circle (0.5);
\draw[draw=black,fill=darkorange25512714,opacity=0.1] (axis cs:0.7,1.7) circle (0.5);
\draw[draw=black,fill=darkorange25512714,opacity=0.1] (axis cs:1.2,2.2) circle (0.5);
\draw[draw=black,fill=darkorange25512714,opacity=0.1] (axis cs:1.65,2.25) circle (0.5);
\draw[draw=black,fill=darkorange25512714,opacity=0.1] (axis cs:2.2,2.1) circle (0.5);
\draw[draw=black,fill=darkorange25512714,opacity=0.1] (axis cs:2.5,1.5) circle (0.5);
\draw[draw=black,fill=darkorange25512714,opacity=0.1] (axis cs:2.2,0.8) circle (0.5);
\draw[draw=black,fill=darkorange25512714,opacity=0.1] (axis cs:2.5,1.2) circle (0.5);
\draw[draw=black,fill=darkorange25512714,opacity=0.1] (axis cs:2.55,0.45) circle (0.5);
\draw[draw=black,fill=darkorange25512714,opacity=0.1] (axis cs:3.2,0.55) circle (0.5);
\draw[draw=black,fill=darkorange25512714,opacity=0.1] (axis cs:3.5,1) circle (0.5);
\draw[draw=black,fill=darkorange25512714,opacity=0.1] (axis cs:3.1,1.5) circle (0.5);
\path [draw=none, fill=steelblue31119180, fill opacity=0.15]
(axis cs:0.8,0.8)
--(axis cs:0.75,1.25)
--(axis cs:0.7,1.7)
--cycle;
\path [draw=none, fill=steelblue31119180, fill opacity=0.15]
(axis cs:2.2,2.1)
--(axis cs:2.5,1.5)
--(axis cs:2.5,1.2)
--cycle;
\path [draw=none, fill=steelblue31119180, fill opacity=0.15]
(axis cs:2.5,1.5)
--(axis cs:2.2,0.8)
--(axis cs:2.5,1.2)
--cycle;
\path [draw=none, fill=steelblue31119180, fill opacity=0.15]
(axis cs:2.5,1.5)
--(axis cs:2.5,1.2)
--(axis cs:3.1,1.5)
--cycle;
\path [draw=none, fill=steelblue31119180, fill opacity=0.15]
(axis cs:2.2,0.8)
--(axis cs:2.5,1.2)
--(axis cs:2.55,0.45)
--cycle;
\path [draw=none, fill=steelblue31119180, fill opacity=0.15]
(axis cs:2.5,1.2)
--(axis cs:2.55,0.45)
--(axis cs:3.2,0.55)
--cycle;
\path [draw=none, fill=steelblue31119180, fill opacity=0.15]
(axis cs:2.5,1.2)
--(axis cs:3.2,0.55)
--(axis cs:3.1,1.5)
--cycle;
\path [draw=none, fill=steelblue31119180, fill opacity=0.15]
(axis cs:3.2,0.55)
--(axis cs:3.5,1)
--(axis cs:3.1,1.5)
--cycle;
\addplot [semithick, steelblue31119180, opacity=0.8]
table {%
1.5 0.5
0.8 0.8
};
\addplot [semithick, steelblue31119180, opacity=0.8]
table {%
1.5 0.5
2.2 0.8
};
\addplot [semithick, steelblue31119180, opacity=0.8]
table {%
0.8 0.8
0.75 1.25
};
\addplot [semithick, steelblue31119180, opacity=0.8]
table {%
0.8 0.8
0.7 1.7
};
\addplot [semithick, steelblue31119180, opacity=0.8]
table {%
0.75 1.25
0.7 1.7
};
\addplot [semithick, steelblue31119180, opacity=0.8]
table {%
0.7 1.7
1.2 2.2
};
\addplot [semithick, steelblue31119180, opacity=0.8]
table {%
1.2 2.2
1.65 2.25
};
\addplot [semithick, steelblue31119180, opacity=0.8]
table {%
1.65 2.25
2.2 2.1
};
\addplot [semithick, steelblue31119180, opacity=0.8]
table {%
2.2 2.1
2.5 1.5
};
\addplot [semithick, steelblue31119180, opacity=0.8]
table {%
2.2 2.1
2.5 1.2
};
\addplot [semithick, steelblue31119180, opacity=0.8]
table {%
2.5 1.5
2.2 0.8
};
\addplot [semithick, steelblue31119180, opacity=0.8]
table {%
2.5 1.5
2.5 1.2
};
\addplot [semithick, steelblue31119180, opacity=0.8]
table {%
2.5 1.5
3.1 1.5
};
\addplot [semithick, steelblue31119180, opacity=0.8]
table {%
2.2 0.8
2.5 1.2
};
\addplot [semithick, steelblue31119180, opacity=0.8]
table {%
2.2 0.8
2.55 0.45
};
\addplot [semithick, steelblue31119180, opacity=0.8]
table {%
2.5 1.2
2.55 0.45
};
\addplot [semithick, steelblue31119180, opacity=0.8]
table {%
2.5 1.2
3.2 0.55
};
\addplot [semithick, steelblue31119180, opacity=0.8]
table {%
2.5 1.2
3.1 1.5
};
\addplot [semithick, steelblue31119180, opacity=0.8]
table {%
2.55 0.45
3.2 0.55
};
\addplot [semithick, steelblue31119180, opacity=0.8]
table {%
3.2 0.55
3.5 1
};
\addplot [semithick, steelblue31119180, opacity=0.8]
table {%
3.2 0.55
3.1 1.5
};
\addplot [semithick, steelblue31119180, opacity=0.8]
table {%
3.5 1
3.1 1.5
};
\end{axis}
\end{tikzpicture}
		\caption{$t=1$}
		\label{fig:t=0.5}
	\end{subfigure}
	\hfill
	\begin{subfigure}[b]{0.3\textwidth}
		\centering
		\input{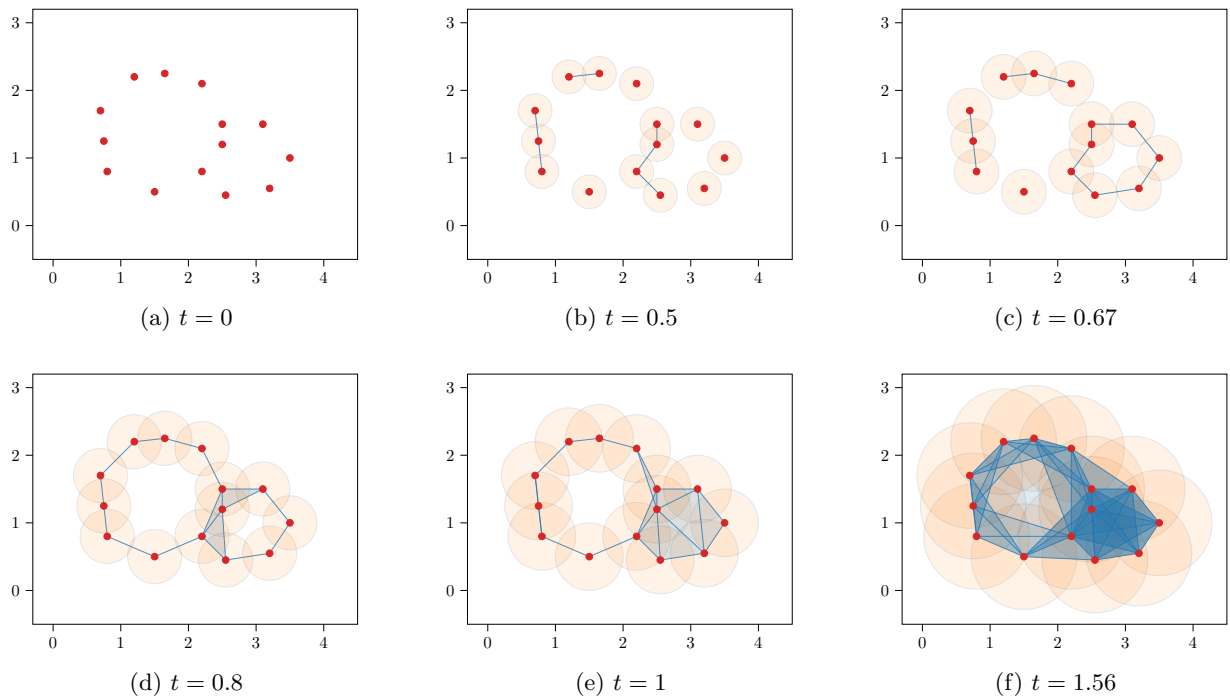}
		\caption{$t=1.56$}
		\label{fig:t=0.78}
	\end{subfigure}
	\caption{Constructing a simplicial complex from a point cloud using the Vietoris--Rips complex. \protect\subref{fig:t=0} The initial point cloud $X$ is obtained for $t=0$ corresponding to the birth of all $0$-dimensional features (connected components). \protect\subref{fig:t=0.25} Some connected components have died due to the creation of $1$-simplices. \protect\subref{fig:t=0.33} A $1$-dimensional hole (loop) has appeared. \protect\subref{fig:t=0.4} The second loop has been created while all but one of the connected components have vanished. \protect\subref{fig:t=0.5} The first loop has been filled yielding its death. \protect\subref{fig:t=0.78} All loops have disappeared; only the connected component of infinite persistence remains.}
	\label{fig:PH_procedure}
\end{figure}

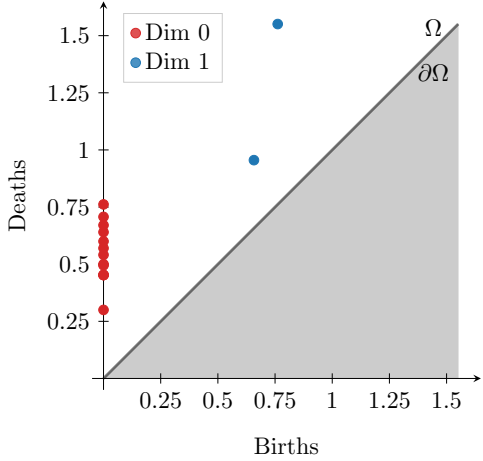
\begin{figure}[t!]
	\centering
        \resizebox{0.4\textwidth}{!}{
	\begin{tikzpicture}[scale=1, every node/.style={scale=1}]
\definecolor{crimson2143940}{RGB}{214,39,40}
\definecolor{darkgray176}{RGB}{176,176,176}
\definecolor{lightgray204}{RGB}{204,204,204}
\definecolor{steelblue31119180}{RGB}{31,119,180}
\begin{axis}[
legend cell align={left},
legend style={
  fill opacity=0.8,
  draw opacity=1,
  text opacity=1,
  at={(0.08,0.97)},
  anchor=north west,
  draw=lightgray204
},
unit vector ratio*=1 1 1,
axis lines=center,
x label style={at={(axis description cs:0.5,-0.1)},anchor=north},
y label style={at={(axis description cs:-0.15,.5)},rotate=90,anchor=south},
x grid style={darkgray176},
xlabel={Births},
xmin=-0.05, xmax=1.65,
xtick style={color=black},
xtick={0.25,0.5,0.75,1,1.25,1.5},
y grid style={darkgray176},
ylabel={Deaths},
ymin=-0.05, ymax=1.65,
ytick style={color=black},
ytick={0.25,0.5,0.75,1,1.25,1.5}
]
\addplot [draw=crimson2143940, fill=crimson2143940, mark=*, only marks]
table{%
x  y
0 0.300000011920929
0 0.452769249677658
0 0.452769249677658
0 0.452769249677658
0 0.494974732398987
0 0.5
0 0.540832698345184
0 0.570087730884552
0 0.600000023841858
0 0.640312433242798
0 0.670820415019989
0 0.70710676908493
0 0.761577308177948
};
\addlegendentry{Dim $0$}
\addplot [draw=steelblue31119180, fill=steelblue31119180, mark=*, only marks]
table{%
x  y
0.761577308177948 1.55080628395081
0.657647311687469 0.955248653888702
};
\addlegendentry{Dim $1$}
\path [draw=black, fill=black, opacity=0.2]
(axis cs:0,0)
--(axis cs:0,0)
--(axis cs:0.01,0.01)
--(axis cs:0.02,0.02)
--(axis cs:0.03,0.03)
--(axis cs:0.04,0.04)
--(axis cs:0.05,0.05)
--(axis cs:0.06,0.06)
--(axis cs:0.07,0.07)
--(axis cs:0.08,0.08)
--(axis cs:0.09,0.09)
--(axis cs:0.1,0.1)
--(axis cs:0.11,0.11)
--(axis cs:0.12,0.12)
--(axis cs:0.13,0.13)
--(axis cs:0.14,0.14)
--(axis cs:0.15,0.15)
--(axis cs:0.16,0.16)
--(axis cs:0.17,0.17)
--(axis cs:0.18,0.18)
--(axis cs:0.19,0.19)
--(axis cs:0.2,0.2)
--(axis cs:0.21,0.21)
--(axis cs:0.22,0.22)
--(axis cs:0.23,0.23)
--(axis cs:0.24,0.24)
--(axis cs:0.25,0.25)
--(axis cs:0.26,0.26)
--(axis cs:0.27,0.27)
--(axis cs:0.28,0.28)
--(axis cs:0.29,0.29)
--(axis cs:0.3,0.3)
--(axis cs:0.31,0.31)
--(axis cs:0.32,0.32)
--(axis cs:0.33,0.33)
--(axis cs:0.34,0.34)
--(axis cs:0.35,0.35)
--(axis cs:0.36,0.36)
--(axis cs:0.37,0.37)
--(axis cs:0.38,0.38)
--(axis cs:0.39,0.39)
--(axis cs:0.4,0.4)
--(axis cs:0.41,0.41)
--(axis cs:0.42,0.42)
--(axis cs:0.43,0.43)
--(axis cs:0.44,0.44)
--(axis cs:0.45,0.45)
--(axis cs:0.46,0.46)
--(axis cs:0.47,0.47)
--(axis cs:0.48,0.48)
--(axis cs:0.49,0.49)
--(axis cs:0.5,0.5)
--(axis cs:0.51,0.51)
--(axis cs:0.52,0.52)
--(axis cs:0.53,0.53)
--(axis cs:0.54,0.54)
--(axis cs:0.55,0.55)
--(axis cs:0.56,0.56)
--(axis cs:0.57,0.57)
--(axis cs:0.58,0.58)
--(axis cs:0.59,0.59)
--(axis cs:0.6,0.6)
--(axis cs:0.61,0.61)
--(axis cs:0.62,0.62)
--(axis cs:0.63,0.63)
--(axis cs:0.64,0.64)
--(axis cs:0.65,0.65)
--(axis cs:0.66,0.66)
--(axis cs:0.67,0.67)
--(axis cs:0.68,0.68)
--(axis cs:0.69,0.69)
--(axis cs:0.7,0.7)
--(axis cs:0.71,0.71)
--(axis cs:0.72,0.72)
--(axis cs:0.73,0.73)
--(axis cs:0.74,0.74)
--(axis cs:0.75,0.75)
--(axis cs:0.76,0.76)
--(axis cs:0.77,0.77)
--(axis cs:0.78,0.78)
--(axis cs:0.79,0.79)
--(axis cs:0.8,0.8)
--(axis cs:0.81,0.81)
--(axis cs:0.82,0.82)
--(axis cs:0.83,0.83)
--(axis cs:0.84,0.84)
--(axis cs:0.85,0.85)
--(axis cs:0.86,0.86)
--(axis cs:0.87,0.87)
--(axis cs:0.88,0.88)
--(axis cs:0.89,0.89)
--(axis cs:0.9,0.9)
--(axis cs:0.91,0.91)
--(axis cs:0.92,0.92)
--(axis cs:0.93,0.93)
--(axis cs:0.94,0.94)
--(axis cs:0.95,0.95)
--(axis cs:0.96,0.96)
--(axis cs:0.97,0.97)
--(axis cs:0.98,0.98)
--(axis cs:0.99,0.99)
--(axis cs:1,1)
--(axis cs:1.01,1.01)
--(axis cs:1.02,1.02)
--(axis cs:1.03,1.03)
--(axis cs:1.04,1.04)
--(axis cs:1.05,1.05)
--(axis cs:1.06,1.06)
--(axis cs:1.07,1.07)
--(axis cs:1.08,1.08)
--(axis cs:1.09,1.09)
--(axis cs:1.1,1.1)
--(axis cs:1.11,1.11)
--(axis cs:1.12,1.12)
--(axis cs:1.13,1.13)
--(axis cs:1.14,1.14)
--(axis cs:1.15,1.15)
--(axis cs:1.16,1.16)
--(axis cs:1.17,1.17)
--(axis cs:1.18,1.18)
--(axis cs:1.19,1.19)
--(axis cs:1.2,1.2)
--(axis cs:1.21,1.21)
--(axis cs:1.22,1.22)
--(axis cs:1.23,1.23)
--(axis cs:1.24,1.24)
--(axis cs:1.25,1.25)
--(axis cs:1.26,1.26)
--(axis cs:1.27,1.27)
--(axis cs:1.28,1.28)
--(axis cs:1.29,1.29)
--(axis cs:1.3,1.3)
--(axis cs:1.31,1.31)
--(axis cs:1.32,1.32)
--(axis cs:1.33,1.33)
--(axis cs:1.34,1.34)
--(axis cs:1.35,1.35)
--(axis cs:1.36,1.36)
--(axis cs:1.37,1.37)
--(axis cs:1.38,1.38)
--(axis cs:1.39,1.39)
--(axis cs:1.4,1.4)
--(axis cs:1.41,1.41)
--(axis cs:1.42,1.42)
--(axis cs:1.43,1.43)
--(axis cs:1.44,1.44)
--(axis cs:1.45,1.45)
--(axis cs:1.46,1.46)
--(axis cs:1.47,1.47)
--(axis cs:1.48,1.48)
--(axis cs:1.49,1.49)
--(axis cs:1.5,1.5)
--(axis cs:1.51,1.51)
--(axis cs:1.52,1.52)
--(axis cs:1.53,1.53)
--(axis cs:1.54,1.54)
--(axis cs:1.55,1.55)
--(axis cs:1.55,0)
--(axis cs:1.55,0)
--(axis cs:1.54,0)
--(axis cs:1.53,0)
--(axis cs:1.52,0)
--(axis cs:1.51,0)
--(axis cs:1.5,0)
--(axis cs:1.49,0)
--(axis cs:1.48,0)
--(axis cs:1.47,0)
--(axis cs:1.46,0)
--(axis cs:1.45,0)
--(axis cs:1.44,0)
--(axis cs:1.43,0)
--(axis cs:1.42,0)
--(axis cs:1.41,0)
--(axis cs:1.4,0)
--(axis cs:1.39,0)
--(axis cs:1.38,0)
--(axis cs:1.37,0)
--(axis cs:1.36,0)
--(axis cs:1.35,0)
--(axis cs:1.34,0)
--(axis cs:1.33,0)
--(axis cs:1.32,0)
--(axis cs:1.31,0)
--(axis cs:1.3,0)
--(axis cs:1.29,0)
--(axis cs:1.28,0)
--(axis cs:1.27,0)
--(axis cs:1.26,0)
--(axis cs:1.25,0)
--(axis cs:1.24,0)
--(axis cs:1.23,0)
--(axis cs:1.22,0)
--(axis cs:1.21,0)
--(axis cs:1.2,0)
--(axis cs:1.19,0)
--(axis cs:1.18,0)
--(axis cs:1.17,0)
--(axis cs:1.16,0)
--(axis cs:1.15,0)
--(axis cs:1.14,0)
--(axis cs:1.13,0)
--(axis cs:1.12,0)
--(axis cs:1.11,0)
--(axis cs:1.1,0)
--(axis cs:1.09,0)
--(axis cs:1.08,0)
--(axis cs:1.07,0)
--(axis cs:1.06,0)
--(axis cs:1.05,0)
--(axis cs:1.04,0)
--(axis cs:1.03,0)
--(axis cs:1.02,0)
--(axis cs:1.01,0)
--(axis cs:1,0)
--(axis cs:0.99,0)
--(axis cs:0.98,0)
--(axis cs:0.97,0)
--(axis cs:0.96,0)
--(axis cs:0.95,0)
--(axis cs:0.94,0)
--(axis cs:0.93,0)
--(axis cs:0.92,0)
--(axis cs:0.91,0)
--(axis cs:0.9,0)
--(axis cs:0.89,0)
--(axis cs:0.88,0)
--(axis cs:0.87,0)
--(axis cs:0.86,0)
--(axis cs:0.85,0)
--(axis cs:0.84,0)
--(axis cs:0.83,0)
--(axis cs:0.82,0)
--(axis cs:0.81,0)
--(axis cs:0.8,0)
--(axis cs:0.79,0)
--(axis cs:0.78,0)
--(axis cs:0.77,0)
--(axis cs:0.76,0)
--(axis cs:0.75,0)
--(axis cs:0.74,0)
--(axis cs:0.73,0)
--(axis cs:0.72,0)
--(axis cs:0.71,0)
--(axis cs:0.7,0)
--(axis cs:0.69,0)
--(axis cs:0.68,0)
--(axis cs:0.67,0)
--(axis cs:0.66,0)
--(axis cs:0.65,0)
--(axis cs:0.64,0)
--(axis cs:0.63,0)
--(axis cs:0.62,0)
--(axis cs:0.61,0)
--(axis cs:0.6,0)
--(axis cs:0.59,0)
--(axis cs:0.58,0)
--(axis cs:0.57,0)
--(axis cs:0.56,0)
--(axis cs:0.55,0)
--(axis cs:0.54,0)
--(axis cs:0.53,0)
--(axis cs:0.52,0)
--(axis cs:0.51,0)
--(axis cs:0.5,0)
--(axis cs:0.49,0)
--(axis cs:0.48,0)
--(axis cs:0.47,0)
--(axis cs:0.46,0)
--(axis cs:0.45,0)
--(axis cs:0.44,0)
--(axis cs:0.43,0)
--(axis cs:0.42,0)
--(axis cs:0.41,0)
--(axis cs:0.4,0)
--(axis cs:0.39,0)
--(axis cs:0.38,0)
--(axis cs:0.37,0)
--(axis cs:0.36,0)
--(axis cs:0.35,0)
--(axis cs:0.34,0)
--(axis cs:0.33,0)
--(axis cs:0.32,0)
--(axis cs:0.31,0)
--(axis cs:0.3,0)
--(axis cs:0.29,0)
--(axis cs:0.28,0)
--(axis cs:0.27,0)
--(axis cs:0.26,0)
--(axis cs:0.25,0)
--(axis cs:0.24,0)
--(axis cs:0.23,0)
--(axis cs:0.22,0)
--(axis cs:0.21,0)
--(axis cs:0.2,0)
--(axis cs:0.19,0)
--(axis cs:0.18,0)
--(axis cs:0.17,0)
--(axis cs:0.16,0)
--(axis cs:0.15,0)
--(axis cs:0.14,0)
--(axis cs:0.13,0)
--(axis cs:0.12,0)
--(axis cs:0.11,0)
--(axis cs:0.1,0)
--(axis cs:0.09,0)
--(axis cs:0.08,0)
--(axis cs:0.07,0)
--(axis cs:0.06,0)
--(axis cs:0.05,0)
--(axis cs:0.04,0)
--(axis cs:0.03,0)
--(axis cs:0.02,0)
--(axis cs:0.01,0)
--(axis cs:0,0)
--cycle;
\addplot [very thick, black, opacity=0.5, forget plot]
table {%
0 0
1.55080628395081 1.55080628395081
};
\end{axis}
\node at (5,5.3) {$\Omega$};
\node at (5,4.65) {$\partial\Omega$};

\end{tikzpicture}}
	\caption{Persistence diagram obtained from the Vietoris--Rips filtration of the point cloud $X$ shown in \cref{fig:PH_procedure} illustrating the birth and death times of the $0$- and $1$-dimensional holes. Here, the $0$-dimensional feature of infinite persistence is discarded.}
	\label{fig:PD}
\end{figure}

\begin{definition}[Persistence diagram]
A \emph{persistence diagram} is a locally finite multiset of points supported on $\Omega \coloneqq \{(t_1,t_2)\in \mathbb R^2 \colon t_1 < t_2\}$ together with points on the diagonal ${\partial\Omega \coloneqq \{(t,t) \in \mathbb{R}^2\}}$ counted with infinite multiplicity.
\end{definition}

\paragraph{Distances Between Persistence Diagrams.}  The collection of all persistence diagrams constitutes a metric space when equipped with an appropriate distance function.  There exist various metrics on the space of persistence diagrams; we focus on the following two distance functions that are widely used in computations and applications in TDA.

\begin{definition}
Let $D$ and $D'$ be two persistence diagrams, and let $1\leq q \leq \infty$. For ${1\leq p<\infty}$, the $p$-\emph{Wasserstein distance} between $D$ and $D'$ is given by
\begin{align}\label{eq:WD_p}
    W_{p,q}(D,D') \coloneqq \inf_{\gamma} \left(\sum_{x\in D\cup \partial \Omega} \|x-\gamma(x)\|_q^p\right)^{1/p},
\end{align}
where the infimum is taken over all bijections $\gamma$ between $D\cup \partial \Omega$ and $D'\cup \partial \Omega$. For $p=\infty$, the \emph{bottleneck distance} is given by
\begin{align}\label{eq:Bottleneck}
    W_{\infty,q}(D,D') \coloneqq \inf_{\gamma} \sup_{x\in D \cup \partial \Omega} \|x-\gamma(x)\|_q.
\end{align}
\end{definition}

Whenever $p=q$, we write $W_{p}(\cdot,\cdot)\coloneqq W_{p,p}(\cdot,\cdot)$ for notational simplicity.
A bijection $\gamma$ that achieves \cref{eq:WD_p} (or \cref{eq:Bottleneck}) is called an \emph{optimal matching}.
The $p$-\emph{total persistence} of a persistence diagram $D$ is given by 
\begin{align*}
    \mathrm{Pers}_p(D)\coloneqq \left(\sum_{x\in D} \left\lVert x-x^\perp\right\rVert_q^p\right)^{1/p}, \quad 1\leq p<\infty,
\end{align*}
where $x^\perp$ denotes the orthogonal projection of $x$ onto $\partial \Omega$.
For $p=\infty$, set
\begin{align*}
    \mathrm{Pers}_\infty(D)\coloneqq \sup_{x\in D} \left\lVert x-x^\perp\right\rVert_q.
\end{align*}
Denote by $\mathcal{D}^p$ the set of persistence diagrams $D$ with finite $p$-total persistence, $\mathrm{Pers}_p(D)<\infty$.

The space of persistence diagrams, as random objects, is a setting where probabilistic and statistical studies are valid: the metric space $(\mathcal{D}^p,W_{p,q})$ is complete and separable for any $1\leq p<\infty$ and $1\leq q\leq \infty$ \citep{mileyko2011probability}.
As a result, probability measures on $(\mathcal{D}^p,W_{p,q})$ are well-defined, which implies the existence of other standard statistical and probabilistic objects on $(\mathcal{D}^p,W_{p,q})$ such as expectations and variances. However, geometrically, the metric space of persistence diagrams when $p=2$ is highly nonlinear (and as yet geometrically uncharacterized for $p>2$), which makes even simple statistical and machine learning tasks with persistence diagrams very challenging \citep{turner-2014-frechet}. To deal with this issue, persistence diagrams are typically embedded into a Banach or Hilbert space via vectorizations, such as explicit feature maps \citep{bubenik2015statistical, adams2017persistence_vectorisation} or implicit kernel methods \citep{Kernel_PD_reininghaus2015stable}. Note, however, that the metric structure of the space of persistence diagrams is not preserved by these embeddings \citep{bubenik2020embeddings}.

\paragraph{Persistence Measures.}
Persistence diagrams generalize well to a measure-theoretic setting, which has desirable probabilistic and statistical properties \citep{Chazal_Divol_EPD}. 
Since persistence diagrams are simply multisets of points supported on $\Omega$, it is natural to represent a persistence diagram as a discrete measure in the following manner. Let $D\in \mathcal{D}^p$, then 
\begin{align}\label{eq:discrete_measure}
    D = \sum_{x} \delta_x,
\end{align}
where $\delta_x$ denotes the Dirac measure in $x\in \Omega$, and where the sum is taken over all points $x=(t_1,t_2)$ that belong to the persistence diagram $D$. Points in $D$ corresponding to topological features of infinite persistence are disregarded in the representation \cref{eq:discrete_measure}. 

Typically, the underlying point cloud $X$ of a persistence diagram $D$ is random so that $D$ becomes a random measure. To analyze the expected behavior of any linear descriptor of the form $\Psi_h(D) \coloneqq \int_{\Omega}h\,\mathrm{d}D$, for some Banach space-valued function $h$ on $\Omega$, it suffices to study the \emph{expected persistence diagram}.

\begin{definition}[Expected persistence diagram, \citep{Chazal_Divol_EPD}]\label{def:EPD}
Let $D$ be a persistence diagram resulting from a random point cloud. The \emph{expected persistence diagram (EPD)} is defined to be the deterministic measure
\begin{align*}
    \mathbb E [D](A) \coloneqq \mathbb E [D(A)], \quad \text{for all Borel sets $A\subset \Omega$.}
\end{align*}
\end{definition}
\cite{Chazal_Divol_EPD} show that the expected persistence diagram has a density with respect to the Lebesgue measure. 
The expected persistence diagram is in general not a persistence diagram, but lies in a natural extension of the space of persistence diagrams, namely, the space of \emph{persistence measures} \citep{divol_und_top}.
The space of persistence measures $\mathcal{M}^p$, $1\leq p\leq \infty$, is defined to be the set of all non-negative Radon measures $\mu$ supported on $\Omega$ such that $\mathrm{Pers}_p(\mu)<\infty$. Here,
\begin{align*}
    \mathrm{Pers}_p(\mu) \coloneqq \begin{cases}
    \int_{\Omega}\left\lVert x-x^\perp\right\rVert_q^p\,\mathrm{d}\mu(x) &\text{ if } 1\leq p<\infty,\\
    \sup_{x\in\mathrm{supp}(\mu)} \left\lVert x-x^\perp\right\rVert_q &\text{ if } p=\infty,
    \end{cases}
    \quad \text{for $1\leq q\leq \infty.$}
\end{align*}
Equip the space of persistence measures $\mathcal{M}^p$ with the following metric \citep{divol_und_top}: for $\mu,\nu\in\mathcal{M}^p$, define
\begin{align}\label{eq:OT_p}
    \mathrm{OT}_{p,q}(\mu,\nu) \coloneqq \inf_{\gamma \in \mathrm{Adm}(\mu,\nu)} \left(\int_{\Bar{\Omega}\times \Bar{\Omega}} \| x-y \|_q^p\,\mathrm{d}\gamma(x,y)\right)^{1/p}, \quad 1\leq p <\infty,
\end{align}
where $\Bar{\Omega} \coloneqq \Omega \cup \partial \Omega$, and where $\mathrm{Adm}(\mu,\nu)$ denotes the set of Radon measures $\gamma$ on $\Bar{\Omega}\times \Bar{\Omega}$ such that
\begin{align*}
    \gamma(A\times \Bar{\Omega}) = \mu(A) \text{ and } \gamma(\Bar{\Omega}\times B) = \nu(B), \text{ for all Borel sets $A,B\subset \Omega$}.
\end{align*}
For $p=\infty$, 
\begin{align*}
    \mathrm{OT}_{\infty,q}(\mu,\nu) \coloneqq \inf_{\gamma \in \mathrm{Adm}(\mu,\nu)} \sup_{(x,y)\in \mathrm{supp}(\gamma)} \| x-y \|_q.
\end{align*}
Whenever $p=q$, we will use the simplified notation $\mathrm{OT}_{p}(\cdot,\cdot)\coloneqq \mathrm{OT}_{p,p}(\cdot,\cdot)$.
We now recall the following results established by \cite{divol_und_top}: The space $(\mathcal{M}^p,\mathrm{OT}_{p,q})$ is complete and separable for $1\leq p<\infty$ and $1\leq q\leq \infty$. The space $(\mathcal{M}^\infty,\mathrm{OT}_{\infty,q})$ is complete but not separable for any $1\leq q\leq \infty$. Note that
$\mathcal{D}^p$ is closed in $\mathcal{M}^p$ with respect to $\mathrm{OT}_{p,q}$. Moreover, for any $\mu,\nu\in\mathcal{D}^p$, it holds that
\begin{align*}
    W_{p,q}(\mu,\nu) = \mathrm{OT}_{p,q}(\mu,\nu).
\end{align*}
Thus, $(\mathcal{M}^p,\mathrm{OT}_{p,q})$ is a natural and proper extension of the space of persistence diagrams $(\mathcal{D}^p,W_{p,q})$. Due to its linear structure the space of persistence measures $\mathcal{M}^p$ is very convenient for both computation and statistical analyses.
\subsection{Wavelets}
Switching perspectives, we now recall some basic elements of wavelet theory.

A wavelet expansion of a function $f\in L^2 (\mathbb R^2)$ corresponds to a multiscale decomposition of $f$, which is obtained by writing $f$ as a sum of its coarse approximation and its local fluctuations. More precisely, consider the closed linear subspace $\overline{\mathrm{span}(\mathcal{V}_j)}$ of $L^2(\mathbb R^2)$, which is generated by some function $\varphi:\mathbb R^2 \to \mathbb R$ according to 
\begin{align*}
    \mathcal{V}_j \coloneqq \left\{2^j\varphi(2^j \cdot +\ell) \colon \ell \in \mathbb Z^2\right\}, \quad j\in \mathbb Z.
\end{align*}
For an appropriate choice of $\varphi$, there exists a nested sequence of closed linear subspaces
\begin{align}\label{eq:subspaces}
    \cdots \subset \overline{\mathrm{span}(\mathcal{V}_{-1})} \subset \overline{\mathrm{span}(\mathcal{V}_{0})} \subset \overline{\mathrm{span}(\mathcal{V}_{1})}\subset \cdots
\end{align}
such that
\begin{enumerate}[label=(\roman*)]
    \item $\cap_{j\in \mathbb Z}\,  \overline{\mathrm{span}(\mathcal{V}_j)}=\{0\}$, 
    \item $\cup_{j\in \mathbb Z}\, \overline{\mathrm{span}(\mathcal{V}_j)}$ is dense in $L^2(\mathbb R^2)$,
    \item $\mathcal{V}_0$ is an orthonormal basis of $\overline{\mathrm{span}(\mathcal{V}_{0})}$, and
    \item for all $h\in L^2(\mathbb R^2)$ and all $j\in \mathbb Z$, $h\in \overline{\mathrm{span}(\mathcal{V}_{j})} \iff h(2\cdot) \in \overline{\mathrm{span}(\mathcal{V}_{j+1})}$.
\end{enumerate}
If these conditions are satisfied, $\{\overline{\mathrm{span}(\mathcal{V}_{j})}\}_{j\in \mathbb Z}$ is called a \emph{multiresolution approximation} of $L^2(\mathbb R^2)$ and $\varphi$ its \emph{scaling function} \citep{cohen2003numerical}.
Thus, projecting $f\in L^2(\mathbb R^2)$ onto the subspace $\overline{\mathrm{span}(\mathcal{V}_j)}$ gives a coarse approximation of $f$ if $j$ is small and a fine approximation if $j$ is large. The details (fluctuations) complementing the approximation between $\overline{\mathrm{span}(\mathcal{V}_j)}$ and $\overline{\mathrm{span}(\mathcal{V}_{j+1})}$ may be governed by the set
\begin{align*}
    \mathcal{W}_{j}\coloneqq \left\{2^j\psi^a(2^j\cdot +\ell),\, 2^j\psi^b(2^j\cdot +\ell),\, 2^j\psi^c(2^j\cdot+ \ell) \colon \ell\in \mathbb Z^2\right\}
\end{align*}
for some functions $\psi^a,\psi ^b,\psi^c \colon \mathbb R^2 \to \mathbb R$ according to
\begin{align}\label{eq:wavelet_mult}
    \overline{\mathrm{span}(\mathcal{V}_{j+1})} = \overline{\mathrm{span}(\mathcal{V}_j)} \oplus \overline{\mathrm{span}(\mathcal{W}_j)}.
\end{align}
Applying \cref{eq:wavelet_mult} recursively, using \cref{eq:subspaces}, and the fact that $\cup_{j\in \mathbb Z}\, \overline{\mathrm{span}(\mathcal{V}_j)}$ is dense in $L^2(\mathbb R^2)$, 
\begin{align*}
    L^2(\mathbb R^2) = \overline{\mathrm{span}(\mathcal{V}_{j_0})} \oplus \bigoplus_{j\geq j_0} \overline{\mathrm{span}(\mathcal{W}_{j})},
\end{align*}
for every $j_0 \in \mathbb N_0$.  Thus, the wavelet expansion of $f\in L^2 (\mathbb R^2)$ reads
\begin{align*}
    f = \sum_{\phi \in \mathcal{V}_{j_0}} \alpha_\phi \phi + \sum_{j\geq j_0} \sum_{\psi \in \mathcal{W}_j} \beta_\psi \psi, 
\end{align*}
where $\alpha_\phi, \beta_\psi \in \mathbb R$ are the wavelet coefficients of this expansion. This decomposition allows for a multiresolution representation of the approximated function because it is equivalent to a description at different scales, where at each scale there exists both a coarse- and fine-grained approximation of the function. Indeed, the first term is a (coarse) approximation of $f$ at scale $j_0$; the second term describes the fine details (fluctuations) \citep{cohen2003numerical,hardle2012wavelets}. 

\paragraph{Approximation with Wavelets.}  
Wavelet transforms are a popular choice for modeling signals and images and for nonparametric density estimation \citep{abramovich2000,hardle2012wavelets}, thanks to their ability to approximate large classes of functions and to provide a locally accurate representation of the data structure.  Wavelets are well known to achieve considerable approximation power in connection to a broad class of functions. Their localized structure (both in space and frequency) allows them to behave smoothly with local irregularities and adapt efficiently to abrupt and small-scale variations in the data. As such, they have been employed for various tasks of filtering, smoothing, and data compression \citep{zhao2005,krommweh2010,fryzlewicz2016}. The use of wavelets in density estimation has a long history dating back to the 1990s \citep{cohen2003numerical,hardle2012wavelets,Donoho_wavelet_est} and wavelet-based density estimators have been successfully applied to a wide range of observational data \citep{abramovich2000,peter2008,hudson2013}. These well-documented advantages motivate the study of wavelet estimators for the expected persistence diagram.

\section{Minimax Density Estimation of the Expected Persistence Diagram}
\label{sec:minimax}
We now present our main results, which are twofold: In our first contribution, we estimate the density of the expected persistence diagram nonparametrically using Haar wavelets and establish its property of minimaxity.  In our second contribution, we propose a sparse thresholding wavelet estimator and show its near-minimaxity.

\subsection{Framework}
Consider a probability distribution $P$ supported on $\mathcal{M}^p$, and denote the expected persistence diagram associated to $P$ by 
\begin{align}\label{eq:epd_prob_dist}
    \mathbf{E}(P)\coloneqq \mathbb E[\mu],
\end{align}
where $\mu$ is a random variable taking values in the set $\mathcal{M}^p$ distributed according to the law $P$. Note that, in this context, it is natural to interpret the expected persistence diagram as the Bochner integral of $\mu$.  

Our goal is to estimate the Lebesgue density of $\mathbf{E}(P)$ based on $N$ independent and identically distributed (i.i.d.)~observations $\{\mu_i\}_{i=1}^N \sim P$. Moreover, we aim to find an estimator which is optimal in the sense of minimax, which we will now define.  Here,  we use the framework introduced by \cite{divol2021estimation}. 

\begin{definition}
\label{def:minimax}
For $R>0$, let
\begin{align*}
   \Omega_R \coloneqq \left\{(t_1,t_2)\in \mathbb R^2 \colon \left\lvert t_1 + R/\sqrt{8}\right\rvert+\left\lvert t_2 - R/\sqrt{8}\right\rvert \leq R/\sqrt{2}\right\}. 
\end{align*}
Denote by $\mathcal{M}^s_{R,M}$ the set of persistence measures $\mu \in \mathcal{M}^s$ that are supported on $\Omega_R$ and satisfy 
$\mathrm{Pers}_s(\mu)\leq M$, where $s\in[0,\infty)$ and $M>0$. Let $\mathcal{P}^s_{R,M}$ be the set of probability distributions on $\mathcal{M}^s_{R,M}$. The \emph{minimax rate} for estimating $\mathbf{E}(P)$ on $\mathcal{P}^s_{R,M}$ is defined to be
\begin{align*}
    \inf_{\hat{\mu}_N} \sup_{P\in\mathcal{P}^s_{R,M}} \mathbb E \left[\mathrm{OT}_p^p\left(\hat{\mu}_N,\mathbf{E}(P)\right)\right],
\end{align*}
where the infimum is taken over all estimators $\hat{\mu}_N$ based on the $N$ i.i.d. samples $\{\mu_i\}_{i=1}^N \sim P$. 
\end{definition}

\begin{theorem}[Lower bound on minimax rate, Theorem 2 in \citep{divol2021estimation}]
Let $1\leq p < \infty$, $s\geq0$, and $M,R>0$. Then there exists a constant $c_{p,s}>0$ depending on $p$ and $s$ such that
\begin{align*}
    \inf_{\hat{\mu}_N} \sup_{P\in\mathcal{P}^s_{R,M}} \mathbb E \left[\mathrm{OT}_p^p\left(\hat{\mu}_N,\mathbf{E}(P)\right)\right] \geq \frac{c_{p,s}MR^{p-s}}{\sqrt{N}}.
\end{align*}
Here, the infimum ranges over all possible estimators $\hat{\mu}_N$ that can be computed from the i.i.d.~$N$-sample $\{\mu_i\}_{i=1}^N \sim P$.
\end{theorem}

\begin{remark}[Regularity of density]\label{rem:reg_minimax_rate}
As noted by \cite{Chazal_Divol_EPD}, the density of the expected persistence diagram is typically smooth. Hence, a natural question to ask is whether the regularity of the density can be leveraged to obtain a faster minimax rate. The answer to this question is negative: in \cite[Theorem 3]{divol2021estimation} it is established that regardless of the assumed regularity on the density, the minimax rate cannot be faster than $1/\sqrt{N}$.
\end{remark}

Given the observed $N$-sample $\{\mu_i\}_{i=1}^N$, a very natural and simple estimator of $\mathbf{E}(P)$ is the \emph{empirical mean} $\Bar{\mu}_N$ defined as
\begin{align}\label{eq:emp_mean_EPD}
    \Bar{\mu}_N \coloneqq \frac{1}{N} \sum_{i=1}^N \mu_i,
\end{align}
which can be computed efficiently. Moreover, the following result states that $\Bar{\mu}_N$ achieves the minimax rate.
\begin{theorem}[Convergence rate empirical mean, Theorem 1 in \citep{divol2021estimation}]\label{thm:emp_mean_minimax}
Let $1\leq p <\infty$ and $0\leq s<p$. Given $N$ i.i.d.~samples $\{\mu_i\}_{i=1}^N \sim P$, compute $\Bar\mu_N$ according to \cref{eq:emp_mean_EPD}. Then 
\begin{align*}
    \sup_{P\in \mathcal{P}^s_{R,M}} \mathbb E \left[\mathrm{OT}_p^p(\Bar{\mu}_N,\mathbf{E}(P))\right] \leq
    c_{p,s}MR^{p-s} a_p(N)\left(\frac{1}{\sqrt{N}}+\frac{1}{N^{p-s}}\right),
\end{align*}
where $c_{p,s}>0$ is a constant depending on $p$ and $s$, and where
\begin{align*}
    a_p(N)\coloneqq \begin{cases} \log_2(N) & \text{if $p=1$,} \\
    1 & \text{otherwise.}
    \end{cases}
\end{align*}
\end{theorem}
Thus, whenever $p>1$ and $p-s\geq 1/2$, we obtain the minimax rate $1/\sqrt{N}$. In particular, this holds under the assumption that the points in the observed persistence diagrams $\{\mu_i\}_{i=1}^N$ are bounded by $M$, i.e., when we assume $s=0$.

However, in practice, if we observe $N$ persistence diagrams, the support of $\Bar{\mu}_N$ is typically very large since it corresponds to the union of the support of each diagram. This makes the estimator $\Bar{\mu}_N$ prohibitive for many applications \citep{divol2021estimation}. Additionally, the estimator $\Bar{\mu}_N$ is in general not absolutely continuous with respect to the Lebesgue measure when observing persistence diagrams which are discrete measures.

\subsection{Wavelet-Based Estimation}
To control transportation distances such as $\mathrm{OT}_p$ from above, we construct an explicit coupling (transportation map) between two measures. Indeed, recalling \cref{eq:OT_p}, we take the infimum over all couplings to compute $\mathrm{OT}_p$. Thus, using an explicit coupling yields an upper bound. A simple and natural way to construct an efficient transportation map between two measures is to use a multiscale (dyadic) partition of the underlying set \citep{Wasserstein_partitioning_arg,Wasserstein_partition_arg2}.

Consider the following multiscale partition of $\Omega_R$: for $k\in \mathbb N_0$, define
\begin{align*}
    A_k \coloneqq \left\{x\in \Omega_R \colon R\,2^{-(k+1)}<\left\lVert x-x^\perp\right\rVert_2\leq R\,2^{-k}\right\} \text{ such that } \cup_{k\in \mathbb N_0} A_k = \Omega_R.
\end{align*}
Let $J\in \mathbb N$, consider the sequence of partitions $\{\mathcal{Q}_{k,j-1}\}_{j=1}^J$ of $A_k$, where $\mathcal{Q}_{k,j-1}$ consists of squares whose side length is $R2^{-(k+1)}2^{-j+1}$ such that $\mathcal{Q}_{k,j}$ is a refinement of $\mathcal{Q}_{k,j-1}$, i.e., for every $Q\in \mathcal{Q}_{k,j}$ there exists a $Q' \in \mathcal{Q}_{k,j-1}$ satisfying $Q\subseteq Q'$. \cref{fig:multiscale_partition} illustrates this multiscale partition. 

\begin{figure}[t!]
	\centering
	\begin{subfigure}{0.4585\textwidth}
		\resizebox{1\columnwidth}{!}{
\def\R{10}
\begin{tikzpicture}[scale=1, every node/.style={scale=1.3}]
\draw ({-\R/sqrt((8)},{\R/sqrt(8)-\R/sqrt(2)}) --
({-\R/sqrt(8)-\R/sqrt(2)},{\R/sqrt(8)}) -- ({-\R/sqrt(8)},{\R/sqrt(8))+\R/sqrt(2)}) --  ({-\R/sqrt(8)+\R/sqrt(2)},{\R/sqrt(8)}) -- ({-\R/sqrt(8)},{\R/sqrt(8)-\R/sqrt(2)});
\foreach \k in {1,...,6} {
    \draw ({-\R/sqrt(8)-2^(-\k)*\R/sqrt(2)},{\R/sqrt(8)-(1-2^(-\k))*\R/sqrt(2)}) --
({-\R/sqrt(8)+(1-2^(-\k))*\R/sqrt(2)},{\R/sqrt(8)+2^(-\k)*\R/sqrt(2)});
}
\fill[color=black!10] ({-\R/sqrt(8)-\R/10},{\R/sqrt(8)-\R/sqrt(2)-\R/10}) -- ({-\R/sqrt(8)+\R/sqrt(2)+\R/10},{\R/sqrt(8)+\R/10}) --
({-\R/sqrt(8)+\R/sqrt(2)+\R/10},{\R/sqrt(8)-\R/sqrt(2)-\R/10});
\draw[very thick] ({-\R/sqrt(8)-\R/10},{\R/sqrt(8)-\R/sqrt(2)-\R/10}) -- ({-\R/sqrt(8)+\R/sqrt(2)+\R/10},{\R/sqrt(8)+\R/10});
\node at ({-\R/sqrt(\R)+\R/sqrt(2)+\R/40},{\R/sqrt(8)}) {$\partial\Omega$};
\node at ({-\R/sqrt(8)-2^(-1)*\R/sqrt(2)-\R/15},{\R/sqrt(8)+(1-2^(-1))*\R/sqrt(2)+\R/15}) {$\Omega_R$};
\node at ({-\R/sqrt(8)-(1-2^(-2))*\R/sqrt(2)-\R/15},{\R/sqrt(8)-(2^(-2))*\R/sqrt(2)-\R/15}) {$A_0$};
\node at ({-\R/sqrt(8)-(2^(-2)+2^(-3))*\R/sqrt(2)-\R/15},{\R/sqrt(8)-(1-2^(-2)-2^(-3))*\R/sqrt(2)-\R/15}) {$A_1$};
\node at ({-\R/sqrt(8)-(2^(-3)+2^(-4))*\R/sqrt(2)-\R/15},{\R/sqrt(8)-(1-2^(-3)-2^(-4))*\R/sqrt(2)-\R/15}) {$A_2$};
\node at ({-\R/sqrt(8)-(2^(-4)+2^(-5))*\R/sqrt(2)-\R/15},{\R/sqrt(8)-(1-2^(-4)-2^(-5))*\R/sqrt(2)-\R/15}) {$A_3$};
\end{tikzpicture}
}
	\end{subfigure}
	\begin{subfigure}{0.49\textwidth}
	   \definecolor{Color2}{RGB}{0,113,188}
\definecolor{Color1}{RGB}{0,155,85}
\definecolor{Color0}{named}{red}
\resizebox{1\columnwidth}{!}{
\def\R{10}
\begin{tikzpicture}[scale=1, every node/.style={scale=1.3}]
\draw ({-\R/sqrt((8)},{\R/sqrt(8)-\R/sqrt(2)}) --
({-\R/sqrt(8)-\R/sqrt(2)},{\R/sqrt(8)}) -- ({-\R/sqrt(8)},{\R/sqrt(8))+\R/sqrt(2)}) --  ({-\R/sqrt(8)+\R/sqrt(2)},{\R/sqrt(8)}) -- ({-\R/sqrt(8)},{\R/sqrt(8)-\R/sqrt(2)});
\foreach \k in {1,...,6} {
    \draw[very thick] ({-\R/sqrt(8)-2^(-\k)*\R/sqrt(2)},{\R/sqrt(8)-(1-2^(-\k))*\R/sqrt(2)}) --
({-\R/sqrt(8)+(1-2^(-\k))*\R/sqrt(2)},{\R/sqrt(8)+2^(-\k)*\R/sqrt(2)});
}
\fill[color=black!10] ({-\R/sqrt(8)-\R/10},{\R/sqrt(8)-\R/sqrt(2)-\R/10}) -- ({-\R/sqrt(8)+\R/sqrt(2)+\R/10},{\R/sqrt(8)+\R/10}) --
({-\R/sqrt(8)+\R/sqrt(2)+\R/10},{\R/sqrt(8)-\R/sqrt(2)-\R/10});
\draw[very thick] ({-\R/sqrt(8)-\R/10},{\R/sqrt(8)-\R/sqrt(2)-\R/10}) -- ({-\R/sqrt(8)+\R/sqrt(2)+\R/10},{\R/sqrt(8)+\R/10});
\node at ({-\R/sqrt(\R)+\R/sqrt(2)+\R/40},{\R/sqrt(8)}) {$\partial\Omega$};
\node at ({-\R/sqrt(8)-2^(-1)*\R/sqrt(2)-\R/15},{\R/sqrt(8)+(1-2^(-1))*\R/sqrt(2)+\R/15}) {$\Omega_R$};
\node at ({-\R/sqrt(8)-(1-2^(-2))*\R/sqrt(2)-\R/15},{\R/sqrt(8)-(2^(-2))*\R/sqrt(2)-\R/15}) {$A_0$};
\node at ({-\R/sqrt(8)-(2^(-2)+2^(-3))*\R/sqrt(2)-\R/15},{\R/sqrt(8)-(1-2^(-2)-2^(-3))*\R/sqrt(2)-\R/15}) {$A_1$};
\node at ({-\R/sqrt(8)-(2^(-3)+2^(-4))*\R/sqrt(2)-\R/15},{\R/sqrt(8)-(1-2^(-3)-2^(-4))*\R/sqrt(2)-\R/15}) {$A_2$};
\node at ({-\R/sqrt(8)-(2^(-4)+2^(-5))*\R/sqrt(2)-\R/15},{\R/sqrt(8)-(1-2^(-4)-2^(-5))*\R/sqrt(2)-\R/15}) {$A_3$};
\foreach[parse=true] \k in {0,...,3}{
    \coordinate (P1) at ({-\R/sqrt(8)-2^(-\k)*\R/sqrt(2)},{\R/sqrt(8)-(1-2^(-\k))*\R/sqrt(2)});
    \coordinate (P2) at ({-\R/sqrt(8)+(1-2^(-\k))*\R/sqrt(2)},{\R/sqrt(8)+2^(-\k)*\R/sqrt(2)});
    \coordinate (P3) at ({-\R/sqrt(8)-2^(-(\k+1))*\R/sqrt(2)},{\R/sqrt(8)-(1-2^(-(\k+1)))*\R/sqrt(2)});
    \coordinate (P4) at ({-\R/sqrt(8)+(1-2^(-(\k+1)))*\R/sqrt(2)},{\R/sqrt(8)+2^(-(\k+1))*\R/sqrt(2)});
    \foreach[parse=true] \j in {2,...,2} {
        \foreach[parse=true] \l in {1,...,2^(\j+\k+1)-1} {
            \draw[Color2] ($(P1)+({\l/sqrt(2)*\R*2^(-\k-1-\j)},{\l/sqrt(2)*\R*2^(-\k-1-\j)})$) -- ($(P3)+({\l/sqrt(2)*\R*2^(-\k-1-\j)},{\l/sqrt(2)*\R*2^(-\k-1-\j)})$);
        }
        \foreach[parse=true] \l in {1,...,2^(\j)-1} {
            \draw[Color2] ($(P1)+({\l/sqrt(2)*\R*2^(-\k-1-\j)},{-\l/sqrt(2)*\R*2^(-\k-1-\j)})$) -- ($(P2)+({\l/sqrt(2)*\R*2^(-\k-1-\j)},{-\l/sqrt(2)*\R*2^(-\k-1-\j)})$);
        }
    }
    \foreach[parse=true] \j in {1,...,1} {
        \foreach[parse=true] \l in {1,...,2^(\j+\k+1)-1} {
            \draw[Color1, thick] ($(P1)+({\l/sqrt(2)*\R*2^(-\k-1-\j)},{\l/sqrt(2)*\R*2^(-\k-1-\j)})$) -- ($(P3)+({\l/sqrt(2)*\R*2^(-\k-1-\j)},{\l/sqrt(2)*\R*2^(-\k-1-\j)})$);
        }
        \foreach[parse=true] \l in {1,...,2^(\j)-1} {
            \draw[Color1, thick] ($(P1)+({\l/sqrt(2)*\R*2^(-\k-1-\j)},{-\l/sqrt(2)*\R*2^(-\k-1-\j)})$) -- ($(P2)+({\l/sqrt(2)*\R*2^(-\k-1-\j)},{-\l/sqrt(2)*\R*2^(-\k-1-\j)})$);
        }
    }
    \foreach[parse=true] \j in {0,...,0} {
        \foreach[parse=true] \l in {1,...,2^(\j+\k+1)-1} {
            \draw[Color0, very thick] ($(P1)+({\l/sqrt(2)*\R*2^(-\k-1-\j)},{\l/sqrt(2)*\R*2^(-\k-1-\j)})$) -- ($(P3)+({\l/sqrt(2)*\R*2^(-\k-1-\j)},{\l/sqrt(2)*\R*2^(-\k-1-\j)})$);
        }
    }
}
\node[rotate=45] at ({-\R/sqrt(8)-(1-2^(-2)-1)*\R/sqrt(2)+\R/8},{\R/sqrt(8)-(2^(-2)-1)*\R/sqrt(2)+\R/8}) {\color{Color0}$\mathcal{Q}_{0,0}$, \color{Color1}$\mathcal{Q}_{0,1}$, \color{Color2}$\mathcal{Q}_{0,2}$};
\node[rotate=45] at ({-\R/sqrt(8)-(2^(-2)+2^(-3)-1)*\R/sqrt(2)+\R/8},{\R/sqrt(8)-(1-2^(-2)-2^(-3)-1)*\R/sqrt(2)+\R/8}) {\color{Color0}$\mathcal{Q}_{1,0}$, \color{Color1}$\mathcal{Q}_{1,1}$, \color{Color2}$\mathcal{Q}_{2,2}$};
\node[rotate=45] at ({-\R/sqrt(8)-(2^(-3)+2^(-4)-1)*\R/sqrt(2)+\R/8},{\R/sqrt(8)-(1-2^(-3)-2^(-4)-1)*\R/sqrt(2)+\R/8}) {\color{Color0}$\mathcal{Q}_{2,0}$, \color{Color1}$\mathcal{Q}_{2,1}$, \color{Color2}$\mathcal{Q}_{2,2}$};
\node[rotate=45] at ({-\R/sqrt(8)-(2^(-4)+2^(-5)-1)*\R/sqrt(2)+\R/8},{\R/sqrt(8)-(1-2^(-4)-2^(-5)-1)*\R/sqrt(2)+\R/8}) {\color{Color0}$\mathcal{Q}_{3,0}$, \color{Color1}$\mathcal{Q}_{3,1}$, \color{Color2}$\mathcal{Q}_{3,2}$};
\end{tikzpicture}
}
	\end{subfigure}
	\vspace{-7.5mm}
	\caption{Multiscale partition of $\Omega_R$ used to bound $\mathrm{OT}_p$ and to construct the wavelet estimator. A similar figure appeared in \citep{divol2021estimation}.}
	\label{fig:multiscale_partition}
\end{figure}
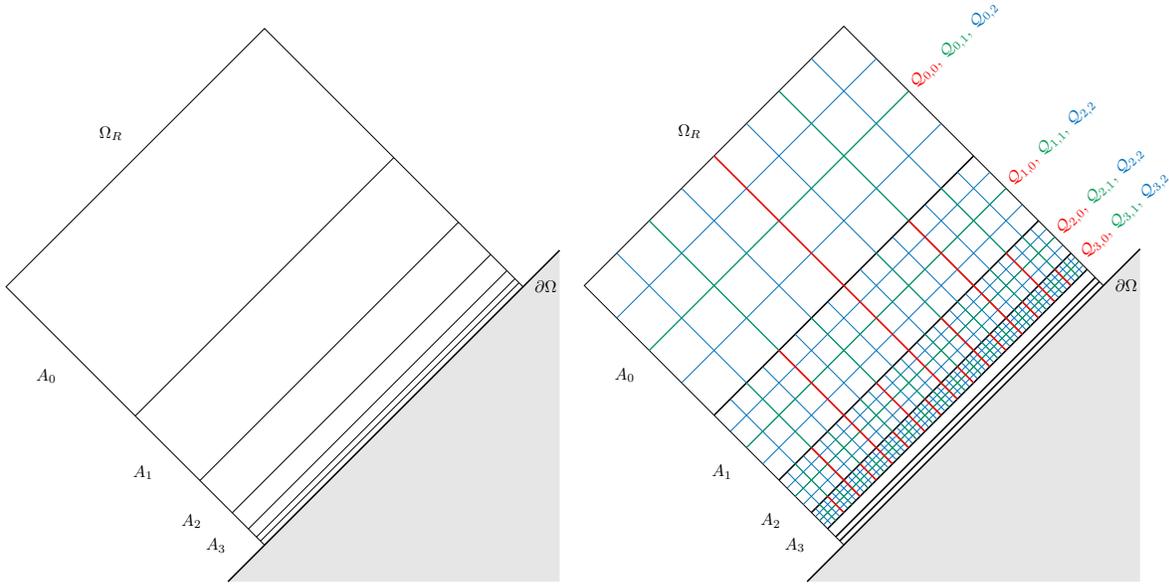

Intuitively, to construct an explicit transportation map, we recursively apply the following steps for each $k\in \mathbb N_0$: given the partition $\mathcal{Q}_{k,0}$, move mass between sets in $\mathcal{Q}_{k,0}$, and then within each set in $\mathcal{Q}_{k,0}$, using the partition $\mathcal{Q}_{k,1}$. This procedure gives the following bound. 
\allowdisplaybreaks
\begin{lemma}[Lemma 4 in \citep{divol2021estimation}]\label{lemma:OT_bound}
Let $\mu,\nu\in\mathcal{M}^p$ supported on $\Omega_R$, and let $J\in \mathbb N$. Then
\begin{align*}
    \mathrm{OT}_p^p(\mu,\nu) &\leq 2^{p/2}R^p \sum_{k\geq 0}2^{-kp}\Bigg(2^{-Jp} (\mu(A_k)\wedge\nu(A_k))+c_p\lvert \mu (A_k)-\nu(A_k)\rvert \\&\quad+\sum_{j=1}^J 2^{-jp} \sum_{Q\in \mathcal{Q}_{k,j-1}}\lvert \mu(Q) -\nu (Q) \rvert\Bigg),
\end{align*}
where $c_p \coloneqq 2^{-p/2}(1+1/(2^p-1))$.
\end{lemma}

\paragraph{Constructing the Wavelet-Based Estimator.}
We start by recalling the standard Haar system in $\mathbb R^2$. Denote by $\chi$ the characteristic function and let $\tilde \varphi\coloneqq\chi_{[0,1]^2}$. Set
\begin{align*}
    \tilde\varphi_{j,\ell}(t_1,t_2)\coloneqq2^j\tilde\varphi(2^j(t_1,t_2)+\ell), \quad (t_1,t_2)\in\mathbb R^2,\, j\in\mathbb N_0,\, \ell=(m,n)\in\mathbb Z^2.
\end{align*}
For $(t_1,t_2)\in\mathbb R^2$, define
\begin{align*}
    \tilde\psi^{a}(t_1,t_2) &\coloneqq \left(\chi_{[0,1/2)}(t_1)-\chi_{[1/2,1)}(t_1)\right)\chi_{[0,1]}(t_2), \\
    \tilde\psi^{b}(t_1,t_2) &\coloneqq \chi_{[0,1]}(t_1) \left(\chi_{[0,1/2)}(t_2)-\chi_{[1/2,1)}(t_2)\right), \\
    \tilde\psi^{c}(t_1,t_2) &\coloneqq \left(\chi_{[0,1/2)}(t_1)-\chi_{[1/2,1)}(t_1)\right)\left(\chi_{[0,1/2)}(t_2)-\chi_{[1/2,1)}(t_2)\right),
\end{align*}
and set, for $\alpha\in\{a,b,c\}$,
\begin{align*}
    \tilde\psi^\alpha_{j,\ell}(t_1,t_2)\coloneqq 2^j\tilde\psi^\alpha (2^j(t_1,t_2)+\ell), \quad (t_1,t_2)\in\mathbb R^2,\, j\in\mathbb N_0,\, \ell=(m,n)\in\mathbb Z^2.
\end{align*}
To adapt the Haar system to the geometry of our setup so that the wavelet functions are piecewise constant on the squares of our multiscale partition (see \cref{fig:multiscale_partition}), we apply a change of variables. First translate by $1/2$ in the negative $t_1$-direction, scale by a factor of $R$, and finally rotate by $+\pi/4$ to get
\begin{align}\label{eq:ch_var}
    \mathbb R^2 \xrightarrow[]{} \mathbb R^2, \quad 
    \begin{pmatrix} t_1 \\ t_2 \end{pmatrix} \mapsto \begin{pmatrix}
    u(t_1,t_2) \\ v(t_1,t_2) \end{pmatrix}\coloneqq
    \begin{pmatrix}
        \displaystyle \frac{t_1+t_2}{\sqrt{2}R}+\frac 12 \\
        \displaystyle \frac{t_2-t_1}{\sqrt{2}R}
    \end{pmatrix}.
\end{align}
Normalizing yields
\begin{align*}
    \varphi(t_1,t_2) \coloneqq R^{-1}\chi_{[0,1]^2}(u(t_1,t_2),\, v(t_1,t_2)) = R^{-1}\chi_{\Omega_R}(t_1,t_2).
\end{align*}
The scaled and translated versions of $\varphi$ then read
\begin{align*}
   \varphi_{j,\ell}(t_1,t_2) = 2^j\varphi (2^jt_1+a_{j,\ell},\, 2^jt_2+b_{j,l}), \quad j\in\mathbb N_{0},\, \ell=(m,n)\in\mathbb Z^2,  
\end{align*}
where $a_{j,\ell}\coloneqq(m-n-2^{-1}+2^{j-1})R/\sqrt{2}$ and $b_{j,\ell}\coloneqq(m+n-2^{-1}+2^{j-1})R/\sqrt{2}$.
Moreover,
\begin{align*}
    \psi^{a}(t_1,t_2) &\coloneqq \frac 1 R \left(\chi_{[0,1/2)}(u(t_1,t_2))-\chi_{[1/2,1)}(u(t_1,t_2))\right)\chi_{[0,1]}(v(t_1,t_2)), \\
    \psi^{b}(t_1,t_2) &\coloneqq \frac 1 R \chi_{[0,1]}(u(t_1,t_2)) \left(\chi_{[0,1/2)}(v(t_1,t_2))-\chi_{[1/2,1)}(v(t_1,t_2))\right), \\
    \psi^{c}(t_1,t_2) &\coloneqq \frac 1 R \left(\chi_{[0,1/2)}(u(t_1,t_2))-\chi_{[1/2,1)}(u(t_1,t_2))\right) \left(\chi_{[0,1/2)}(v(t_1,t_2))-\chi_{[1/2,1)}(v(t_1,t_2))\right),
\end{align*}
and
\begin{align}\label{eq:psi_abc}
   \psi^{\alpha}_{j,\ell}(t_1,t_2) = 2^j\psi^{\alpha} (2^jt_1+a_{j,\ell},\, 2^jt_2+b_{j,\ell}), \quad j\in\mathbb N_{0},\, \ell=(m,n)\in\mathbb Z^2,\, \alpha\in\{a,b,c\}.  
\end{align}
Note that for any $j_0 \in \mathbb N_0$, the set 
\begin{align*}
    \mathcal{V}_{j_0} \cup \bigcup_{j\geq j_0} \mathcal{W}_{j}
\end{align*}
is a complete orthonormal system in $L^2(\mathbb R^2)$. We thus have the following wavelet expansion for any function $f\in L^2(\mathbb R^2)$: 
\begin{align}\label{eq:wav_exp}
    f = \sum_{\phi \in \mathcal{V}_{j_0}} \alpha_\phi \phi + \sum_{j\geq j_0} \sum_{\psi \in \mathcal{W}_j} \beta_\psi \psi,
\end{align}
where
\begin{align*}
    \alpha_\phi \coloneqq \int_{\mathbb R^2} \phi(x)f(x)\,\mathrm{d}x, \quad \phi \in \mathcal{V}_{j_0}, \quad \text{ and } \quad
    \beta_\psi \coloneqq \int_{\mathbb R^2} \psi(x)f(x)\,\mathrm{d}x, \quad \psi \in \mathcal{W}_j, j\geq j_0.
\end{align*}
That is, the wavelet estimator of an unknown density is constructed by estimating its projection on the wavelet basis. If $f\in L^r(\mathbb R^2)$, $1\leq r<\infty$, the convergence of the infinite sum in \cref{eq:wav_exp} holds in $L^r(\mathbb R^2)$.

Fix any probability distribution $P\in \mathcal{P}_{R,M}^s$ and let $f$ be the Lebesgue density of $\mathbf{E}(P)$. Given $k\in \mathbb N_0$, denote the restriction of $f$ to the set $A_k$ by $\left.f\right|_{A_k}$, and consider the following wavelet expansion
\begin{align*}
    \left.f\right|_{A_k} = \sum_{\phi\in \mathcal{V}_{k+1}}\alpha_\phi\phi + \sum_{j\geq k+1} \sum_{\psi \in \mathcal{W}_j} \beta_{\psi} \psi,
\end{align*}
where $\alpha_\phi = \int_{\Omega_R} \phi(x) f(x)\,\mathrm{d}x$ and $\beta_\psi = \int_{\Omega_R} \psi(x) f(x)\,\mathrm{d}x$. If $f\in L^{r}(\Omega_R)$, $r\in[1,\infty)$, the convergence also holds in $L^r(\Omega_R)$. Note that since $\Omega_R$ is bounded and the functions $\phi\in\mathcal{V}_{k+1}$ and $\psi \in \mathcal{W}_j$ are compactly supported, the sums over $\mathcal{V}_{k+1}$ and $\mathcal{W}_j$ consist of finitely many terms.

\paragraph{Haar Wavelet Estimator.} Given $N$ i.i.d.~observations $\{\mu_i\}_{i=1}^N\sim P$, we construct the estimator $\hat{f}$ such that
\begin{align}\label{eq:Haar_estimator}
    \left.\hat{f}\right|_{A_k} = \sum_{\phi\in \mathcal{V}_{k+1}}\hat{\alpha}_\phi\phi + \sum_{j=k+1}^{J+K} \sum_{\psi \in \mathcal{W}_j} \hat{\beta}_{\psi} \psi, \quad k\in\{0,1,\dots,K\},
\end{align}
where $K\in\mathbb N_0, J\in\mathbb N$ are fixed, and where
\begin{align*}
    \hat{\alpha}_{\phi} = \frac 1N \sum_{i=1}^N \int_{\Omega_R} \phi\,\mathrm{d}\mu_i \quad \text{and} \quad \hat{\beta}_{\psi} = \frac 1N \sum_{i=1}^N \int_{\Omega_R} \psi\,\mathrm{d}\mu_i.
\end{align*}
Set $\left.\hat{f}\right|_{A_k}=0$ for all $k\geq K+1$. 

The function $\hat{f}$ gives the nonparametric Haar wavelet estimator. Its construction is based on a truncation of the series expansion onto the wavelet basis; the coefficients of the expansion are set through an empirical estimate. We will denote by $\hat{\mu}_\text{H}$ the measure with Lebesgue density $\hat{f}$. 
 
We have now introduced all the required quantities in order to state our first main result, namely, the minimaxity of the Haar wavelet estimator.

\begin{theorem}[Minimaxity of the Haar wavelet estimator]
\label{thm:minimax_haar}
Let $P$ be a probability distribution supported on $\mathcal{M}^s_{M,R}$ such that the Lebesgue density $f$ of $\mathbf{E}(P)$ satisfies 
\begin{align*}
    \|f\|_{L^\infty(A_k)} \leq C^2MR^{-s}2^{ks}, \quad k\in\mathbb N_0, \text{ for some $C>0$,}
\end{align*}
and let $1\leq p<\infty$, $0\leq s<p$. Then, for $\hat{\mu}_\text{H}$ with $J=K=\lceil \log_2(N) \rceil$,
\begin{align*}
    \mathbb E \left[\mathrm{OT}_p^p(\hat{\mu}_\text{H},\mathbf{E}(P))\right] \leq c_{p,s}MR^{p-s}\left(\frac{R}{N^p}+\frac{CRa_p(N)}{\sqrt{N}}+\frac{1}{N^{p-s}}\right),
\end{align*}
where $c_{p,s}$ is a constant depending only on $p$ and $s$, and where
\begin{align*}
    a_p(N)\coloneqq \begin{cases} \log_2(N) & \text{if $p=1$,} \\
    1 & \text{otherwise.}
    \end{cases}
\end{align*}
\end{theorem}
\begin{remark}[Discussion of assumption]
Note that the condition on the probability distribution $P$ holds in particular if the Lebesgue density $f$ belongs the class of $\ell$-times continuously differentiable functions, $C^\ell$, $\ell\in \mathbb N_0$. It is shown in \cite[Theorem 3.5]{Chazal_Divol_EPD} that $f \in C^{\ell}$ if the underlying dataset admits a density of class $C^{\ell}$ with respect to the Hausdorff measure. This means that whenever the data points are sampled in a smooth manner, $f$ will exhibit the same regularity. 
Also, note that the minimax rate cannot be faster than $1/\sqrt{N}$ irrespective of the regularity assumed on the EPD as previously mentioned; see \cref{rem:reg_minimax_rate}.
\end{remark}

To prove \cref{thm:minimax_haar}, we require the following property of the EPD, which is consequence of the linearity of expectation.
\begin{lemma}\label{lem:EPD}
Let $\mu \sim P$. Then, for every measurable function $g$,
\begin{align*}
    \mathbb E \left[\int_{\Omega_R} g\,\mathrm{d}\mu\right] = \int_{\Omega_R} g\,\mathrm{d}\mathbf{E}(P).
\end{align*}
\end{lemma}
\begin{proof}
First, suppose that $g$ is a simple function, i.e., $g = \sum_{k=1}^n b_k \chi_{B_k}$, where $n\in \mathbb N$, $\{b_k\}_{k=1}^n \subset \mathbb R$, and where $\{B_k\}_{k=1}^n \subset \Omega_R$ is a sequence of Borel sets. We obtain
\begin{align}
    \mathbb E \left[\int_{\Omega_R} g\,\mathrm{d}\mu\right] &= \mathbb E \left[\int_{\Omega_R} \sum_{k=1}^n b_k \chi_{B_k}\,\mathrm{d}\mu\right] \nonumber \\
    &= \sum_{k=1}^n b_k \mathbb E \left[\int_{\Omega_R} \chi_{B_k}\,\mathrm{d}\mu\right] \label{eq:epd_lin} \\
    &= \sum_{k=1}^n b_k \mathbb E \left[\mu(B_k)\right] \nonumber\\
    &= \sum_{k=1}^n b_k \mathbb E \left[\mu\right](B_k) \label{eq:epd_def} \\
    &= \sum_{k=1}^n b_k \int_{\Omega_R} \chi_{B_k}\,\mathrm{d}\mathbf{E}(P) \label{eq:epd_def2} \\
    &= \int_{\Omega_R} \sum_{k=1}^n b_k\chi_{B_k}\,\mathrm{d}\mathbf{E}(P) \nonumber \\
    &=\int_{\Omega_R} g\,\mathrm{d}\mathbf{E}(P) \nonumber,
\end{align}
where \cref{eq:epd_lin} follows from linearity; \cref{eq:epd_def} holds by \cref{def:EPD}; and in \cref{eq:epd_def2}, we utilized \cref{eq:epd_prob_dist}.

Finally, using the monotone convergence theorem, the desired result can be established for any measurable function $g$.
\end{proof}

We are now ready to prove \cref{thm:minimax_haar}. The key in this proof is first to decompose the Lebesgue density, $f$, of $\mathbf{E}(P)$ according the multiscale partition illustrated in \cref{fig:multiscale_partition} and then to bound the distance between the estimator $\hat{f}$ and $f$ on each dyadic square. 
\begin{proof}[Proof of \cref{thm:minimax_haar}]
Let the Haar wavelet expansion of the Lebesgue density of $\mathbf{E}(P)$ be given by 
\begin{align}\label{eq:Haar_expansion_true_density}
    \left.f\right|_{A_k} = \sum_{\phi\in \mathcal{V}_{k+1}}\alpha_\phi\phi + \sum_{j\geq k+1} \sum_{\psi \in \mathcal{W}_j} \beta_{\psi} \psi, \quad k\in\mathbb N_0,
\end{align}
where $\alpha_{\phi}=\int_{\Omega_R} \phi\,\mathrm{d}\mathbf{E}(P)$ and $\beta_{\psi}=\int_{\Omega_R}\psi\,\mathrm{d}\mathbf{E}(P)$.

Apply \cref{lemma:OT_bound} to the measures $\hat{\mu}_\text{H}$ and $\mathbf{E}(P)$, and split the sum over $k$ into two parts to obtain
\begin{align}\label{eq:OTp_bound_mu_H_EPD}
    \mathrm{OT}_p^p (\hat{\mu}_\text{H}, \mathbf{E}(P)) \leq S_1 + S_2,
\end{align}
where $S_1$ contains the first $K+1$ terms and $S_2$ the remaining terms:
\begin{align}
\begin{split}\label{eq:S1}
    S_1 &\coloneqq 
    2^{p/2}R^p \sum_{k=0}^{K} 2^{-kp}\Bigg(2^{-Jp} (\hat{\mu}_\text{H}(A_k)\wedge\mathbf{E}(P)(A_k))+c_p\lvert \hat{\mu}_\text{H} (A_k)-\mathbf{E}(P)(A_k)\rvert \\&\quad+\sum_{j=1}^J 2^{-jp} \sum_{Q\in \mathcal{Q}_{k,j-1}}\lvert \hat{\mu}_\text{H}(Q) -\mathbf{E}(P)(Q) \rvert\Bigg),
\end{split}
\end{align}
\begin{align}
\begin{split}\label{eq:S2}
    S_2 &\coloneqq 2^{p/2}R^p \sum_{k\geq K+1} 2^{-kp}\Bigg(2^{-Jp} (\hat{\mu}_\text{H}(A_k)\wedge\mathbf{E}(P)(A_k))+c_p\lvert \hat{\mu}_\text{H} (A_k)-\mathbf{E}(P)(A_k)\rvert \\&\quad+\sum_{j=1}^J 2^{-jp} \sum_{Q\in \mathcal{Q}_{k,j-1}}\lvert \hat{\mu}_\text{H}(Q) -\mathbf{E}(P)(Q) \rvert\Bigg).
\end{split}
\end{align}
The proof may be divided into three steps. In Step 1 and Step 2,
we will estimate $S_1$ and $S_2$ in expectation, respectively. The desired result will be established in Step 3.
\paragraph{Step 1 (Estimation of $S_1$).} 
Let $k\in\{0,1,\dots,K\}$.
First observe that for $Q\in \mathcal{Q}_{k,j-1}$ with $j\geq 1$, $R^{-1}2^{k+j}\chi_Q\in \mathcal{V}_{k+j}$. Since $\mathcal{V}_{k+j} \cup \bigcup_{j'\geq k+j} \mathcal{W}_{j'}$ forms an orthonormal system in $L^2(\mathbb R^2)$, it holds that 
\begin{align}\label{eq:ortho_Q_psi}
    \int_Q \psi \,\mathrm{d}x = \int_{\mathbb R^2} \chi_Q \psi \,\mathrm{d}x = 0, \quad \text{for all $\psi \in \bigcup_{j'\geq k+j}\mathcal{W}_{j'}$}.
\end{align}
Therefore, we have, for $Q\in \mathcal{Q}_{k,j-1}$ with $2\leq j\leq J$,
\begin{align}
    \lvert \hat{\mu}_\text{H} (Q) - \mathbf{E}(P)(Q) \rvert
    &\quad= \left\lvert \sum_{\phi\in \mathcal{V}_{k+1}}\left(\hat{\alpha}_{\phi}-\alpha_\phi\right) \int_{Q}\phi\,\mathrm{d}x + \sum_{j'=k+1}^{k+j-1} \sum_{\psi\in\mathcal{W}_{j'}} (\hat{\beta}_\psi-\beta_\psi)\int_{Q}\psi\,\mathrm{d}x \right\rvert \label{eq:mu_Q_EP_Q_1}\\
    &\quad\leq \sum_{\phi\in \mathcal{V}_{k+1}}\left\lvert\hat{\alpha}_{\phi}-\alpha_\phi\right\rvert \int_{Q}\phi\,\mathrm{d}x + \sum_{j'=k+1}^{k+j-1} \sum_{\psi\in\mathcal{W}_{j'}} \left\lvert\hat{\beta}_\psi-\beta_\psi\right\rvert\int_{Q}\lvert\psi\rvert\,\mathrm{d}x \nonumber\\
    &\quad\leq R^{-1}\lvert Q\rvert \left(\sum_{\substack{\phi\in \mathcal{V}_{k+1} \\ \mathrm{supp}(\phi)\supseteq Q}}2^{k+1}\left\lvert\hat{\alpha}_{\phi}-\alpha_\phi\right\rvert + \sum_{j'=k+1}^{k+j-1} \sum_{\substack{\psi\in\mathcal{W}_{j'}\\\mathrm{supp}(\psi)\supseteq Q}}2^{j'} \left\lvert\hat{\beta}_\psi-\beta_\psi\right\rvert\right), \label{eq:mu_Q_EP_Q_2}
\end{align}
where we used in \cref{eq:mu_Q_EP_Q_1} the wavelet expansions \cref{eq:Haar_estimator,eq:Haar_expansion_true_density} as well as the orthogonality relation \cref{eq:ortho_Q_psi}; and where \cref{eq:mu_Q_EP_Q_2} holds because $\|\phi\|_{L^\infty(\mathbb R^2)}=R^{-1}2^{k+1}$ for $\phi \in \mathcal{V}_{k+1}$, $\|\psi\|_{L^\infty(\mathbb R^2)}=R^{-1}2^{j'}$ for $\psi \in \mathcal{W}_{j'}$, and $\lvert Q \rvert = \int_Q\,\mathrm{d}x$.   
Note that, for $Q\in\mathcal{Q}_{k,0}$, $R^{-1}2^{k+j}\chi_Q\in \mathcal{V}_{k+1}$ so that, by \cref{eq:ortho_Q_psi}, $\int_Q \psi\,\mathrm{d}x=0$ for all $\psi \in \mathcal{W}_{j'}$ with $j'\geq k+1$. Thus, whenever $Q\in\mathcal{Q}_{k,0}$, it holds that 
\begin{align*}
    \lvert \hat{\mu}_\text{H} (Q) - \mathbf{E}(P)(Q) \rvert &= \left\lvert \sum_{\phi\in \mathcal{V}_{k+1}}\left(\hat{\alpha}_{\phi}-\alpha_\phi\right) \int_{Q}\phi\,\mathrm{d}x\right\rvert 
    \leq R^{-1}2^{k+1}\lvert Q\rvert\sum_{\substack{\phi\in \mathcal{V}_{k+1} \\ \mathrm{supp}(\phi)\supseteq Q}}\left\lvert\hat{\alpha}_{\phi}-\alpha_\phi\right\rvert.
\end{align*}
Next, we will derive an upper bound for the expected error of the wavelet coefficients. To this end, let $\mu\sim P$. Because $\{\mu_i\}_{i=1}^N\sim P$ are i.i.d., we obtain, using \cref{lem:EPD}, the following identity:
\begin{align*}
    \mathbb E \left[\left(\hat{\alpha}_\phi-\alpha_\phi\right)^2\right]
    &= \mathbb E \left[\left( \frac 1N \sum_{i=1}^N \int_{\Omega_R} \phi\,\mathrm{d}\mu_i -\int_{\Omega_R} \phi \,\mathrm{d}\mathbf{E}(P)\right)^2\right] \\
    &= \mathbb E \left[ \frac{1}{N^2} \sum_{i=1}^N \sum_{j=1}^N \int_{\Omega_R} \phi\,\mathrm{d}\mu_i \int_{\Omega_R} \phi\,\mathrm{d}\mu_j - \frac{2}{N} \sum_{i=1}^N \int_{\Omega_R} \phi\,\mathrm{d}\mu_i \int_{\Omega_R} \phi \,\mathrm{d}\mathbf{E}(P)+\left(\int_{\Omega_R} \phi \,\mathrm{d}\mathbf{E}(P)\right)^2 \right] \\
    &=\frac 1N \mathbb E \left[\left(\int_{\Omega_R}\phi\,\mathrm{d}\mu\right)^2\right]+\frac{N-1}{N}\left(\int_{\Omega_R} \phi\,\mathrm{d}\mathbf{E}(P)\right)^2 - \left(\int_{\Omega_R} \phi \,\mathrm{d}\mathbf{E}(P)\right)^2 \\
    &= \frac 1N \left(\mathbb E \left[\left(\int_{\Omega_R}\phi\,\mathrm{d}\mu\right)^2\right]-\left(\int_{\Omega_R} \phi\,\mathrm{d}\mathbf{E}(P)\right)^2\right).
\end{align*}
Applying Jensen's inequality yields
\begin{align}\label{eq:alpha_coeff}
    \mathbb E \left[\left\lvert\hat{\alpha}_\phi-\alpha_\phi\right\rvert\right] &\leq \sqrt{\frac{\mathrm{Var}(\int_{\Omega_R} \phi\,\mathrm{d}\mu)}{N}} \leq \sqrt{\frac{\mathbb E \left[\left(\int_{\Omega_R}\phi\,\mathrm{d}\mu\right)^2\right]}{N}}.
\end{align}
Likewise,
\begin{align}\label{eq:beta_coeff}
    \mathbb E \left[\left\lvert\hat{\beta}_\psi-\beta_\psi\right\rvert\right] &\leq \sqrt{\frac{\mathrm{Var}(\int_{\Omega_R} \psi\,\mathrm{d}\mu)}{N}} \leq \sqrt{\frac{\mathbb E \left[\left(\int_{\Omega_R}\psi\,\mathrm{d}\mu\right)^2\right]}{N}}.
\end{align}
To bound the second moment $\mathbb E \left[\left(\int_{\Omega_R}\phi\,\mathrm{d}\mu\right)^2\right]$, we may assume $\mu(\mathrm{supp}(\phi))>0$ almost surely; indeed, otherwise $\mathbb E \left[\left(\int_{\Omega_R}\phi\,\mathrm{d}\mu\right)^2\right]=0$. We then get: 
\allowdisplaybreaks
\begin{align}
    \mathbb E \left[\left(\int_{\Omega_R}\phi\,\mathrm{d}\mu\right)^2\right] &= \mathbb E \left[\left(\int_{\mathrm{supp}(\phi)}\phi\,\mathrm{d}\mu\right)^2\right] \nonumber\\
    &= \mathbb E \left[\mu(\mathrm{supp(\phi)})^2 \left(\int_{\mathrm{supp}(\phi)}\phi\,\frac{\mathrm{d}\mu}{\mu(\mathrm{supp(\phi)})}\right)^2 \right] \nonumber\\
    &\leq \mathbb E \left[\mu(\mathrm{supp(\phi)})^2\int_{\mathrm{supp}(\phi)}\phi^2\,\frac{\mathrm{d}\mu}{\mu(\mathrm{supp(\phi)})}\right] \label{eq:2_moment_phi_1}\\
    &\leq MR^{-s}2^{ks} \mathbb E\left[\int_{\mathrm{supp}(\phi)}\phi^2\, \mathrm{d}\mu\right] \label{eq:2_moment_phi_2}\\
    &= MR^{-s}2^{ks}\int_{\mathrm{supp}(\phi)}\phi^2\, \mathrm{d}\mathbf{E}(P) \label{eq:2_moment_lem}\\
    &=MR^{-s}2^{ks}\int_{\mathrm{supp}(\phi)}\phi^2 f\,\mathrm{d}x \nonumber\\
    &\leq MR^{-s}2^{ks}\|f\|_{L^\infty(A_k)} \int_{\mathrm{supp}(\phi)}\phi^2\,\mathrm{d}x \nonumber\\
    &= MR^{-s}2^{ks}\|f\|_{L^\infty(A_k)} \nonumber \\
    &\leq \left(CMR^{-s}2^{ks}\right)^2 \label{eq:2_moment_phi_3},
\end{align}
where \cref{eq:2_moment_phi_1} follows from Jensen's inequality; in \cref{eq:2_moment_phi_2} we used that $\mathrm{supp}(\phi)\subset A_k$ and the following fact stated in \cite[(A.5)]{divol2021estimation}: for any $\mu\in \mathcal{M}^s_{M,R}$,
\begin{align}\label{eq:bound_any_mu}
    \mu(B)=\int_{B} \frac{\mathrm{dist}(x,\partial\Omega)^s}{\mathrm{dist}(x,\partial\Omega)^s}\,\mathrm{d}\mu(x) \leq MR^{-s}2^{ks}, \quad \text{for all $B\subseteq A_k$;}
\end{align}
the identity \cref{eq:2_moment_lem} is due to \cref{lem:EPD}; and where \cref{eq:2_moment_phi_3} holds based on our assumption that $\|f\|_{L^\infty(A_k)}\leq C^2MR^{-s}2^{ks}$.
Likewise,
\begin{align}\label{eq:2_moment_psi}
    \mathbb E \left[\left(\int_{\Omega_R}\psi\,\mathrm{d}\mu\right)^2\right] \leq \left(CMR^{-s}2^{ks}\right)^2. 
\end{align}
Now observe that $\mathrm{card}\left(\{ \phi\in\mathcal{V}_{k+1} \colon \mathrm{supp}(\phi)\supseteq Q \}\right) =1$ for all $Q\in \mathcal{Q}_{k,j-1}$ with $j\geq 1$. Thanks to \cref{eq:alpha_coeff,eq:2_moment_phi_3}, we have
\begin{align}\label{eq:sum_alpha}
    \sum_{\substack{\phi\in \mathcal{V}_{k+1} \\ \mathrm{supp}(\phi)\supseteq Q}}2^{k+1}\mathbb E \left[\left\lvert\hat{\alpha}_{\phi}-\alpha_\phi\right\rvert\right]
    \leq \sum_{\substack{\phi\in \mathcal{V}_{k+1} \\ \mathrm{supp}(\phi)\supseteq Q}} 2^{k+1} \sqrt{\frac{\mathbb E \left[\left(\int_{\Omega_R}\phi\,\mathrm{d}\mu\right)^2\right]}{N}} 
    \leq \frac{CMR^{-s}2^{k(s+1)+1}}{\sqrt{N}}.
\end{align}
Similarly, since $\mathrm{card}\left(\{ \psi\in\mathcal{W}_{j'} \colon \mathrm{supp}(\psi)\supseteq Q \}\right)= 3$ (see \cref{eq:psi_abc}), it follows from \cref{eq:beta_coeff,eq:2_moment_psi} that
\begin{align*}
    \sum_{\substack{\psi\in\mathcal{W}_{j'} \\ \mathrm{supp}(\psi)\supseteq Q }} \mathbb E \left[\left\lvert\hat{\beta}_\psi-\beta_\psi\right\rvert\right] \leq \sum_{\substack{\psi\in\mathcal{W}_{j'} \\ \mathrm{supp}(\psi)\supseteq Q }} \sqrt{\frac{\mathbb E \left[\left(\int_{\Omega_R}\psi\,\mathrm{d}\mu\right)^2\right]}{N}} 
    \leq \frac{3CMR^{-s}2^{ks}}{\sqrt{N}}.
\end{align*}
Using $\sum_{j'=k+1}^{k+j-1} 2^{j'}\leq 2^{k+j}$, we obtain
\begin{align}\label{eq:sum_beta}
    \sum_{j'=k+1}^{k+j-1} \sum_{\substack{\psi\in\mathcal{W}_{j'}\\\mathrm{supp}(\psi)\supseteq Q }}2^{j'} \mathbb E\left[ \left\lvert\hat{\beta}_\psi-\beta_\psi\right\rvert\right] \leq \frac{3CMR^{-s}2^{k(s+1)+j}}{\sqrt{N}}.
\end{align}
Substituting \cref{eq:sum_alpha,eq:sum_beta} into \cref{eq:mu_Q_EP_Q_2} yields
\begin{align}\label{eq:mu_Q_EP_Q_est}
    \mathbb E \left[\lvert \hat{\mu}_\text{H} (Q) - \mathbf{E}(P)(Q) \rvert\right] &\leq \frac{\lvert Q\rvert }{\sqrt{N}} CMR^{-s-1}2^{k(s+1)+1}\left(1+3\cdot 2^{j-1}\right).
\end{align}
Since $\mathcal{Q}_{k,j-1}$ is a partition of $A_k$ and $|A_k|=R^2 2^{-(k+1)}$, we get, by summing \cref{eq:mu_Q_EP_Q_est} over all $Q\in \mathcal{Q}_{k,j-1}$, for $j\geq 2$, 
\begin{align}\label{eq:bound_mu_EPD_Qj2}
    \sum_{Q\in \mathcal{Q}_{k,j-1}} \mathbb E \left[\lvert \hat{\mu}_\text{H}(Q) -\mathbf{E}(P)(Q) \rvert \right] &\leq \frac{\lvert A_k\rvert }{\sqrt{N}} CMR^{-s-1}2^{k(s+1)+1}\left(1+3\cdot 2^{j-1}\right) \nonumber \\
    &=\frac{1}{\sqrt{N}} CMR^{-s+1}2^{ks}\left(1+3\cdot 2^{j-1}\right).
\end{align}
When $j=1$, 
\begin{align}\label{eq:bound_mu_EPD_Q1}
    \sum_{Q\in \mathcal{Q}_{k,j-1}} \mathbb E \left[\lvert \hat{\mu}_\text{H}(Q) -\mathbf{E}(P)(Q) \rvert \right] &\leq \frac{\lvert A_k\rvert }{\sqrt{N}} CMR^{-s-1}2^{k(s+1)+1}
    =\frac{1}{\sqrt{N}} CMR^{-s+1}2^{ks}.
\end{align}
Moreover, since 
\begin{align*}
    \left\lvert\hat{\mu}_\text{H}(A_k)-\mathbf{E}(P)(A_k)\right\rvert = \left\lvert \sum_{Q\in\mathcal{Q}_{k,0}}\hat{\mu}_\text{H}(Q)-\mathbf{E}(P)(Q) \right\rvert,
\end{align*}
it follows from \cref{eq:bound_mu_EPD_Q1} that
\begin{align}\label{eq:bound_mu_EPD_Ak}
    \mathbb E \left[\left\lvert\hat{\mu}_\text{H}(A_k)-\mathbf{E}(P)(A_k)\right\rvert \right] 
    \leq \frac{1}{\sqrt{N}} CMR^{-s+1}2^{ks}.
\end{align}
Thanks to \cref{eq:bound_any_mu}, we also have
\begin{align}\label{eq:bound_min_mu_EPD}
    \hat{\mu}_\text{H}(A_k) \wedge \mathbf{E}(P)(A_k) \leq MR^{-s}2^{ks}.
\end{align}
Using the estimates \cref{eq:bound_mu_EPD_Qj2,eq:bound_mu_EPD_Q1,eq:bound_mu_EPD_Ak,eq:bound_min_mu_EPD} we can bound $S_1$, defined in \cref{eq:S1}, in expectation:
\begin{align}
    \mathbb E \left[S_1\right] &\leq 2^{p/2}MR^{p-s+1} \sum_{k=0}^{K} 2^{-k(p-s)}\left(2^{-Jp}+c_p\frac{C}{\sqrt{N}}+\sum_{j=1}^J 2^{-jp} \left(\frac{C\left(1+3\cdot 2^{j-1}\right)}{\sqrt{N}}\right)\right) \nonumber\\
    &\leq 2^{p/2}MR^{p-s+1} \sum_{k=0}^{K} 2^{-k(p-s)}\left(2^{-Jp}+c_p\frac{C}{\sqrt{N}}+\frac{2C}{\sqrt{N}}\sum_{j=1}^J 2^{-j(p-1)}\right). \label{eq:sum_kleqK}
\end{align}

\paragraph{Step 2 (Estimation of $S_2$).} 
Let $k\geq K+1$.
Note that, by definition, $\hat{\mu}_\text{H}(B)=0$ for any $B\subseteq A_k$. 
Thanks to \cref{eq:bound_any_mu}, we get
\begin{align}
    \sum_{Q\in \mathcal{Q}_{k,j-1}}\lvert \hat{\mu}_\text{H}(Q) -\mathbf{E}(P)(Q) \rvert &= \sum_{Q\in \mathcal{Q}_{k,j-1}}\mathbf{E}(P)(Q) = \mathbf{E}(P)(A_k) \leq  MR^{-s}2^{ks}, \label{eq:est_kgeqK+1_1} \\
    \left\lvert\hat{\mu}_\text{H}(A_k)-\mathbf{E}(P)(A_k)\right\rvert &\leq MR^{-s}2^{ks}, \label{eq:est_kgeqK+1_2}\\
    \intertext{and }
    \hat{\mu}_\text{H}(A_k) \wedge \mathbf{E}(P)(A_k) &= 0. \label{eq:est_kgeqK+1_3}
\end{align}
Hence, substituting the estimates \cref{eq:est_kgeqK+1_1,eq:est_kgeqK+1_2,eq:est_kgeqK+1_3} into \cref{eq:S2} gives
\begin{align}
    S_2 &\leq 2^{p/2}R^p \sum_{k\geq K+1}2^{-kp}\left(c_p MR^{-s}2^{ks}+\sum_{j=1}^J 2^{-jp} MR^{-s}2^{ks}\right) \nonumber\\
    &\leq 2^{p/2}\left(c_p + 1/(2^p-1)\right)MR^{p-s} \sum_{k\geq K+1}2^{-k(p-s)} \nonumber\\
    &= {c}'_pMR^{p-s} \sum_{k\geq K+1}2^{-k(p-s)} \nonumber\\
    &= {c}'_pMR^{p-s} \frac{2^{-(p-s)K}}{2^{p-s}-1}, \label{eq:sum_kgeqK+1_bound}
\end{align}
where $c'_p\coloneqq2^{p/2}\left(c_p + 1/(2^p-1)\right)=(2^{p/2}+2^p)/(2^p-1)$.

\paragraph{Step 3.} 
Finally, combining \cref{eq:OTp_bound_mu_H_EPD,eq:sum_kleqK,eq:sum_kgeqK+1_bound} results in  
\begin{align*}
    &\mathbb E \left[\mathrm{OT}_p^p(\hat{\mu}_\text{H},\mathbf{E}(P))\right] \\ &\quad\leq 2^{p/2}MR^{p-s+1} \sum_{k=0}^{K} 2^{-k(p-s)}\left(2^{-Jp}+c_p\frac{C}{\sqrt{N}}+\frac{2C}{\sqrt{N}}\sum_{j=1}^J 2^{-j(p-1)}\right) + c'_pMR^{p-s} \frac{2^{-(p-s)K}}{2^{p-s}-1}.
\end{align*}
Setting $K=\lceil \log_2(N) \rceil$ and $J=\lceil \log_2(N) \rceil$ yields the desired estimate
\begin{align*}
    \mathbb E \left[\mathrm{OT}_p^p(\hat{\mu}_\text{H},\mathbf{E}(P))\right] \leq c_{p,s}MR^{p-s}\left(\frac{R}{N^p}+\frac{CRa_p(N)}{\sqrt{N}}+\frac{1}{N^{p-s}}\right),
\end{align*}
where $c_{p,s}$ is a constant depending only on $p$ and $s$.
\end{proof}

\Cref{thm:minimax_haar} also provides an estimate of the optimal rate of convergence of the Haar wavelet estimator, defined up to constants and variable factors depending on $s$ and $p$ which stay bounded.

\paragraph{Thresholding Haar Wavelet Estimator.}
We now turn to our practical contribution, which entails a class of nonparametric density estimators and uses a Haar wavelet basis that discards small coefficients by introducing a threshold below which coefficients are set to zero (hard thresholding). The threshold is set following the procedure proposed in \citep{Donoho_wavelet_est}, where it was designed to attain exactly or approximately optimal convergence rates.

Consider the following hard thresholding technique: fix $\tau >0$ and let
\begin{align*}
    \tau_j \coloneqq  CMR^{-s}\tau 2^{j/p}\frac{j}{\sqrt{N}}.
\end{align*}
For any $\psi \in \mathcal{W}_j$, consider
\begin{align*}
    \tilde{\beta}_\psi \coloneqq \begin{cases}\hat{\beta}_\psi &\text{if $\lvert \hat{\beta}_\psi\rvert>\tau_j$,} \\ 0 &\text{otherwise}.
    \end{cases}
\end{align*}
Construct the estimator $\tilde{f}$ according to
\begin{align*}
    \left.\tilde{f}\right|_{A_k} = \sum_{\phi\in \mathcal{V}_{k+1}}\hat{\alpha}_\phi\phi + \sum_{j=k+1}^{J+K} \sum_{\psi \in \mathcal{W}_j} \tilde{\beta}_{\psi} \psi, \quad k\in\{0,1,\dots,K\},
\end{align*}
and $\left.\tilde{f}\right|_{A_k}=0$, $k\geq K+1$. The function $\tilde{f}$ represents the thresholding Haar wavelet estimator as first proposed by \cite{Donoho_wavelet_est}. Let $\tilde\mu_{\text{H}}$ be the measure whose Lebesgue density is $\tilde{f}$.

\begin{corollary}[Near-optimal minimaxity of the thresholding Haar wavelet estimator]
\label{cor:thr}
Under the assumptions of \cref{thm:minimax_haar}, we have that
\begin{align*}
    \mathbb E \left[\mathrm{OT}_p^p(\tilde{\mu}_\mathrm{H},\mathbf{E}(P))\right] \leq \tilde c_{p,s}MR^{p-s}\left(\frac{R}{N^p}+\frac{CRa_p(N)}{\sqrt{N}}+\frac{CR\tau\log_2(N)}{\sqrt{N}}+\frac{1}{N^{p-s}}\right),
\end{align*}
where $\tilde c_{p,s}$ is a constant depending only on $p$ and $s$, and where
\begin{align*}
    a_p(N)\coloneqq \begin{cases} \log_2(N) & \text{if $p=1$,} \\
    1 & \text{otherwise.}
    \end{cases}
\end{align*}
\end{corollary}

\begin{proof}
Let $\psi\in \mathcal{W}_{j'}$. Using the triangle inequality, we compute
\begin{align*}
    \mathbb E \left[\lvert \tilde\beta_\psi -\beta_\psi \rvert\right] &\leq  \mathbb E \left[\lvert \tilde \beta_\psi - \hat\beta_\psi \rvert + \lvert \hat\beta_\psi - \beta_\psi \rvert \right] \\
    &= \mathbb E \left[\lvert\hat{\beta}_\psi\rvert\mathbbm{1}_{\{\lvert \hat{\beta}_\psi\rvert\leq\tau_{j'}\}}+\lvert \beta_\psi - \hat\beta_\psi \rvert \right] \\
    &\leq \tau_{j'}+ \mathbb E \left[\lvert \beta_\psi - \hat\beta_\psi \rvert\right].
\end{align*}
Thus, 
\begin{align}
    \sum_{j'=k+1}^{k+j-1} \sum_{\substack{\psi\in\mathcal{W}_{j'}\\\mathrm{supp}(\psi)\supseteq Q }}2^{j'} \mathbb E\left[ \left\lvert\tilde{\beta}_\psi-\beta_\psi\right\rvert\right] 
    &\leq \sum_{j'=k+1}^{k+j-1} \sum_{\substack{\psi\in\mathcal{W}_{j'}\\\mathrm{supp}(\psi)\supseteq Q }}2^{j'} \left(\mathbb E \left[\lvert \beta_\psi - \hat\beta_\psi \rvert\right] + \tau_{j'}\right) \nonumber\\
    &\leq \frac{3CMR^{-s}2^{k(s+1)+j}}{\sqrt{N}} + 3 \sum_{j'=k+1}^{k+j-1} 2^{j'}\tau_{j'},\label{eq:sum_beta_thr}
\end{align}
where we used \cref{eq:sum_beta} and $\mathrm{card}\left(\{ \psi\in\mathcal{W}_{j'} \colon \mathrm{supp}(\psi)\supseteq Q \} \right)= 3$ in the last inequality.
Computing
\begin{align*}
    \sum_{j'=k+1}^{k+j-1} 2^{j'}\tau_{j'} &\leq CMR^{-s}\tau 2^{(k+j)/p} \frac{j+k}{\sqrt{N}}
\end{align*}
and substituting this into \cref{eq:sum_beta_thr}, we get
\begin{align}\label{eq:sum_beta_thr2}
    \sum_{j'=k+1}^{k+j-1} \sum_{\substack{\psi\in\mathcal{W}_{j'}\\\mathrm{supp}(\psi)\supseteq Q }}2^{j'} \mathbb E\left[ \left\lvert\tilde{\beta}_\psi-\beta_\psi\right\rvert\right] 
    &\leq \frac{3CMR^{-s}\left(2^{k(s+1)+j}+\tau(k+j)2^{(j+k)/p}\right)}{\sqrt{N}}.
\end{align}
The remainder of this proof follows exactly the same steps as in the proof of \cref{thm:minimax_haar}, where we replace $\hat \mu_\mathrm{H}$ by $\tilde \mu_\mathrm{H}$ and where \cref{eq:sum_beta_thr2} plays the role of \cref{eq:sum_beta}.
\end{proof}
Thus, the thresholding Haar wavelet estimator achieves near-optimal rates. Because of the truncation, only a few wavelet coefficients are nonzero and contribute to the estimator $\tilde{f}$. Therefore, $\tilde{f}$ can be considered as a \emph{sparse} representation of the expected persistence diagram $\mathbf{E}(P)$. This is in contrast to the empirical mean $\bar{\mu}_N$ whose support tends to be very large. The sparsity of the thresholding Haar wavelet estimator has been leveraged for data compression in several contexts \citep{krommweh2010,fryzlewicz2016} and may be more practical in TDA applications where the use of $\bar{\mu}_N$ may be restrictive due to  its large support.

\section{Numerical Experiments}
\label{sec:simulations}
We will now verify our theoretical results established in \cref{sec:minimax} by numerical experiments and illustrate the practical use of the presented estimation techniques in a classification problem in the context of dynamical systems. Here, persistent homology is computed using the \texttt{GUDHI} library \citep{gudhi:urm} as well as the \texttt{giotto-tda} library \citep{tauzin2020giottotda} for \texttt{Python}. The implementation of the optimal partial transport metric is based on tools from the \texttt{POT} library \citep{flamary2021pot} for \texttt{Python}. 

\subsection{Convergence Rates: Three Examples}
To investigate the convergence rates of the aforementioned estimators, we need to compute $\mathrm{OT}_p(\hat{\mu}_N,\mathbf{E}(P))$, where $\hat{\mu}_N$ is an arbitrary estimator. Note, however, that the (true) expected persistence diagram cannot be computed explicitly in general. A closed-form solution of the expected persistence diagram is known for only a few examples \citep{divol2021estimation}. To deal with this issue, we approximate $\mathrm{OT}_p(\hat{\mu}_N,\mathbf{E}(P))$ by $\mathrm{OT}_p(\hat{\mu}_N,\Bar{\mu}_M)$, where $\Bar{\mu}_M = \frac{1}{M} \sum_{i=1}^M \mu_{i}$, and where $M$ is significantly larger than $N$, assuming we have access to $M$ samples.

We study three datasets, the torus and double torus, as well as a clustered process.

\paragraph{Torus.} 
We observe the point clouds $\{X_i\}_{i=1}^M$, $M=10,000$, consisting of $n=1000$ points which are sampled from a torus with inner radius $0.5$, outer radius $2$ perturbed by additive Gaussian noise with variance $0.5$. We compute the persistence diagrams $\{\mu_i\}_{i=1}^M$. 

The convergence rates of the Haar wavelet estimator $\hat \mu_{\mathrm{H}}$ are shown in \cref{fig:minimax_Haar_torus} for $p~\in \{1,2,3,4\}$, where $\bar{\mu}_N$ is based on $N$ samples from $\{\mu_i\}_{i=1}^M$, and where $N$ takes values between $10$ and $500$. We obtain the convergence rates predicted by \cref{thm:minimax_haar}, namely $1/\sqrt{N}$ for $p>1$ and $\log_2(N)/\sqrt{N}$ for $p=1$. 

\begin{figure}[t!]
	\centering
        \resizebox{0.8\textwidth}{!}{
	\input{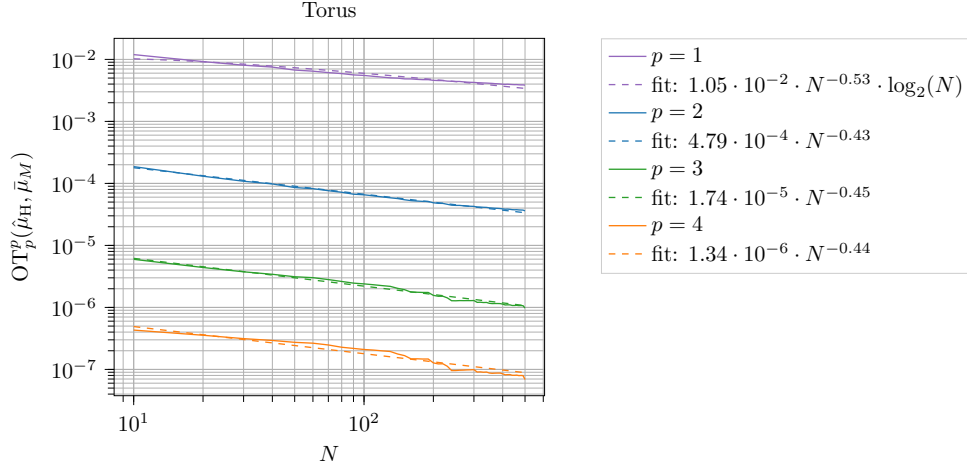}
        }
	\caption{Convergence rates of the Haar wavelet estimator $\hat{\mu}_\mathrm{H}$, averaged over 10 iterations, for different sample sizes $N$ ranging from $10$ to $500$. We take $N\mapsto a\cdot N^{-b}\cdot \log_2(N)$ as a model function (dashed line) for $p=1$ and $N\mapsto a\cdot N^{-b}$ for $p>1$, where $a,b\in \mathbb R$ are the parameters to fit.}
\label{fig:minimax_Haar_torus}
\end{figure}
\begin{figure}[t!]
	\centering
        \resizebox{0.8\textwidth}{!}{
	\input{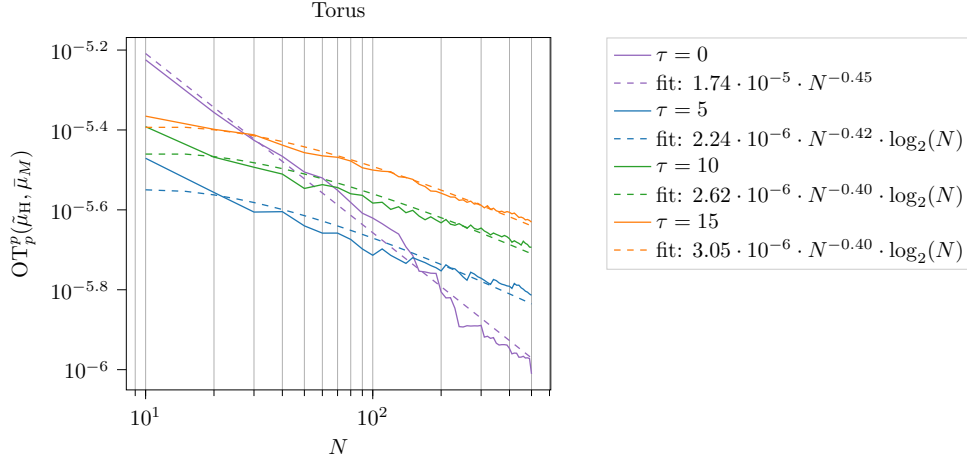}
        }
	\caption{Convergence rates of the thresholding Haar wavelet estimator $\tilde{\mu}_\mathrm{H}$, averaged over $10$ iterations, where $N\in \{10,\dots,500\}$ and $p=3$. We consider the model function (dashed line) $N\mapsto a\cdot N^{-b}\cdot \log_2(N)$ whenever $\tau >0$ with $a,b\in \mathbb R$ being the model parameters.}
	\label{fig:minimax_Haar_thr_torus}
\end{figure}

Evaluating the convergence behavior of the thresholding Haar wavelet estimator for $\tau \in \{5,10,15\}$ (see \cref{fig:minimax_Haar_thr_torus}), we deduce that the obtained rates are consistent with \cref{cor:thr}; namely, we get a rate of $\log_2(N)/\sqrt{N}$ for each $\tau \in \{5,10,15\}$. From \cref{tab:wav_coeff} it becomes apparent that thresholding significantly reduces the number of nonzero wavelet coefficients. In particular, we have a compression rate of up to $99\%$ compared to the Haar wavelet estimator ($\tau =0$). In addition, we can observe that describing the estimator $\Bar{\mu}_N$ requires substantially more information (as quantified by the cardinality of its support) than the thresholding Haar wavelet estimators, making the latter more convenient in applications where, for example, fast subsequent computations are required or storage is limited.

\begin{table}[t!]
    \centering
    \begin{tabular}{|l||c|c|c|}
        \hline
         & $N=100$ & $N=300$ & $N=500$ \\ \hline\hline
         Support of $\bar \mu_N$ & $456,900$ & $1,370,700$ & $2,284,500$ \\ \hline
        \hline
        Nonzero $\Tilde{\mu}_{\mathrm{H}}$ coefficients, $\tau =0$ & $272,206$ & $370,999$ & $420,456$ \\ \hline
         Nonzero $\Tilde{\mu}_{\mathrm{H}}$ coefficients, $\tau =5$ & $20,000$ (c.r.$\,\approx 93\%$) & $21,101$  (c.r.$\,\approx 94\%$) & $21,930$ (c.r.$\,\approx 95\%$)\\ \hline
         Nonzero $\Tilde{\mu}_{\mathrm{H}}$ coefficients, $\tau = 10$ & $7,602$ (c.r.$\,\approx 97\%$) & $8,280$ (c.r.$\,\approx 98\%$) & $9,056$ (c.r.$\,\approx 98\%$) \\ \hline
         Nonzero $\Tilde{\mu}_{\mathrm{H}}$ coefficients, $\tau = 15$ & $4,668$ (c.r.$\,\approx 98\%$) & $5,571$ (c.r.$\,\approx 98\%$) & $6,154$ (c.r.$\,\approx 99\%$) \\ \hline
    \end{tabular}
    \caption{Cardinality of the support set of the empirical mean $\Bar{\mu}_N$ and the number of nonzero coefficients of $\Tilde{\mu}_{\mathrm{H}}$, $\tau \in \{0,5,10,15\}$, averaged over $10$ runs, for different sample sizes $N$. The corresponding compression rates (c.r.) from thresholding ($\tau > 0$) compared to the case $\tau=0$ are provided in brackets.}
    \label{tab:wav_coeff}
\end{table}

Notice that there is a trade-off between sparsity and convergence rate of the (thresholding) Haar wavelet estimator: a very sparse estimator entails an extra factor of $\log_2(N)$ in terms of convergence rate (see the $\tau>0$ compared to $\tau=0$ curves in \cref{fig:minimax_Haar_thr_torus}). Also, note that the error $\mathrm{OT}_p^p(\tilde{\mu}_{\mathrm{H}},\Bar{\mu}_M)$ increases as $\tau$ gets larger, i.e., with a more stringent thresholding of the wavelets coefficients, the approximation becomes increasingly coarse, compromising accuracy.

In \cref{fig:heatmaps}, we see that for low $\tau$, the estimated density contains fine details, due to the presence of many small coefficients. By increasing $\tau$, these fine details vanish, and we obtain a rougher density estimator.

We have analyzed the computational cost of the Haar wavelet estimator as a function of $N$, and  \cref{fig:runtimes} shows that the runtime scales as $\mathcal{O}(N\log_2(N))$. This is consistent with what we would also expect in theory. Indeed, according to \cref{eq:Haar_estimator}, the Haar wavelet estimator requires computing at most $\mathcal{O}(J+K)=\mathcal{O}(\log_2(N))$ wavelet coefficients, and computing a wavelet coefficient has complexity of order $\mathcal{O}(N)$. Thus, compared to the empirical mean $\bar{\mu}_N = \frac 1N \sum_{i=1}^N \mu_i$, which clearly exhibits a computational complexity of $\mathcal{O}(N)$, the Haar wavelet estimator is computationally more expensive, but only by a logarithmic factor.

\begin{figure}[t!]
	\begin{subfigure}[b]{0.49\textwidth}
		\centering
		\includegraphics[scale=0.82,left]{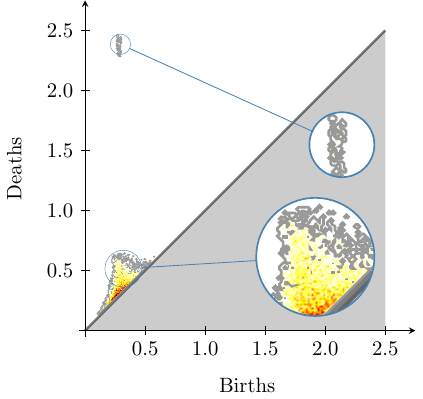}
		\caption{$\tau=0$}
		\label{fig:heatmap_torus}
	\end{subfigure}
	\begin{subfigure}[b]{0.49\textwidth}
		\centering
		\includegraphics[scale=0.82,left]{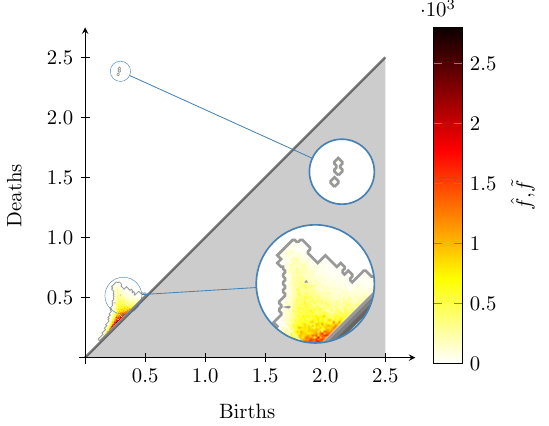}
		\caption{$\tau=5$}
		\label{fig:heatmap_torus_thr}
	\end{subfigure}
	\caption{\subref{fig:heatmap_torus} The density of the Haar wavelet estimator, denoted $\hat{f}$, based on persistence diagrams of samples of a torus. \subref{fig:heatmap_torus_thr} The density of a thresholding Haar wavelet estimator, denoted $\tilde{f}$. Here, we used grey contour lines to indicate the support of the density functions.}
	\label{fig:heatmaps}
\end{figure}
\begin{figure}[t!]
	\centering
        \resizebox{0.5\textwidth}{!}{
	\begin{tikzpicture}[scale=1, every node/.style={scale=1}]
\definecolor{darkgray176}{RGB}{176,176,176}
\definecolor{steelblue31119180}{RGB}{31,119,180}
\definecolor{lightgray204}{RGB}{204,204,204}
\begin{axis}[
legend cell align={left},
legend style={
  fill opacity=0.8,
  draw opacity=1,
  text opacity=1,
  at={(0.01,0.99)},
  anchor=north west,
  draw=lightgray204
},
grid,
tick align=outside,
tick pos=left,
x grid style={darkgray176},
xlabel={$N$},
xmin=50, xmax=500,
xminorgrids,
xtick = {50, 100,...,500},
ytick = {1,5,10,15},
xtick style={color=black},
y grid style={darkgray176},
ylabel={Computing time (normalized)},
ymin=1, ymax=19,
yminorgrids,
ytick style={color=black}
]
\addplot [semithick, steelblue31119180]
table {%
50 1
100 2.12589513970362
150 3.54275821589728
200 5.32036486721268
250 7.21106312248331
300 8.98528465524642
350 11.2837067177833
400 13.3211702148311
450 15.8980089976685
500 18.1279210315161
};
\addlegendentry{$\hat{\mu}_{\mathrm{H}}$}
\addplot [semithick, steelblue31119180, dashed]
table {%
50 1.09467530750156
100 2.57726812058552
150 4.20627797208675
200 5.93037125233581
250 7.72516792428375
300 9.57630846092067
350 11.4743045677825
400 13.4124125270012
450 15.3855900650064
500 17.3899233764794
};
\addlegendentry{fit: $3.88\cdot 10^{-3}\cdot N\cdot \log_2(N)$}
\end{axis}
\end{tikzpicture}
        }
	\caption{Computing time of Haar wavelet estimator for sample sizes $N \in \{50,\dots,500\}$, where we normalized to the computing time corresponding to $N=50$ ($\approx 200$ seconds on one CPU core).}
	\label{fig:runtimes}
\end{figure}
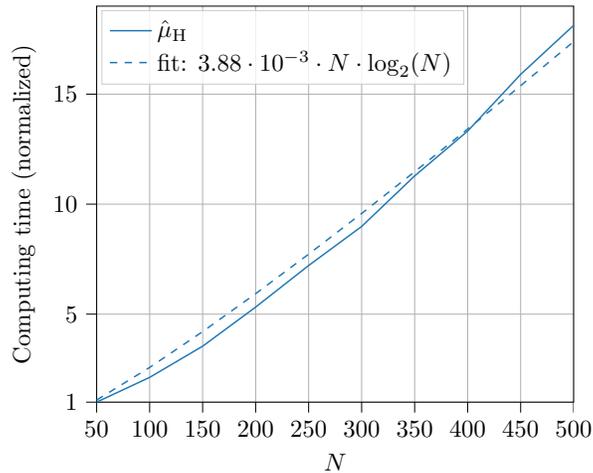

\paragraph{Double torus.} We sample $n=1000$ points uniformly at random from a double torus with inner and outer radii of $0.5$ and $2$, respectively, and add Gaussian noise with variance $0.5$ to generate the point clouds $\{X_i\}_{i=1}^M$ and the persistence diagrams $\{\mu_i\}_{i=1}^M$, where $M=10,000$. As in the previous example, the convergence rates of the Haar wavelet estimator (see \cref{fig:minimax_Haar_2torus}) coincide with the rates predicted by \cref{thm:minimax_haar}. In \cref{fig:minimax_Haar_thr_2torus}, the convergence rates of the thresholding Haar wavelet estimator are shown, corresponding to the rates from \cref{cor:thr}. 

\begin{figure}[t!]
	\centering
        \resizebox{0.8\textwidth}{!}{
	\input{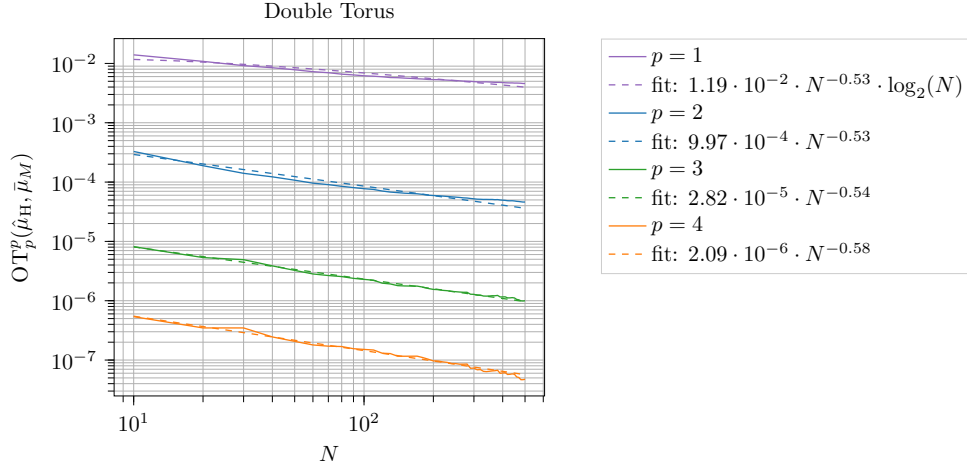}
        }
	\caption{Convergence rates of the Haar wavelet estimator $\hat{\mu}_\mathrm{H}$, averaged over $10$ iterations, for different sample sizes $N$ ranging from $10$ to $500$.}
	\label{fig:minimax_Haar_2torus}
\end{figure}
\begin{figure}[t!]
	\centering
        \resizebox{0.8\textwidth}{!}{%
	\input{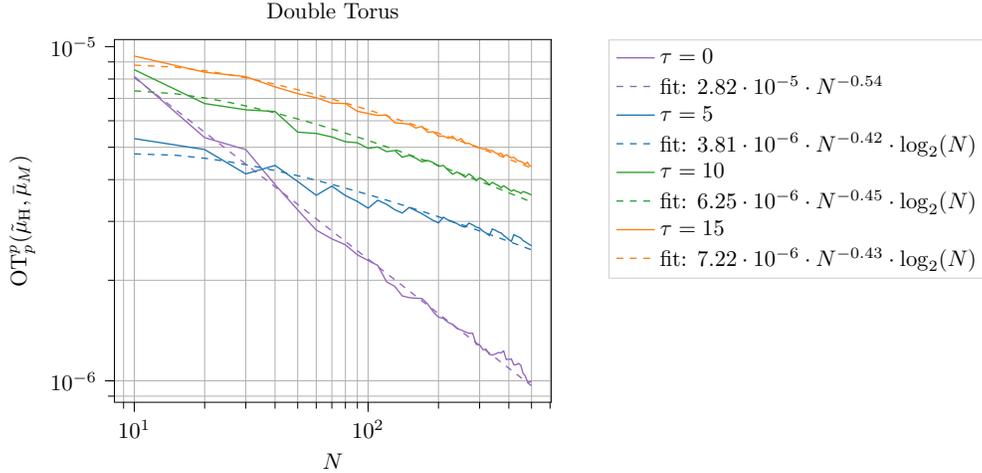}%
        }
	\caption{Convergence rates of the thresholding Haar wavelet estimator $\tilde{\mu}_\mathrm{H}$, averaged over $10$ iterations, where $N\in \{10,\dots,500\}$ and $p=3$.}
	\label{fig:minimax_Haar_thr_2torus}
\end{figure}
\paragraph{Clustered process.} Consider the following clustered process consisting of $n=1000$ points. We first select uniformly at random $n_{\mathrm{c}}=4$ points from the square $[0,1]^2$, denoted $\{c_i\}_{i=1}^{n_c}$, which serve as the clusters' centers. For each center $c_i$, $i\in \{1,\dots,n_{\mathrm{c}}\}$, we generate $n/n_{\mathrm{c}} = 250$ points by sampling from a Gaussian distribution with mean $c_i$ and variance $0.1$. Following this procedure, we create the point clouds $\{X_i\}_{i=1}^M$ and compute the corresponding persistence diagrams $\{\mu_i\}_{i=1}^M$, $M=10,000$. Also in this example of a multimodal distribution, the convergence rates of the Haar wavelet estimator and the thresholding Haar wavelet estimator, provided respectively in \cref{fig:minimax_Haar_clust_proc,fig:minimax_Haar_thr_clust_proc}, verify the results from \cref{thm:minimax_haar,cor:thr}. 
\begin{figure}[t!]
	\centering
        \resizebox{0.8\textwidth}{!}{
	\input{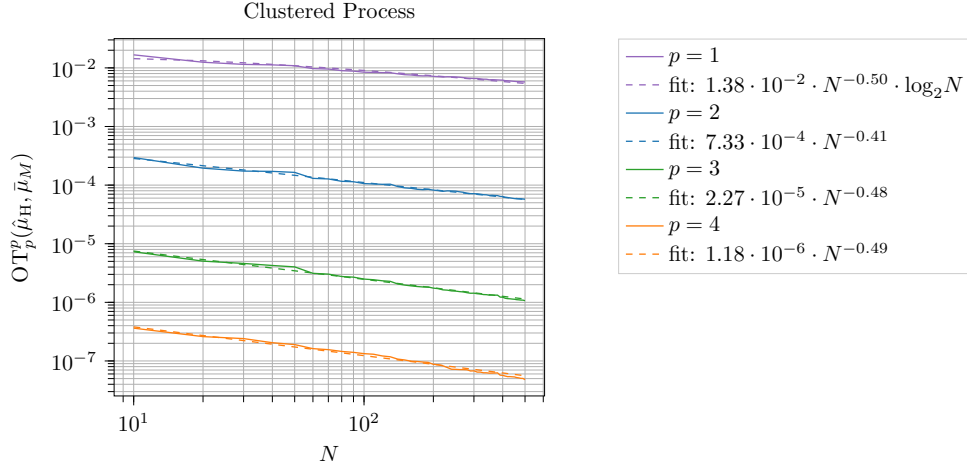}
        }
	\caption{Convergence rates of the Haar wavelet estimator $\hat{\mu}_\mathrm{H}$, averaged over $10$ iterations, for $N \in \{10,\dots,500\}$.}
	\label{fig:minimax_Haar_clust_proc}
\end{figure}
\begin{figure}[htbp]
	\centering
        \resizebox{0.8\textwidth}{!}{%
	\input{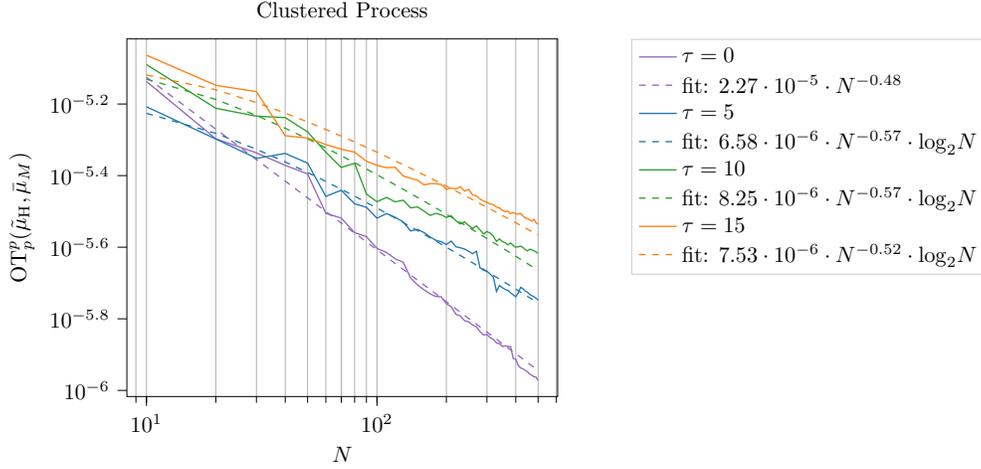}%
        }
	\caption{Convergence rates of the thresholding Haar wavelet estimator $\tilde{\mu}_\mathrm{H}$, averaged over $10$ iterations, for $N\in \{10,\dots,500\}$, where $p=3$.}
	\label{fig:minimax_Haar_thr_clust_proc}
\end{figure}

\subsection{An Application: Classification of a Dynamical System}

We now demonstrate the utility of our proposed wavelet-based estimators by applying it to a core task in machine learning, namely, classification. 
We consider a discrete-time dynamical system previously studied topologically by \citep{conti2022topological}.  The dynamical system models fluid flow as a linked twisted map and Poincar\'{e} section, in particular, which is the discretization of a continuous dynamical system given by following the location path of a particle at discrete time intervals.  This linked twisted map is given by the following system of equations:
\begin{align}\label{eq:dyn_sys}
\begin{split}
    x_{k+1} &= \left(x_k + r y_{k}(1-y_{k})\right)\bmod{1}, \\
    y_{k+1} &= \left(y_k + r x_{k+1}(1-x_{k+1})\right)\bmod{1},
\end{split}
\end{align}
where the time index $k \in \mathbb N_0$, the initial conditions $(x_0,y_0) \in (0,1)^2$ and $r>0$. The orbits of this dynamical system, $\{(x_k,y_k) \colon k \in \mathbb N_0\}$, are used to model fluid flow. 
As shown by \cite{conti2022topological}, the shape of the orbit depends on the parameter $r$ but not on the initial condition $(x_0,y_0) \in (0,1)^2$ in general; see \cref{fig:dyn_sys}. A machine learning task relevant to the study of such dynamical systems is to classify the value of the parameter $r$ into the classes, based on the observed orbit $\{(x_k,y_k) \colon k \in \mathbb N_0\}$. We take the classes to be specified by five labels $\mathcal{R} \coloneqq\{2,\, 3.5,\, 4,\, 4.1,\, 4.3\}$ as in \citep{conti2022topological}. To address this problem, we construct a dataset as follows.

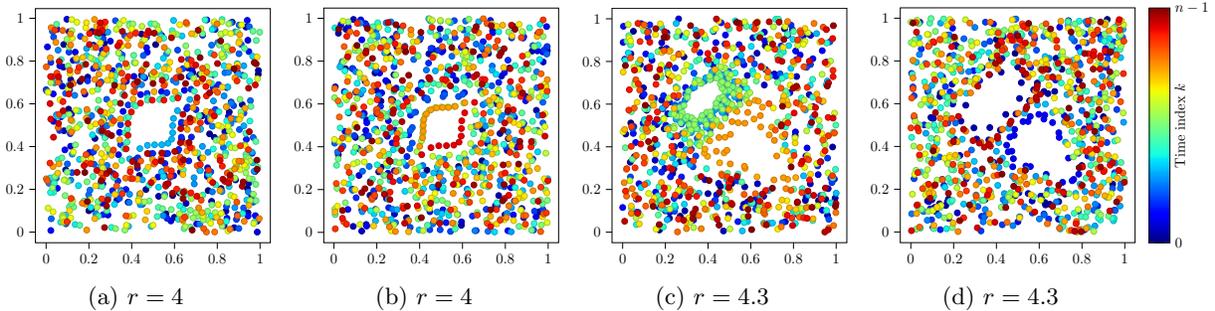
\begin{figure}[t!]
	\centering
        \begin{subfigure}{0.225\textwidth}
            \resizebox{\textwidth}{!}{
            \includegraphics[]{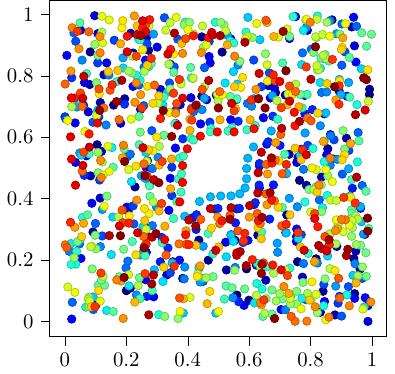}}
            \caption{$r=4$}
            \label{fig:ds_4_1}
        \end{subfigure}
        \begin{subfigure}{0.225\textwidth}
            \resizebox{\textwidth}{!}{
            \includegraphics[]{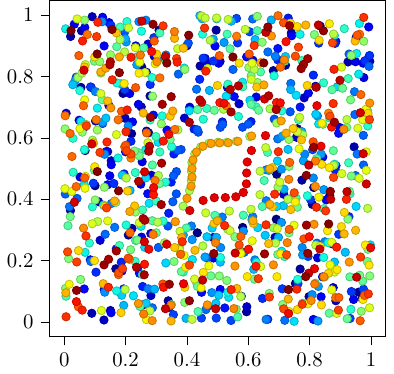}}
            \caption{$r = 4$}
            \label{fig:ds_4_2}
        \end{subfigure}
        \begin{subfigure}{0.225\textwidth}
            \resizebox{\textwidth}{!}{
            \includegraphics[]{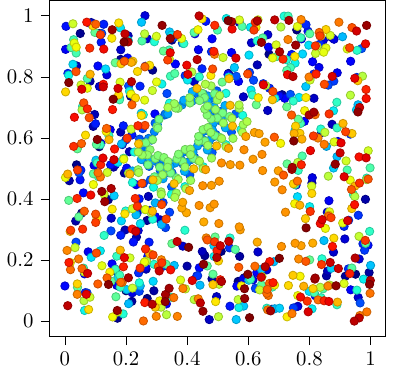}}
            \caption{$r=4.3$}
            \label{fig:ds_43_1}
        \end{subfigure}
        \begin{subfigure}{0.225\textwidth}
            \resizebox{\textwidth}{!}{
            \includegraphics[]{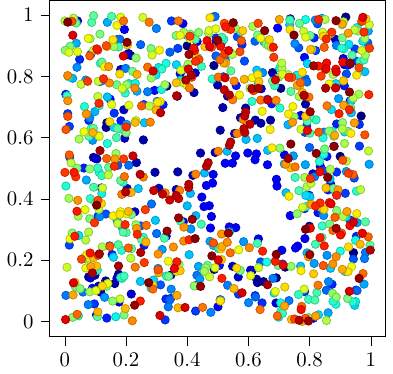}}
            \caption{$r = 4.3$}
            \label{fig:ds_43_2}
        \end{subfigure}
        \begin{subfigure}{0.07\textwidth}
            \begin{tikzpicture}[scale=0.5483, every node/.style={scale=1}]
\begin{axis}[
    hide axis,
    colorbar,
    colorbar style={
    at={(0.25,1.8)},
    ytick={0,999},
    yticklabels={$0$,$n-1$},
    ylabel={Time index $k$},
    ylabel style={at={(1,0.5)}},},
    colormap={mymap}{[1pt]
  rgb(0pt)=(0,0,0.5);
  rgb(22pt)=(0,0,1);
  rgb(25pt)=(0,0,1);
  rgb(68pt)=(0,0.86,1);
  rgb(70pt)=(0,0.9,0.967741935483871);
  rgb(75pt)=(0.0806451612903226,1,0.887096774193548);
  rgb(128pt)=(0.935483870967742,1,0.0322580645161291);
  rgb(130pt)=(0.967741935483871,0.962962962962963,0);
  rgb(132pt)=(1,0.925925925925926,0);
  rgb(178pt)=(1,0.0740740740740741,0);
  rgb(182pt)=(0.909090909090909,0,0);
  rgb(200pt)=(0.5,0,0)
},
point meta max=999,
point meta min=0,]
    \addplot [draw=none] coordinates {(0,0)};
\end{axis}
\end{tikzpicture}
        \end{subfigure}
	\caption{Examples of orbits $\{(x_k,y_k) \colon 0 \leq k \leq n-1\}$ of the dynamical system \cref{eq:dyn_sys}, consisting of $n=1000$ points in $[0,1)^2$. The orbits in \protect\subref{fig:ds_4_1} and \protect\subref{fig:ds_4_2} correspond to the same parameter, $r=4$, but they have different initial conditions, namely, $(0.2,0.1)$ and $(0.4,0.9)$, respectively. The initial conditions in \protect\subref{fig:ds_43_1} and \protect\subref{fig:ds_43_2} are also $(0.2,0.1)$ and $(0.4,0.9)$, respectively.}
        \label{fig:dyn_sys}
\end{figure}
For each class (label) $r \in \mathcal{R}$, $N=100$ orbits $\{X^r_i\}_{i=1}^N$ of the form $\{(x_k,y_k) \colon 0 \leq k \leq n-1\}$, each consisting of $n=1000$ points, are generated. From the orbits $\{X^r_i\}_{i=1}^N$, which can be viewed as point clouds in $[0,1)^2$, the persistence diagrams $\{\mu^r_i\}_{i=1}^N$ are computed. Based on the persistence diagrams $\{\mu^r_i\}_{i=1}^N$, 
we compute the normalized Haar density estimator $\hat{\mu}^r_{\mathrm{H}}$, for each class $r \in \mathcal{R}$. To successfully perform classification, it essential that the estimated densities are sufficiently distinguishable, i.e., that $\hat{\mu}^r_{\mathrm{H}}$ and $\hat{\mu}^{r'}_{\mathrm{H}}$ are dissimilar for all distinct $r,r' \in \mathcal{R}$. Since these probability measures admit Lebesgue densities, we can quantify their (dis)similarity using the Hellinger distance, which is a metric between probability measures.
\begin{definition}[Hellinger distance] 
Let $\nu_1$ and $\nu_2$ be two probability measures on $\Omega$ with Lebesgue densities $g_1$ and $g_2$, respectively. The Hellinger distance between $\nu_1$ and $\nu_2$ is defined as $\mathrm{H}(\nu_1,\nu_2) \coloneqq \sqrt{\mathrm{H}^2(\nu_1,\nu_2)}$, where
    \begin{align*}
        \mathrm{H}^2(\nu_1,\nu_2) \coloneqq \frac 12 \int_{\Omega} \left(\sqrt{g_1(x)}-\sqrt{g_2(x)}\right)^2\,\mathrm{d}x.
    \end{align*}
\end{definition}

The Hellinger distance $\mathrm{H}(\nu_1,\nu_2)$ between two probability measures $\nu_1$ and $\nu_2$ admitting a Lebesgue density is symmetric, nonnegative, and it is zero if and only if $\nu_1=\nu_2$. Moreover, an application of the Cauchy--Schwarz inequality shows that 0 $\leq \mathrm{H}(\nu_1,\nu_2) \leq 1$.

Computing the pairwise distances $\mathrm{H}(\hat{\mu}^r_{\mathrm{H}},\hat{\mu}^{r'}_{\mathrm{H}})$, $r,r' \in \mathcal{R}$, results in the $5 \times 5$ dissimilarity matrix presented in \cref{fig:hell_0}. A value of $1$ indicates complete dissimilarity between $\hat{\mu}^r_{\mathrm{H}}$ and $\hat{\mu}^{r'}_{\mathrm{H}}$, whereas $0$ implies  equality.
Performing the same computations with the thresholding Haar wavelet estimator $\tilde{\mu}^r_{\mathrm{H}}$ yields the results shown in \cref{fig:hell_1}.  
Note that $\mathrm{H}(\tilde{\mu}^r_{\mathrm{H}},\tilde{\mu}^{r'}_{\mathrm{H}})$ is systematically smaller than $\mathrm{H}(\hat{\mu}^r_{\mathrm{H}},\hat{\mu}^{r'}_{\mathrm{H}})$, for all distinct $r,r' \in \mathcal{R}$, which results from the loss of accuracy when enforcing sparsity through thresholding. 
Nevertheless, we obtain that both $\mathrm{H}(\hat{\mu}^r_{\mathrm{H}},\hat{\mu}^{r'}_{\mathrm{H}})$ and $\mathrm{H}(\tilde{\mu}^r_{\mathrm{H}},\tilde{\mu}^{r'}_{\mathrm{H}})$ are close to $1$ for all distinct $r,r' \in \mathcal{R}$. 
Thus, we can conclude that both the Haar wavelet estimator and the thresholding Haar wavelet estimator capture the topological differences in the trajectories generated by the different values of $r$ and can be used to discriminate them.

Considering the case where only $N=10$ orbits (and hence persistence diagrams) can be accessed, we infer from \cref{fig:hell10_0,fig:hell10_1} that these results still hold, showing that the classification performance of Haar wavelet-based estimators is robust to a decrease in sample size $N$.

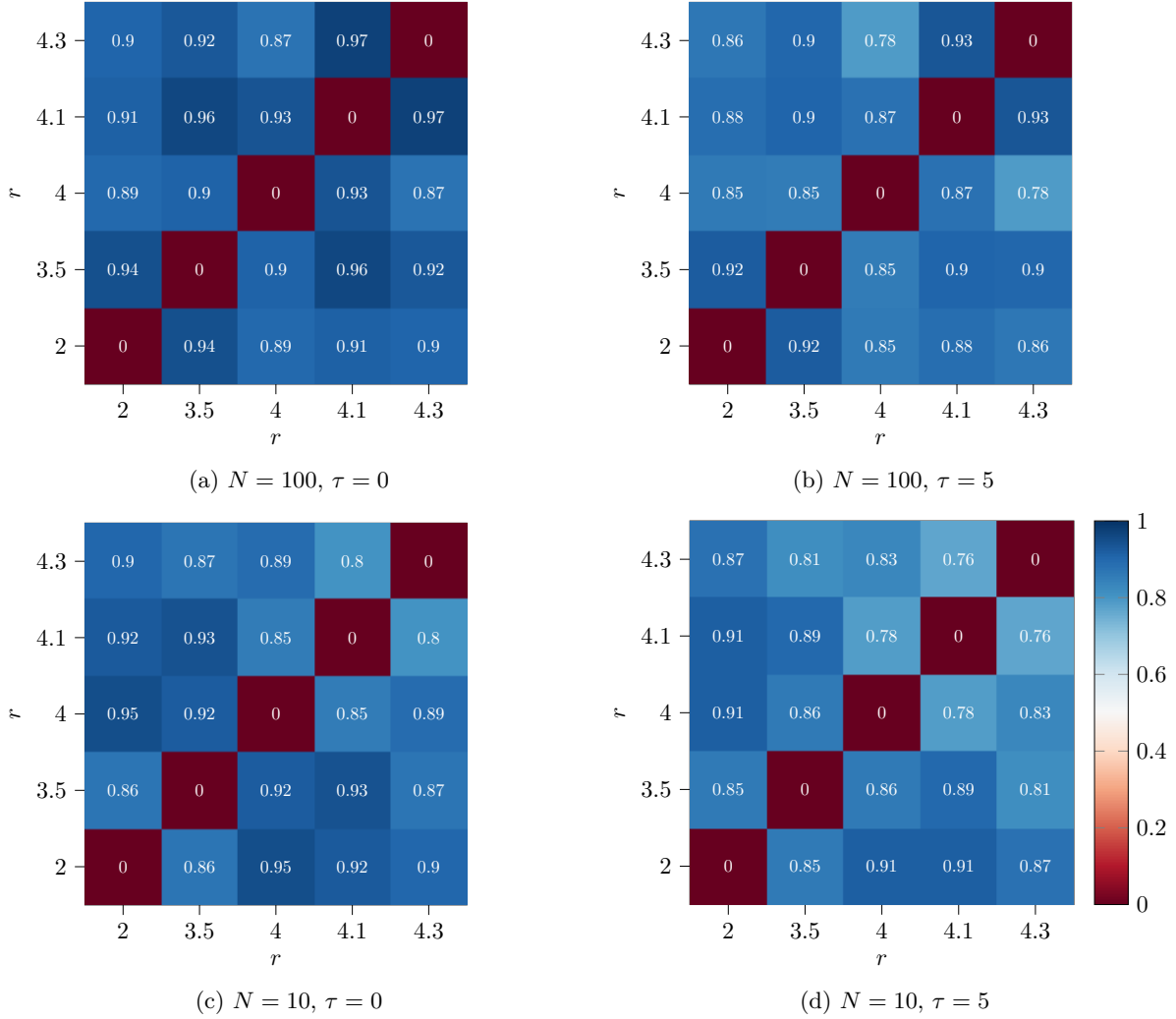
\begin{figure}[t!]
    \centering
    \begin{tabular}[c]{ll}
        \begin{subfigure}{0.47\textwidth}
            \resizebox{0.825\textwidth}{!}{%
\begin{tikzpicture}

\definecolor{darkgray176}{RGB}{176,176,176}

\begin{axis}[
tick align=outside,
tick pos=left,
unit vector ratio*=1 1 1,
x grid style={darkgray176},
xlabel={$r$},
xmin=0, xmax=5,
xtick style={color=black},
xtick={0.5,1.5,2.5,3.5,4.5},
xticklabels={2,3.5,4,4.1,4.3},
y grid style={darkgray176},
ylabel={$r$},
ymin=0, ymax=5,
ytick style={color=black},
ytick={0.5,1.5,2.5,3.5,4.5},
yticklabels={2,3.5,4,4.1,4.3}
]
\addplot graphics [includegraphics cmd=\pgfimage,xmin=0, xmax=5, ymin=0, ymax=5] {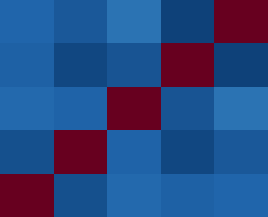};
\draw (axis cs:0.5,0.5) node[
  scale=0.75,
  text=white,
  rotate=0.0
]{0};
\draw (axis cs:1.5,0.5) node[
  scale=0.75,
  text=white,
  rotate=0.0
]{0.94};
\draw (axis cs:2.5,0.5) node[
  scale=0.75,
  text=white,
  rotate=0.0
]{0.89};
\draw (axis cs:3.5,0.5) node[
  scale=0.75,
  text=white,
  rotate=0.0
]{0.91};
\draw (axis cs:4.5,0.5) node[
  scale=0.75,
  text=white,
  rotate=0.0
]{0.9};
\draw (axis cs:0.5,1.5) node[
  scale=0.75,
  text=white,
  rotate=0.0
]{0.94};
\draw (axis cs:1.5,1.5) node[
  scale=0.75,
  text=white,
  rotate=0.0
]{0};
\draw (axis cs:2.5,1.5) node[
  scale=0.75,
  text=white,
  rotate=0.0
]{0.9};
\draw (axis cs:3.5,1.5) node[
  scale=0.75,
  text=white,
  rotate=0.0
]{0.96};
\draw (axis cs:4.5,1.5) node[
  scale=0.75,
  text=white,
  rotate=0.0
]{0.92};
\draw (axis cs:0.5,2.5) node[
  scale=0.75,
  text=white,
  rotate=0.0
]{0.89};
\draw (axis cs:1.5,2.5) node[
  scale=0.75,
  text=white,
  rotate=0.0
]{0.9};
\draw (axis cs:2.5,2.5) node[
  scale=0.75,
  text=white,
  rotate=0.0
]{0};
\draw (axis cs:3.5,2.5) node[
  scale=0.75,
  text=white,
  rotate=0.0
]{0.93};
\draw (axis cs:4.5,2.5) node[
  scale=0.75,
  text=white,
  rotate=0.0
]{0.87};
\draw (axis cs:0.5,3.5) node[
  scale=0.75,
  text=white,
  rotate=0.0
]{0.91};
\draw (axis cs:1.5,3.5) node[
  scale=0.75,
  text=white,
  rotate=0.0
]{0.96};
\draw (axis cs:2.5,3.5) node[
  scale=0.75,
  text=white,
  rotate=0.0
]{0.93};
\draw (axis cs:3.5,3.5) node[
  scale=0.75,
  text=white,
  rotate=0.0
]{0};
\draw (axis cs:4.5,3.5) node[
  scale=0.75,
  text=white,
  rotate=0.0
]{0.97};
\draw (axis cs:0.5,4.5) node[
  scale=0.75,
  text=white,
  rotate=0.0
]{0.9};
\draw (axis cs:1.5,4.5) node[
  scale=0.75,
  text=white,
  rotate=0.0
]{0.92};
\draw (axis cs:2.5,4.5) node[
  scale=0.75,
  text=white,
  rotate=0.0
]{0.87};
\draw (axis cs:3.5,4.5) node[
  scale=0.75,
  text=white,
  rotate=0.0
]{0.97};
\draw (axis cs:4.5,4.5) node[
  scale=0.75,
  text=white,
  rotate=0.0
]{0};
\end{axis}

\end{tikzpicture}%
            }
            \caption{$N=100$, $\tau=0$}
            \label{fig:hell_0}
        \end{subfigure} &
        \begin{subfigure}{0.47\textwidth}
            \resizebox{0.825\textwidth}{!}{%
	    \begin{tikzpicture}

\definecolor{darkgray176}{RGB}{176,176,176}

\begin{axis}[
colorbar style={ylabel={}},
colormap={mymap}{[1pt]
  rgb(0pt)=(0.403921568627451,0,0.12156862745098);
  rgb(1pt)=(0.698039215686274,0.0941176470588235,0.168627450980392);
  rgb(2pt)=(0.83921568627451,0.376470588235294,0.301960784313725);
  rgb(3pt)=(0.956862745098039,0.647058823529412,0.509803921568627);
  rgb(4pt)=(0.992156862745098,0.858823529411765,0.780392156862745);
  rgb(5pt)=(0.968627450980392,0.968627450980392,0.968627450980392);
  rgb(6pt)=(0.819607843137255,0.898039215686275,0.941176470588235);
  rgb(7pt)=(0.572549019607843,0.772549019607843,0.870588235294118);
  rgb(8pt)=(0.262745098039216,0.576470588235294,0.764705882352941);
  rgb(9pt)=(0.129411764705882,0.4,0.674509803921569);
  rgb(10pt)=(0.0196078431372549,0.188235294117647,0.380392156862745)
},
point meta max=1,
point meta min=0,
tick align=outside,
tick pos=left,
unit vector ratio*=1 1 1,
x grid style={darkgray176},
xlabel={$r$},
xmin=0, xmax=5,
xtick style={color=black},
xtick={0.5,1.5,2.5,3.5,4.5},
xticklabels={2,3.5,4,4.1,4.3},
y grid style={darkgray176},
ylabel={$r$},
ymin=0, ymax=5,
ytick style={color=black},
ytick={0.5,1.5,2.5,3.5,4.5},
yticklabels={2,3.5,4,4.1,4.3}
]
\addplot graphics [includegraphics cmd=\pgfimage,xmin=0, xmax=5, ymin=0, ymax=5] {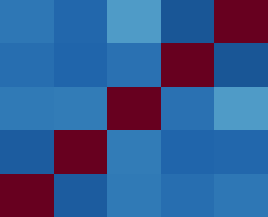};
\draw (axis cs:0.5,0.5) node[
  scale=0.75,
  text=white,
  rotate=0.0
]{0};
\draw (axis cs:1.5,0.5) node[
  scale=0.75,
  text=white,
  rotate=0.0
]{0.92};
\draw (axis cs:2.5,0.5) node[
  scale=0.75,
  text=white,
  rotate=0.0
]{0.85};
\draw (axis cs:3.5,0.5) node[
  scale=0.75,
  text=white,
  rotate=0.0
]{0.88};
\draw (axis cs:4.5,0.5) node[
  scale=0.75,
  text=white,
  rotate=0.0
]{0.86};
\draw (axis cs:0.5,1.5) node[
  scale=0.75,
  text=white,
  rotate=0.0
]{0.92};
\draw (axis cs:1.5,1.5) node[
  scale=0.75,
  text=white,
  rotate=0.0
]{0};
\draw (axis cs:2.5,1.5) node[
  scale=0.75,
  text=white,
  rotate=0.0
]{0.85};
\draw (axis cs:3.5,1.5) node[
  scale=0.75,
  text=white,
  rotate=0.0
]{0.9};
\draw (axis cs:4.5,1.5) node[
  scale=0.75,
  text=white,
  rotate=0.0
]{0.9};
\draw (axis cs:0.5,2.5) node[
  scale=0.75,
  text=white,
  rotate=0.0
]{0.85};
\draw (axis cs:1.5,2.5) node[
  scale=0.75,
  text=white,
  rotate=0.0
]{0.85};
\draw (axis cs:2.5,2.5) node[
  scale=0.75,
  text=white,
  rotate=0.0
]{0};
\draw (axis cs:3.5,2.5) node[
  scale=0.75,
  text=white,
  rotate=0.0
]{0.87};
\draw (axis cs:4.5,2.5) node[
  scale=0.75,
  text=white,
  rotate=0.0
]{0.78};
\draw (axis cs:0.5,3.5) node[
  scale=0.75,
  text=white,
  rotate=0.0
]{0.88};
\draw (axis cs:1.5,3.5) node[
  scale=0.75,
  text=white,
  rotate=0.0
]{0.9};
\draw (axis cs:2.5,3.5) node[
  scale=0.75,
  text=white,
  rotate=0.0
]{0.87};
\draw (axis cs:3.5,3.5) node[
  scale=0.75,
  text=white,
  rotate=0.0
]{0};
\draw (axis cs:4.5,3.5) node[
  scale=0.75,
  text=white,
  rotate=0.0
]{0.93};
\draw (axis cs:0.5,4.5) node[
  scale=0.75,
  text=white,
  rotate=0.0
]{0.86};
\draw (axis cs:1.5,4.5) node[
  scale=0.75,
  text=white,
  rotate=0.0
]{0.9};
\draw (axis cs:2.5,4.5) node[
  scale=0.75,
  text=white,
  rotate=0.0
]{0.78};
\draw (axis cs:3.5,4.5) node[
  scale=0.75,
  text=white,
  rotate=0.0
]{0.93};
\draw (axis cs:4.5,4.5) node[
  scale=0.75,
  text=white,
  rotate=0.0
]{0};
\end{axis}

\end{tikzpicture}%
            }
            \caption{$N=100$, $\tau = 5$}
            \label{fig:hell_1}
        \end{subfigure} \\
        \begin{subfigure}{0.47\textwidth}
            \resizebox{0.825\textwidth}{!}{%
\begin{tikzpicture}

\definecolor{darkgray176}{RGB}{176,176,176}

\begin{axis}[
unit vector ratio*=1 1 1,
colorbar style={ylabel={}},
colormap={mymap}{[1pt]
  rgb(0pt)=(0.403921568627451,0,0.12156862745098);
  rgb(1pt)=(0.698039215686274,0.0941176470588235,0.168627450980392);
  rgb(2pt)=(0.83921568627451,0.376470588235294,0.301960784313725);
  rgb(3pt)=(0.956862745098039,0.647058823529412,0.509803921568627);
  rgb(4pt)=(0.992156862745098,0.858823529411765,0.780392156862745);
  rgb(5pt)=(0.968627450980392,0.968627450980392,0.968627450980392);
  rgb(6pt)=(0.819607843137255,0.898039215686275,0.941176470588235);
  rgb(7pt)=(0.572549019607843,0.772549019607843,0.870588235294118);
  rgb(8pt)=(0.262745098039216,0.576470588235294,0.764705882352941);
  rgb(9pt)=(0.129411764705882,0.4,0.674509803921569);
  rgb(10pt)=(0.0196078431372549,0.188235294117647,0.380392156862745)
},
point meta max=1,
point meta min=0,
tick align=outside,
tick pos=left,
x grid style={darkgray176},
xlabel={$r$},
xmin=0, xmax=5,
xtick style={color=black},
xtick={0.5,1.5,2.5,3.5,4.5},
xticklabels={2,3.5,4,4.1,4.3},
y grid style={darkgray176},
ylabel={$r$},
ymin=0, ymax=5,
ytick style={color=black},
ytick={0.5,1.5,2.5,3.5,4.5},
yticklabels={2,3.5,4,4.1,4.3}
]
\addplot graphics [includegraphics cmd=\pgfimage,xmin=0, xmax=5, ymin=0, ymax=5] {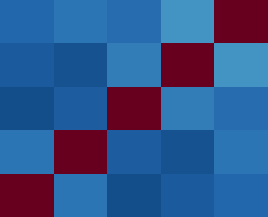};
\draw (axis cs:0.5,0.5) node[
  scale=0.75,
  text=white,
  rotate=0.0
]{0};
\draw (axis cs:1.5,0.5) node[
  scale=0.75,
  text=white,
  rotate=0.0
]{0.86};
\draw (axis cs:2.5,0.5) node[
  scale=0.75,
  text=white,
  rotate=0.0
]{0.95};
\draw (axis cs:3.5,0.5) node[
  scale=0.75,
  text=white,
  rotate=0.0
]{0.92};
\draw (axis cs:4.5,0.5) node[
  scale=0.75,
  text=white,
  rotate=0.0
]{0.9};
\draw (axis cs:0.5,1.5) node[
  scale=0.75,
  text=white,
  rotate=0.0
]{0.86};
\draw (axis cs:1.5,1.5) node[
  scale=0.75,
  text=white,
  rotate=0.0
]{0};
\draw (axis cs:2.5,1.5) node[
  scale=0.75,
  text=white,
  rotate=0.0
]{0.92};
\draw (axis cs:3.5,1.5) node[
  scale=0.75,
  text=white,
  rotate=0.0
]{0.93};
\draw (axis cs:4.5,1.5) node[
  scale=0.75,
  text=white,
  rotate=0.0
]{0.87};
\draw (axis cs:0.5,2.5) node[
  scale=0.75,
  text=white,
  rotate=0.0
]{0.95};
\draw (axis cs:1.5,2.5) node[
  scale=0.75,
  text=white,
  rotate=0.0
]{0.92};
\draw (axis cs:2.5,2.5) node[
  scale=0.75,
  text=white,
  rotate=0.0
]{0};
\draw (axis cs:3.5,2.5) node[
  scale=0.75,
  text=white,
  rotate=0.0
]{0.85};
\draw (axis cs:4.5,2.5) node[
  scale=0.75,
  text=white,
  rotate=0.0
]{0.89};
\draw (axis cs:0.5,3.5) node[
  scale=0.75,
  text=white,
  rotate=0.0
]{0.92};
\draw (axis cs:1.5,3.5) node[
  scale=0.75,
  text=white,
  rotate=0.0
]{0.93};
\draw (axis cs:2.5,3.5) node[
  scale=0.75,
  text=white,
  rotate=0.0
]{0.85};
\draw (axis cs:3.5,3.5) node[
  scale=0.75,
  text=white,
  rotate=0.0
]{0};
\draw (axis cs:4.5,3.5) node[
  scale=0.75,
  text=white,
  rotate=0.0
]{0.8};
\draw (axis cs:0.5,4.5) node[
  scale=0.75,
  text=white,
  rotate=0.0
]{0.9};
\draw (axis cs:1.5,4.5) node[
  scale=0.75,
  text=white,
  rotate=0.0
]{0.87};
\draw (axis cs:2.5,4.5) node[
  scale=0.75,
  text=white,
  rotate=0.0
]{0.89};
\draw (axis cs:3.5,4.5) node[
  scale=0.75,
  text=white,
  rotate=0.0
]{0.8};
\draw (axis cs:4.5,4.5) node[
  scale=0.75,
  text=white,
  rotate=0.0
]{0};
\end{axis}

\end{tikzpicture}%
            }
            \caption{$N=10$, $\tau=0$}
            \label{fig:hell10_0}
        \end{subfigure} &
        \begin{subfigure}{0.472\textwidth}
            \resizebox{\textwidth}{!}{%
\begin{tikzpicture}

\definecolor{darkgray176}{RGB}{176,176,176}

\begin{axis}[
unit vector ratio*=1 1 1,
colorbar,
colorbar style={ylabel={}},
colormap={mymap}{[1pt]
  rgb(0pt)=(0.403921568627451,0,0.12156862745098);
  rgb(1pt)=(0.698039215686274,0.0941176470588235,0.168627450980392);
  rgb(2pt)=(0.83921568627451,0.376470588235294,0.301960784313725);
  rgb(3pt)=(0.956862745098039,0.647058823529412,0.509803921568627);
  rgb(4pt)=(0.992156862745098,0.858823529411765,0.780392156862745);
  rgb(5pt)=(0.968627450980392,0.968627450980392,0.968627450980392);
  rgb(6pt)=(0.819607843137255,0.898039215686275,0.941176470588235);
  rgb(7pt)=(0.572549019607843,0.772549019607843,0.870588235294118);
  rgb(8pt)=(0.262745098039216,0.576470588235294,0.764705882352941);
  rgb(9pt)=(0.129411764705882,0.4,0.674509803921569);
  rgb(10pt)=(0.0196078431372549,0.188235294117647,0.380392156862745)
},
point meta max=1,
point meta min=0,
tick align=outside,
tick pos=left,
x grid style={darkgray176},
xlabel={$r$},
xmin=0, xmax=5,
xtick style={color=black},
xtick={0.5,1.5,2.5,3.5,4.5},
xticklabels={2,3.5,4,4.1,4.3},
y grid style={darkgray176},
ylabel={$r$},
ymin=0, ymax=5,
ytick style={color=black},
ytick={0.5,1.5,2.5,3.5,4.5},
yticklabels={2,3.5,4,4.1,4.3}
]
\addplot graphics [includegraphics cmd=\pgfimage,xmin=0, xmax=5, ymin=0, ymax=5] {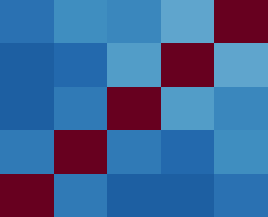};
\draw (axis cs:0.5,0.5) node[
  scale=0.75,
  text=white,
  rotate=0.0
]{0};
\draw (axis cs:1.5,0.5) node[
  scale=0.75,
  text=white,
  rotate=0.0
]{0.85};
\draw (axis cs:2.5,0.5) node[
  scale=0.75,
  text=white,
  rotate=0.0
]{0.91};
\draw (axis cs:3.5,0.5) node[
  scale=0.75,
  text=white,
  rotate=0.0
]{0.91};
\draw (axis cs:4.5,0.5) node[
  scale=0.75,
  text=white,
  rotate=0.0
]{0.87};
\draw (axis cs:0.5,1.5) node[
  scale=0.75,
  text=white,
  rotate=0.0
]{0.85};
\draw (axis cs:1.5,1.5) node[
  scale=0.75,
  text=white,
  rotate=0.0
]{0};
\draw (axis cs:2.5,1.5) node[
  scale=0.75,
  text=white,
  rotate=0.0
]{0.86};
\draw (axis cs:3.5,1.5) node[
  scale=0.75,
  text=white,
  rotate=0.0
]{0.89};
\draw (axis cs:4.5,1.5) node[
  scale=0.75,
  text=white,
  rotate=0.0
]{0.81};
\draw (axis cs:0.5,2.5) node[
  scale=0.75,
  text=white,
  rotate=0.0
]{0.91};
\draw (axis cs:1.5,2.5) node[
  scale=0.75,
  text=white,
  rotate=0.0
]{0.86};
\draw (axis cs:2.5,2.5) node[
  scale=0.75,
  text=white,
  rotate=0.0
]{0};
\draw (axis cs:3.5,2.5) node[
  scale=0.75,
  text=white,
  rotate=0.0
]{0.78};
\draw (axis cs:4.5,2.5) node[
  scale=0.75,
  text=white,
  rotate=0.0
]{0.83};
\draw (axis cs:0.5,3.5) node[
  scale=0.75,
  text=white,
  rotate=0.0
]{0.91};
\draw (axis cs:1.5,3.5) node[
  scale=0.75,
  text=white,
  rotate=0.0
]{0.89};
\draw (axis cs:2.5,3.5) node[
  scale=0.75,
  text=white,
  rotate=0.0
]{0.78};
\draw (axis cs:3.5,3.5) node[
  scale=0.75,
  text=white,
  rotate=0.0
]{0};
\draw (axis cs:4.5,3.5) node[
  scale=0.75,
  text=white,
  rotate=0.0
]{0.76};
\draw (axis cs:0.5,4.5) node[
  scale=0.75,
  text=white,
  rotate=0.0
]{0.87};
\draw (axis cs:1.5,4.5) node[
  scale=0.75,
  text=white,
  rotate=0.0
]{0.81};
\draw (axis cs:2.5,4.5) node[
  scale=0.75,
  text=white,
  rotate=0.0
]{0.83};
\draw (axis cs:3.5,4.5) node[
  scale=0.75,
  text=white,
  rotate=0.0
]{0.76};
\draw (axis cs:4.5,4.5) node[
  scale=0.75,
  text=white,
  rotate=0.0
]{0};
\end{axis}

\end{tikzpicture}%
            }
            \caption{$N=10$, $\tau = 5$}
            \label{fig:hell10_1}
        \end{subfigure}
    \end{tabular}
	\caption{Pairwise Hellinger distances between the Haar wavelet estimators $\{\hat{\mu}^r_{\mathrm{H}}\}_{r \in \mathcal{R}}$ for \protect\subref{fig:hell_0} $N=100$ and \protect\subref{fig:hell10_0} $N=10$ as well as between the thresholding Haar wavelet estimators $\{\tilde{\mu}^r_{\mathrm{H}}\}_{r \in \mathcal{R}}$ for \protect\subref{fig:hell_1} $N=100$ and \protect\subref{fig:hell10_1} $N=10$. Values close to $1$ (blue) indicate dissimilarity, while a value of $0$ (red) represents equality.}
	\label{fig:hellinger}
\end{figure}

\subsection*{Software and Data Availability}

The \texttt{Python} code to implement all numerical experiments presented in this paper is publicly available and located on
the \texttt{Persistence Wavelets} GitHub repository at 

\url{https://github.com/konstantin-haberle/PersistenceWavelets}.

\section{Discussion}
\label{sec:end}
In this work, we studied the distributional behavior of persistence diagrams in terms of the expected persistence diagram.  We used the flexibility and accuracy of wavelets to nonparametrically estimate the density function describing the distribution. 
Specifically, we have shown that the Haar wavelet estimator $\hat{\mu}_{\mathrm{H}}$ is a minimax estimator for the expected persistence diagram. In contrast to the empirical mean, which is a discrete measure whose support is typically very large, $\hat{\mu}_{\mathrm{H}}$ has a Lebesgue density. We also considered a Haar wavelet estimator which employs hard thresholding of the expansion coefficients, showing that it also achieves near-optimal convergence rates. This estimator, by retaining only the coefficients above a given threshold, offers a sparse representation of the expected persistence diagram. We verified the theoretical results by numerical experiments on two prototypical datasets in TDA, the torus and the double torus, as well as on a clustered process. These numerical results support our findings that Haar wavelets can provide an efficient and accurate density estimation of the expected persistence diagram. The theoretical convergence rates are already apparent for modest sample sizes ($N\sim 100$), moreover, for the same sample size regime, a rather high degree of accuracy is achieved ($\mathrm{OT}_{p} < 10^{-5}$ for $p>1$ and $\mathrm{OT}_{p} \sim 10^{-2}$ for $p=1$). 
The computational complexity evaluating the Haar wavelet-based density estimator at a point in $\Omega$ scales as $\mathcal{O}(N \log_2(N))$. 
We characterized the effect of sparsity induced by hard thresholding, both at the level of the estimated density and its accuracy, showing that the $\mathrm{OT}_{p}$ magnitude stays comparable to that of Haar for a large set of thresholds and that there is a reduction in the occurrence of finely detailed features in the estimated density.
Finally, we demonstrated the practical utility of the Haar wavelet density estimator by applying it to a classification task for dynamical trajectories. Our results illustrated that the proposed Haar wavelet estimation techniques can effectively distinguish between classes of trajectories.

The application of tools from uncertainty quantification in TDA has a number of advantages. Density function estimation enables extrapolation to regions at low sampling as well as the reconstruction of structural features of the distribution of random persistence diagrams, such as gaps, which can lead to more robust insights into the underlying data topology. Due to their localized nature, wavelet transforms are particularly well-suited to capture the local features of the approximated signal. Hard thresholding of wavelet coefficients provides meaningful and sparse representation of the inputs that has been exploited for various applications, such as image compression and denoising \citep{krommweh2010,fryzlewicz2016}.
This sparsity property may be exploited for various data reduction tasks when persistence diagrams consist of numerous points, and to facilitate feature extraction, which can be then leveraged for additional inference tasks, such as prediction and classification.

In general, wavelet transforms provide a multiresolution representation that is especially useful for signal decomposition and processing \citep{mallat1989}. Thanks to their multiresolution representational properties, cascades of filters implementing wavelet transforms have been incorporated in convolutional neural networks, showing advantages in several settings \citep{bruna2013,pedersen2022}, and have been proposed as a mathematical framework to analyze deep convolutional architectures \citep{mallat2016}.
Multiresolution representations of the distribution of persistence diagrams are a promising perspective for further studies for their potential to gain a better understanding of the topology of complex datasets by enabling multiscale separation and the discovery of scale invariant features.  These advantages may be incorporated in the advancement and development of machine learning algorithms for persistence diagrams.

Finally, the Haar basis is the simplest wavelet basis, but other bases, such as Symmlet wavelets, are often more suitable to approximate smooth functions. In particular, choosing an appropriate wavelet basis can lead to a sparser representation thanks to the regularity of the function to be estimated. Also the investigation on convergence could be extended to studying different wavelet transforms, such as Symmlet and Daubechies wavelets and Coiflets \citep{hardle2012wavelets,daubechies1992ten}. 
 These studies, however, would require alternative proof techniques since the construction employed here does not immediately extend to other types of wavelets. The majority of wavelet-based density estimation approaches rely on the empirical estimate of the wavelet series coefficients, as we did in this work, but it is also possible to take into account cases where the basis coefficients are learned via parametric estimators such as maximum likelihood (see, e.g., \citep{peter2008}).


\section*{Acknowledgments}

We wish to thank Mauricio Barahona, Yueqi Cao, Adi Ditkowski, Th\'{e}o Lacombe, and Primo\v{z} \v{S}kraba for helpful discussions.  We also would like to thank two anonymous referees for their helpful comments and suggestions which improved our work.

We wish to acknowledge the Information and Communication Technologies resources at Imperial College London for their computing resources which were used to implement the experiments in this paper.


\bibliographystyle{authordate3}
\bibliography{references.bib}

\begin{thebibliography}{}

\bibitem[\protect\citename{Abramovich {\em et~al.}, }2000]{abramovich2000}
{\sc Abramovich, Felix, Bailey, Trevor~C, \& Sapatinas, Theofanis}. 2000.
\newblock Wavelet analysis and its statistical applications.
\newblock {\em Journal of the Royal Statistical Society: Series D (The Statistician)}, {\bf 49}(1), 1--29.

\bibitem[\protect\citename{Adams {\em et~al.}, }2017]{adams2017persistence_vectorisation}
{\sc Adams, Henry, Emerson, Tegan, Kirby, Michael, Neville, Rachel, Peterson, Chris, Shipman, Patrick, Chepushtanova, Sofya, Hanson, Eric, Motta, Francis, \& Ziegelmeier, Lori}. 2017.
\newblock Persistence images: A stable vector representation of persistent homology.
\newblock {\em Journal of Machine Learning Research}, {\bf 18}.

\bibitem[\protect\citename{Bruna \& Mallat, }2013]{bruna2013}
{\sc Bruna, Joan, \& Mallat, St{\'e}phane}. 2013.
\newblock Invariant scattering convolution networks.
\newblock {\em IEEE Transactions on Pattern Analysis and Machine Intelligence}, {\bf 35}(8), 1872--1886.

\bibitem[\protect\citename{Bubenik, }2015]{bubenik2015statistical}
{\sc Bubenik, Peter}. 2015.
\newblock Statistical topological data analysis using persistence landscapes.
\newblock {\em J. Mach. Learn. Res.}, {\bf 16}(1), 77--102.

\bibitem[\protect\citename{Bubenik \& Wagner, }2020]{bubenik2020embeddings}
{\sc Bubenik, Peter, \& Wagner, Alexander}. 2020.
\newblock Embeddings of persistence diagrams into Hilbert spaces.
\newblock {\em Journal of Applied and Computational Topology}, {\bf 4}(3), 339--351.

\bibitem[\protect\citename{Chazal \& Divol, }2019]{Chazal_Divol_EPD}
{\sc Chazal, Fr\'{e}d\'{e}ric, \& Divol, Vincent}. 2019.
\newblock The density of expected persistence diagrams and its kernel based estimation.
\newblock {\em Journal of Computational Geometry}, {\bf 10}(2), 127--153.

\bibitem[\protect\citename{Cohen, }2003]{cohen2003numerical}
{\sc Cohen, Albert}. 2003.
\newblock {\em Numerical analysis of wavelet methods}.
\newblock Elsevier.

\bibitem[\protect\citename{Conti {\em et~al.}, }2022]{conti2022topological}
{\sc Conti, Francesco, Moroni, Davide, \& Pascali, Maria~Antonietta}. 2022.
\newblock A topological machine learning pipeline for classification.
\newblock {\em Mathematics}, {\bf 10}(17), 3086.

\bibitem[\protect\citename{Crawford {\em et~al.}, }2020]{TDA_progression_disease}
{\sc Crawford, Lorin, Monod, Anthea, Chen, Andrew~X., Mukherjee, Sayan, \& Rabad{\'a}n, Ra{\'u}l}. 2020.
\newblock Predicting Clinical Outcomes in Glioblastoma: An Application of Topological and Functional Data Analysis.
\newblock {\em Journal of the American Statistical Association}, {\bf 115}(531), 1139--1150.

\bibitem[\protect\citename{Daubechies, }1992]{daubechies1992ten}
{\sc Daubechies, Ingrid}. 1992.
\newblock {\em Ten lectures on wavelets}.
\newblock SIAM.

\bibitem[\protect\citename{Divol \& Lacombe, }2021a]{divol2021estimation}
{\sc Divol, Vincent, \& Lacombe, Th{\'e}o}. 2021a.
\newblock Estimation and quantization of expected persistence diagrams.
\newblock {\em Pages  2760--2770 of:} {\em International Conference on Machine Learning}.
\newblock PMLR.

\bibitem[\protect\citename{Divol \& Lacombe, }2021b]{divol_und_top}
{\sc Divol, Vincent, \& Lacombe, Th{\'e}o}. 2021b.
\newblock Understanding the topology and the geometry of the space of persistence diagrams via optimal partial transport.
\newblock {\em Journal of Applied and Computational Topology}, {\bf 5}(1), 1--53.

\bibitem[\protect\citename{Donoho {\em et~al.}, }1996]{Donoho_wavelet_est}
{\sc Donoho, David~L., Johnstone, Iain~M., Kerkyacharian, G\'{e}rard, \& Picard, Dominique}. 1996.
\newblock Density estimation by wavelet thresholding.
\newblock {\em The Annals of Statistics}, {\bf 24}(2), 508--539.

\bibitem[\protect\citename{Flamary {\em et~al.}, }2021]{flamary2021pot}
{\sc Flamary, R{\'e}mi, Courty, Nicolas, Gramfort, Alexandre, Alaya, Mokhtar~Z., Boisbunon, Aur{\'e}lie, Chambon, Stanislas, Chapel, Laetitia, Corenflos, Adrien, Fatras, Kilian, Fournier, Nemo, Gautheron, L{\'e}o, Gayraud, Nathalie~T.H., Janati, Hicham, Rakotomamonjy, Alain, Redko, Ievgen, Rolet, Antoine, Schutz, Antony, Seguy, Vivien, Sutherland, Danica~J., Tavenard, Romain, Tong, Alexander, \& Vayer, Titouan}. 2021.
\newblock POT: Python Optimal Transport.
\newblock {\em Journal of Machine Learning Research}, {\bf 22}(78), 1--8.

\bibitem[\protect\citename{Fournier \& Guillin, }2015]{Wasserstein_partitioning_arg}
{\sc Fournier, Nicolas, \& Guillin, Arnaud}. 2015.
\newblock On the rate of convergence in {W}asserstein distance of the empirical measure.
\newblock {\em Probability Theory and Related Fields}, {\bf 162}(3-4), 707--738.

\bibitem[\protect\citename{Fryzlewicz \& Timmermans, }2016]{fryzlewicz2016}
{\sc Fryzlewicz, Piotr, \& Timmermans, Catherine}. 2016.
\newblock SHAH: SHape-Adaptive Haar wavelets for image processing.
\newblock {\em Journal of Computational and Graphical Statistics}, {\bf 25}(3), 879--898.

\bibitem[\protect\citename{H{\"a}rdle {\em et~al.}, }2012]{hardle2012wavelets}
{\sc H{\"a}rdle, Wolfgang, Kerkyacharian, Gerard, Picard, Dominique, \& Tsybakov, Alexander}. 2012.
\newblock {\em Wavelets, approximation, and statistical applications}.
\newblock  Vol. 129.
\newblock Springer Science \& Business Media.

\bibitem[\protect\citename{Kang {\em et~al.}, }2013]{hudson2013}
{\sc Kang, In, Hudson, Irene, Rudge, Andrew, \& Chase, J~Geoffrey}. 2013.
\newblock Density estimation and wavelet thresholding via Bayesian methods: A wavelet probability band and related metrics approach to assess agitation and sedation in ICU patients.
\newblock {\em Discrete Wavelet Transforms: A Compendium of New Approaches and Recent Applications. 1st ed. Rijeka: IntechOpen},  127--162.

\bibitem[\protect\citename{Krommweh, }2010]{krommweh2010}
{\sc Krommweh, Jens}. 2010.
\newblock Tetrolet transform: A new adaptive Haar wavelet algorithm for sparse image representation.
\newblock {\em Journal of Visual Communication and Image Representation}, {\bf 21}(4), 364--374.

\bibitem[\protect\citename{Li {\em et~al.}, }2014]{TDA_CompVision}
{\sc Li, Chunyuan, Ovsjanikov, Maks, \& Chazal, Fr\'{e}d\'{e}ric}. 2014.
\newblock Persistence-Based Structural Recognition.
\newblock {\em Pages  2003--2010 of:} {\em 2014 IEEE Conference on Computer Vision and Pattern Recognition}.

\bibitem[\protect\citename{Mallat, }2016]{mallat2016}
{\sc Mallat, St{\'e}phane}. 2016.
\newblock Understanding deep convolutional networks.
\newblock {\em Philosophical Transactions of the Royal Society A: Mathematical, Physical and Engineering Sciences}, {\bf 374}(2065), 20150203.

\bibitem[\protect\citename{Mallat, }1989]{mallat1989}
{\sc Mallat, Stephane~G}. 1989.
\newblock A theory for multiresolution signal decomposition: the wavelet representation.
\newblock {\em IEEE Transactions on Pattern Analysis and Machine Intelligence}, {\bf 11}(7), 674--693.

\bibitem[\protect\citename{Mileyko {\em et~al.}, }2011]{mileyko2011probability}
{\sc Mileyko, Yuriy, Mukherjee, Sayan, \& Harer, John}. 2011.
\newblock Probability measures on the space of persistence diagrams.
\newblock {\em Inverse Problems}, {\bf 27}(12), 124007.

\bibitem[\protect\citename{Monod {\em et~al.}, }2019]{TDA_viral_evolution}
{\sc Monod, Anthea, Kali\v{s}nik, Sara, Pati\~{n}o Galindo, Juan~\'{A}ngel, \& Crawford, Lorin}. 2019.
\newblock Tropical Sufficient Statistics for Persistent Homology.
\newblock {\em SIAM Journal on Applied Algebra and Geometry}, {\bf 3}(2), 337--371.

\bibitem[\protect\citename{Pedersen {\em et~al.}, }2022]{pedersen2022}
{\sc Pedersen, Christian, Eickenberg, Michael, \& Ho, Shirley}. 2022.
\newblock Learnable wavelet neural networks for cosmological inference.
\newblock {\em In:} {\em ICML 2022 Workshop on Machine Learning for Astrophysics}.

\bibitem[\protect\citename{Perea \& Harer, }2015]{TDA_signal_analysis}
{\sc Perea, Jose~A, \& Harer, John}. 2015.
\newblock Sliding windows and persistence: An application of topological methods to signal analysis.
\newblock {\em Foundations of Computational Mathematics}, {\bf 15}(3), 799--838.

\bibitem[\protect\citename{Peter \& Rangarajan, }2008]{peter2008}
{\sc Peter, Adrian~M, \& Rangarajan, Anand}. 2008.
\newblock Maximum likelihood wavelet density estimation with applications to image and shape matching.
\newblock {\em IEEE Transactions on Image Processing}, {\bf 17}(4), 458--468.

\bibitem[\protect\citename{Reininghaus {\em et~al.}, }2015]{Kernel_PD_reininghaus2015stable}
{\sc Reininghaus, Jan, Huber, Stefan, Bauer, Ulrich, \& Kwitt, Roland}. 2015.
\newblock A stable multi-scale kernel for topological machine learning.
\newblock {\em Pages  4741--4748 of:} {\em Proceedings of the IEEE conference on computer vision and pattern recognition}.

\bibitem[\protect\citename{Tauzin {\em et~al.}, }2020]{tauzin2020giottotda}
{\sc Tauzin, Guillaume, Lupo, Umberto, Tunstall, Lewis, P{\'e}rez, Julian~Burella, Caorsi, Matteo, Medina-Mardones, Anibal, Dassatti, Alberto, \& Hess, Kathryn}. 2020.
\newblock {\em giotto-tda: A Topological Data Analysis Toolkit for Machine Learning and Data Exploration}.

\bibitem[\protect\citename{{The GUDHI Project}, }2015]{gudhi:urm}
{\sc {The GUDHI Project}}. 2015.
\newblock {\em {GUDHI} User and Reference Manual}.
\newblock {GUDHI Editorial Board}.

\bibitem[\protect\citename{Turner {\em et~al.}, }2014]{turner-2014-frechet}
{\sc Turner, Katharine, Mileyko, Yuriy, Mukherjee, Sayan, \& Harer, John}. 2014.
\newblock Fr{\'e}chet means for distributions of persistence diagrams.
\newblock {\em Discrete \& Computational Geometry}, {\bf 52}(1), 44--70.

\bibitem[\protect\citename{Weed \& Bach, }2019]{Wasserstein_partition_arg2}
{\sc Weed, Jonathan, \& Bach, Francis}. 2019.
\newblock Sharp asymptotic and finite-sample rates of convergence of empirical measures in {W}asserstein distance.
\newblock {\em Bernoulli}, {\bf 25}(4A), 2620--2648.

\bibitem[\protect\citename{Zhao \& Zhang, }2005]{zhao2005}
{\sc Zhao, Qibin, \& Zhang, Liqing}. 2005.
\newblock ECG feature extraction and classification using wavelet transform and support vector machines.
\newblock {\em Pages  1089--1092 of:} {\em 2005 International Conference on Neural Networks and Brain},  vol. 2.
\newblock IEEE.

\bibitem[\protect\citename{Zomorodian \& Carlsson, }2005]{zomorodian_computing_2005}
{\sc Zomorodian, Afra, \& Carlsson, Gunnar}. 2005.
\newblock Computing {Persistent} {Homology}.
\newblock {\em Discrete \& Computational Geometry}, {\bf 33}(2), 249--274.

\end{thebibliography}

\end{document}